\documentclass[12pt,reqno]{amsart}
\usepackage{amsbooka}
\makeatletter
\def\chaptermark#1{}%whatever

\def\chapter{%
  \if@openright\cleardoublepage\else\clearpage\fi
  \thispagestyle{plain}\global\@topnum\z@
  \@afterindenttrue \secdef\@chapter\@schapter}

\def\@chapter[#1]#2{\refstepcounter{chapter}%
  \ifnum\c@secnumdepth<\z@ \let\@secnumber\@empty
  \else \let\@secnumber\thechapter \fi
  \typeout{\chaptername\space\@secnumber}%
  \def\@toclevel{0}%
  \ifx\chaptername\appendixname \@tocwriteb\tocappendix{chapter}{#2}%
  \else \@tocwriteb\tocchapter{chapter}{#2}\fi
  \chaptermark{#1}%
  \addtocontents{lof}{\protect\addvspace{10\p@}}%
  \addtocontents{lot}{\protect\addvspace{10\p@}}%
  \@makechapterhead{#2}\@afterheading}
\def\@schapter#1{\typeout{#1}%
  \let\@secnumber\@empty
  \def\@toclevel{0}%
  \ifx\chaptername\appendixname \@tocwriteb\tocappendix{chapter}{#1}%
  \else \@tocwriteb\tocchapter{chapter}{#1}\fi
  \chaptermark{#1}%
  \addtocontents{lof}{\protect\addvspace{10\p@}}%
  \addtocontents{lot}{\protect\addvspace{10\p@}}%
  \@makeschapterhead{#1}\@afterheading}
\newcommand\chaptername{Chapter}

\def\@makechapterhead#1{\global\topskip 7.5pc\relax
  \begingroup
  \fontsize{\@xivpt}{18}\bfseries\centering
    \ifnum\c@secnumdepth>\m@ne
      \leavevmode \hskip-\leftskip
      \rlap{\vbox to\z@{\vss
          \centerline{\normalsize\mdseries
              \uppercase\@xp{\chaptername}\enspace\thechapter}
          \vskip 3pc}}\hskip\leftskip\fi
     #1\par \endgroup
  \skip@34\p@ \advance\skip@-\normalbaselineskip
  \vskip\skip@ }
\def\@makeschapterhead#1{\global\topskip 7.5pc\relax
  \begingroup
  \fontsize{\@xivpt}{18}\bfseries\centering
  #1\par \endgroup
  \skip@34\p@ \advance\skip@-\normalbaselineskip
  \vskip\skip@ }
\def\appendix{\par
  \c@chapter\z@ \c@section\z@
  \let\chaptername\appendixname
  \def\thechapter{\@Alph\c@chapter}}

\newcounter{chapter}
\newif\if@openright

\makeatother

\makeatletter
\@addtoreset{section}{chapter}
\makeatother

\usepackage{amsmath, amssymb, amsthm, mathrsfs, amscd}
\usepackage[table]{xcolor}
\usepackage{stmaryrd} %\llbracket, \rrbracket
\usepackage{color} %, pxfonts}
\usepackage[all]{xy}
\usepackage{yhmath}
\usepackage{soul}
\usepackage{bbm} % \mathbbm
\usepackage{hyperref}
\usepackage{leftidx} %\ltrans, \leftidx
\usepackage{mathabx} %\widecheck
\usepackage{mathtools} %\prescript
\usepackage{rotating} %\rotatebox
\usepackage{stmaryrd} %\llbracket, \rrbracket

\usepackage{youngtab,young}
\newcommand{\bY}[1]{    \begin{Young}  #1  \end{Young}}

\usepackage{todonotes}
\usepackage{arydshln}
\makeatletter
\renewcommand*\env@matrix[1][*\c@MaxMatrixCols c]{%
  \hskip -\arraycolsep
  \let\@ifnextchar\new@ifnextchar
  \array{#1}}
\makeatother
\usepackage{pgf,tikz}
\usetikzlibrary{arrows, positioning, calc, chains}
\tikzset{
	ch/.style={circle,draw,on chain,inner sep=2pt},
	chj/.style={ch,join},
	every path/.style={shorten >=4pt,shorten <=4pt}
	}
\newcommand{\dnode}[2][chj]{%
	\node[#1,label={below:#2}] (#1) {};}
\newcommand{\dnodenj}[1]{%
	\dnode[ch]{#1}}
\newcommand{\dydots}{%
	\node[chj,draw=none,inner sep=1pt] {\dots};}
\numberwithin{equation}{section}
\newcommand{\Bjj}{\dot{\mathfrak B}^{\fc}}
\newcommand{\KcL}{\bK^{\fc}_{\breve n,  1, 0}}
\newcommand{\Kcap}{\bK^{\fc,ap}_{\breve n}}
\newcommand{\Kap}{\bK^{ap}_{\breve n}}

\newcommand{\KL}{\bK_{\breve n,  1, 0}}

\newcommand{\dKcL}{\dot{\bK}^{\fc}_{\breve n,  1, 0}}
\newcommand{\dKcb}{\dot{\bK}^{\fc}_{\breve n}}
\newcommand{\dKL}{\dot{\bK}_{\breve n,  1, 0}}

\newcommand{\XitL}{\Xit_{\breve n,  1, 0}}
\newcommand{\TtL}{\~{\Tt}_{\breve n,  1, 0}}
\newcommand{\hKcL}{\^{\bK}^{\fc}_{\breve n,  1, 0}}

\newcommand{\hKL}{\^{\bK}_{\breve n,  1, 0}}
\newcommand{\hKb}{\^{\bK}_{\breve n}}
\newcommand{\trho}{\~{\rho}}
\newcommand{\hrho}{\^{\rho}}
\newcommand{\hKjj}{\^{\bK}^{\fc}_n}
\newcommand{\Xiz}{\Xi^0_n}
\newcommand{\Xizb}{\Xi^0_{\breve n}}
\newcommand{\Dcbab}{\Delta^{\fc}_{\breve n,\mathbf a', \mathbf b', \mathbf a'', \mathbf b''}}

\newcommand{\hKcn}{\^{\bK}^{\fc}_n}
\newcommand{\dKcn}{\dot{\K}^{\fc}_n}
\newcommand{\Kcn}{\bK^{\fc}_n}

\newcommand{\adKcb}{\leftidx{_\mathbf{a}}{\dot{\bK}^\fc_\mathbf{b}}}
\newcommand{\adKcbp}{\leftidx{_{\mathbf{a}'}}{\dot{\bK}^\fc_{\mathbf{b}'}}}

\newcommand{\hKn}{\^{\bK}_n}
\newcommand{\dKn}{\dot{\K}_n}
\newcommand{\Kn}{\bK_n}

\newcommand{\adKbpp}{\leftidx{_{\mathbf{a}''}}{\dot{\bK}_{\mathbf{b}''}}}

\newcommand{\egeo}{e^{\textup{geo}}}

\newcommand{\Sjjg}{\bS^{\fc, \textup{geo}}_{n,d}}

\newcommand{\Xjj}{\mathcal{X}^{\fc}_{n,d}}
\newcommand{\Yjj}{\mathcal{Y}^{\fc}_{d}}
\newcommand{\Sjj}{\bS^{\fc}_{n,d}}
\newcommand{\Kjj}{\bK^{\fc}_n}

\newcommand{\dKjj}{\dot{\bK}^{\fc}_n}
\newcommand{\dKjjp}{\dot{\bK}^{\fc,>}_n}

\newcommand{\Sji}{\bS^{\jmath\imath}_{\nn,d}}
\newcommand{\Kji}{\bK^{\jmath\imath}_{\nn}}

\newcommand{\Sij}{\bS^{\imath\jmath}_{\nn,d}}
\newcommand{\Kij}{\bK^{\imath\jmath}_{\nn}}

\newcommand{\Sii}{\bS^{\imath\imath}_{\eta,d}}
\newcommand{\Kii}{\bK^{\imath\imath}_\eta}

\newcommand{\af}{\alpha}

\newcommand{\alg}{\textup{alg}}

\newcommand{\bA}[1]{    \begin{aligned}  #1  \end{aligned}}
\newcommand{\ba}[1]{    \begin{array}   #1   \end{array}}
\newcommand{\bc}[1]{     \begin{cases}   #1   \end{cases}}

\newcommand{\bH}{\mbf H}

\newcommand{\bj}{\mbf j}
\newcommand{\bJ}{\mbf J}
\newcommand{\bK}{\mbf K}
\newcommand{\bk}{\mbf k}
\newcommand{\bM}[1]{    \begin{bmatrix}    #1    \end{bmatrix}}
\newcommand{\bp}[1]{\big{(} #1\big{)}}
\newcommand{\Bp}[1]{\Big{(} #1\Big{)}}
\newcommand{\bS}{\mbf S}

\newcommand{\bU}{\mathbf{U}}

\newcommand{\cA}{\mathcal{A}}

\newcommand{\bbA}{\mathbb{A}}

\newcommand{\cI}{\mathcal{I}}

\newcommand{\cO}{\mathcal{O}}
\newcommand{\coA}{\textup{col}_{\mathfrak{a}}}
\newcommand{\coC}{\textup{col}_{\mathfrak{c}}}

\newcommand{\cP}{\mathcal{P}}
\newcommand{\cX}{\mathcal{X}}

\newcommand{\D}{\mathscr{D}} %distinguished coset representative
\newcommand{\diag}{\textup{diag}}
\newcommand{\ds}{\displaystyle}

\newcommand{\ep}{\epsilon}

\newcommand{\Ett}{E_\tt}

\newcommand{\f}{\mathbf f}

\newcommand{\fa}{\mathfrak{a}}

\newcommand{\fc}{\mathfrak{c}}
\newcommand{\FF}{\mathbb{F}}

\newcommand{\gc}{\cellcolor{gray!40}}

\newcommand{\g}{\mathfrak g}
\newcommand{\geo}{\textup{geo}}
\newcommand{\gl}{\mathfrak{gl}}
\newcommand{\GL}{\mrm{GL}}
\newcommand{\gstd}{g^{\textup{std}}}
\newcommand{\Ia}{I^+_\fa}

\newcommand{\Hom}{\textup{Hom}}

\newcommand{\inv}{^{-1}}

\newcommand{\id}{\mathbbm{1}}

\newcommand{\Xitij}{\~{\Xi}^{\imath\jmath}_{\nn}}
\newcommand{\Xitii}{\~{\Xi}^{\imath\imath}_{\eta}}
\newcommand{\Xitji}{\~{\Xi}^{\jmath\imath}_{\nn}}

\newcommand{\Jql}{\mathbf{J}}
\newcommand{\K}{\mbf K}
\newcommand{\Ker}{\textup{Ker}~}

\newcommand{\dKij}{\dot{\bK}_{\nn}^{\imath\jmath}} %
\newcommand{\dKii}{\dot{\bK}_\eta^{\imath\imath}} %
\newcommand{\dKji}{\dot{\bK}_{\nn}^{\jmath\imath}} %

\newcommand{\ld}{\lambda}
\newcommand{\Ld}{\Lambda}

\newcommand{\lr}[1]{\left[ #1 \right]} %usual quantum number
\newcommand{\lrb}[2]{\left[\begin{array}{c} #1\\#2\end{array} \right]} %usual q binomial
\newcommand{\LR}[1]{\left\llbracket #1 \right\rrbracket} % bar-invariant
\newcommand{\mA}{\mathcal{A}}
\newcommand{\mbb}{\mathbb}
\newcommand{\mbf}{\mathbf}

\newcommand{\mrm}{\mathrm}

\newcommand{\nn}{\mathfrak{n}}
\newcommand{\NN}{\mathbb{N}}

\newcommand{\offd}{\star}
\newcommand{\otw}{\textup{otherwise}}

\newcommand{\Perm}{\mathrm{Perm}}
\newcommand{\SO}{\mrm{SO}}

\newcommand{\p}[1]{\leftidx{_{p}}{#1}}

\newcommand{\Ldij}{\Ld^{\imath\jmath}}
\newcommand{\Ldii}{\Ld^{\imath\imath}}
\newcommand{\Ldji}{\Ld^{\jmath\imath}}
\newcommand{\pp}[1]{\leftidx{_{\breve{p}}}{#1}}
\newcommand{\Tp}{\texttt{T}}
\newcommand{\Xij}{\Xi^{\imath\jmath}_{\nn,d}}
\newcommand{\Xijp}{\V{\Xi}^{\imath\jmath}_{\nn,d}}
\newcommand{\Xii}{\Xi^{\imath\imath}_{\eta,d}}

\newcommand{\Xji}{\Xi^{\jmath\imath}_{\nn,d}}

\newcommand{\QQ}{\mathbb{Q}}
\newcommand{\qbinom}[2]{\begin{bmatrix} #1\\#2 \end{bmatrix} }

\newcommand{\roA}{\textup{row}_{\fa}}
\newcommand{\roC}{\textup{row}_{\fc}}

\newcommand{\RR}{\mathbb{R}}
\newcommand{\rw}{\rightarrow}

\newcommand{\sig}{\sigma}

\newcommand{\SP}{\mrm{Sp}}
\newcommand{\Span}{\textup{Span}}
\newcommand{\Stab}{\textup{Stab}}
\newcommand{\stdP}{^{\cP}}

\newcommand{\tA}{\mathrm{A}}
\newcommand{\taffA}{\~{\mathrm{A}}}
\newcommand{\taffB}{\~{\mathrm{B}}}
\newcommand{\taffC}{\~{\mathrm{C}}}
\newcommand{\taffD}{\~{\mathrm{D}}}
\newcommand{\tB}{\mathrm{B}}
\newcommand{\tC}{\mathrm{C}}
\newcommand{\tD}{\mathrm{D}}

\newcommand{\tinv}{\mathrm{inv}}
\newcommand{\tMat}{\mathrm{Mat}}
\newcommand{\tneg}{\mathrm{neg}}
\newcommand{\tnsp}{\mathrm{nsp}}
\newcommand{\tor}{\textup{or }}
\renewcommand{\tt}{\theta}
\newcommand{\tX}{\mathrm{X}}
\newcommand{\ur}{\mathrm{ur}}

\newcommand{\tif}{\textup{if }}
\newcommand{\tfor}{\textup{for }}
\newcommand{\tand}{\textup{and }}

\newcommand{\Tt}{\Theta}
\newcommand{\V}[1]{\widecheck{#1}}
\newcommand{\ve}{\varepsilon}

\newcommand{\W}{W}
\newcommand{\wo}{w_\circ}

\newcommand{\Xit}{\~{\Xi}}
\newcommand{\Y}{\mathcal Y}

\newcommand{\ZZ}{\mathbb{Z}}
\newcommand{\ZZij}{\mathbb{Z}^{\imath\jmath}}

\renewcommand{\^}[1]{\widehat{#1}}
\renewcommand{\~}[1]{\widetilde{#1}}
\renewcommand{\=}[1]{\overline{#1}}
\renewcommand{\tt}{\theta}

\newenvironment{enu}{\begin{enumerate}}{\end{enumerate}}
\usepackage{enumitem }
\newenvironment{enua}{\begin{enumerate}[label=$(\alph*)$]
}{\end{enumerate}}
\newenvironment{eq}{\begin{equation}}{\end{equation}}
\newenvironment{eqs}{\begin{equation*}}{\end{equation*}}
\newenvironment{itm}{\begin{itemize}}{\end{itemize}}

\theoremstyle{definition}
\newtheorem{Def}{Definition}[section] %[subsection]

\newtheorem{exa}[Def]{Example}
\newtheorem{rem}[Def]{Remark}
\newtheorem{rmk}[Def]{Remark}
\newtheorem{Alg}[Def]{Algorithm}

\theoremstyle{plain}
\newtheorem{prop}[Def]{Proposition}
\newtheorem{thm}[Def]{Theorem}
\newtheorem{lem}[Def]{Lemma}

\newtheorem{cor}[Def]{Corollary}

\newtheorem{lemma}[Def]{Lemma}

\newtheorem*{thmA}{Theorem A}
\newtheorem*{thmB}{Theorem B}
\newtheorem*{thmC}{Theorem C}
\newtheorem*{thmD}{Theorem D}
\newtheorem*{thmE}{Theorem E}
\newtheorem*{thmF}{Theorem F}

\makeatletter
\newenvironment{mysmallmatrix}{%
    \let\saved@CT@row@color\CT@row@color
    \begin{smallmatrix}%
}{%
    \end{smallmatrix}%
    \global\let\CT@row@color\saved@CT@row@color
}
\makeatother
\title[Affine Hecke algebras and quantum symmetric pairs]{Affine Hecke algebras and quantum symmetric pairs}
\author[Z. Fan, C. Lai, Y. Li, L. Luo and W. Wang]{Zhaobing Fan,  Chun-Ju Lai, Yiqiang Li,  Li Luo and Weiqiang Wang}
\address{School of science, Harbin Engineering University, Harbin 150001, China}
    \email{fanz@math.ksu.edu  (Fan)}
\address{Department of Mathematics, University of Virginia, Charlottesville, VA 22904}
\curraddr{Department of Mathematics, University of Georgia, Athens, GA 30602}
    \email{cjlai@uga.edu (Lai)}
\address{Department of Mathematics, SUNY at Buffalo,  Buffalo, NY 14260}
    \email{yiqiang@buffalo.edu (Li)}
\address{ %${}^{\dagger}$
    Department of mathematics, Shanghai Key Laboratory of Pure Mathematics and Mathematical Practice, East China Normal University, Shanghai 200241, China}
\email{lluo@math.ecnu.edu.cn (Luo)}
\address{Department of Mathematics, University of Virginia, Charlottesville, VA 22904}
    \email{ww9c@virginia.edu (Wang) }

\keywords{}
\subjclass{}

\begin{document}

\begin{abstract}
We introduce an affine Schur algebra via the affine Hecke algebra associated to Weyl group of affine type C.
We establish multiplication formulas on the affine Hecke algebra and affine Schur algebra.
Then we construct monomial bases and canonical bases for the affine Schur algebra.
The multiplication formula allows us to establish
a stabilization property of the family of affine Schur algebras
 that leads to the modified version of an algebra ${\mathbf K}^{\mathfrak c}_n$.
We show that ${\mathbf K}^{\mathfrak c}_n$  is a coideal subalgebra of quantum affine algebra
${\bf U}(\widehat{\mathfrak{gl}}_n)$,
and $\big({\bf U}(\widehat{ \mathfrak{gl}}_n), {\mathbf K}^{\mathfrak c}_n)$ forms a quantum symmetric pair.
The modified coideal subalgebra is shown to  admit monomial and stably canonical bases.
We also formulate several variants of the affine Schur algebra and
the (modified) coideal subalgebra above, as well as their monomial and canonical bases.
This work provides a new and algebraic approach which complements and sheds new light on
our previous geometric approach on the subject.
In the appendix by four of the authors, new length formulas for the Weyl groups of affine classical types are obtained in a symmetrized fashion.

\end{abstract}

\maketitle

\setcounter{tocdepth}{1}
\tableofcontents

\newpage

%=========================================================
\chapter{Introduction}

%====
\section{History}

Dipper and James \cite{DJ89} introduced the finite Schur algebra
as the endomorphism algebra of a sum of permutation modules of the finite type A Hecke algebra.
(All Schur algebras in this paper are understood as quantum Schur algebras.)
Around the same time, the same Schur algebra was constructed geometrically in \cite{BLM90}.
The paper \cite{BLM90} further provides a construction of the
(modified) quantum group of finite type A and its stably canonical basis
by studying the stabilization property of the structures of the family of the Schur algebras.

There have been several works on the generalization to affine type A of the Schur algebras \cite{GV93, Gr99, Lu99, SV00, DDF12}.
However the stabilization phenomenon in the affine type A setting is understood only in recent years \cite{DF14, DF15}.
In the affine type A, the Chevalley generators do not form a generating set for the affine Schur algebras or the corresponding stabilization algebra.
It was shown in \cite{DF14}  (also cf. \cite{FL15}) that there exists a generating set consisting of semisimple generators
which correspond to the bidiagonal matrices; moreover a multiplication formula with semisimple generators was provided. A different approach for some main results in \cite{DF14} was given in \cite{FLLLW}.

The Schur algebra via the finite type B/C Hecke algebra \`a la Dipper-James was studied by R.~Green \cite{Gr97} (also cf. \cite{HL06}).
This construction was generalized and put in a much broader context under the so-called $\imath$-Schur duality \cite{BW13},
where the coideal subalgebra of the quantum group of type A features naturally.
A geometric realization of such a Schur algebra via flag varieties of type B/C is given in \cite{BKLW} (and also \cite{LW15}),
where a BLM-type stabilization leads to a (modified) coideal subalgebra of the quantum group of type A and its canonical basis.
An affinization of this geometric construction in the setting of affine type C flag varieties is developed extensively in \cite{FLLLW}.

%====
\section{The goal}

The goal of this paper is to develop a Hecke algebraic approach toward the affine Schur algebras,
their stabilization algebras, as well as their monomial and canonical bases. This paper can be viewed as a companion paper of our previous work using a geometric approach \cite{FLLLW}.
The current approach provides new algebraic constructions and proofs which shed new light on the underlying algebra structures, and it complements the geometric constructions in \cite{FLLLW}. It offers new multiplication formulas for the affine Hecke and affine Schur algebras.

The paper is naturally organized into 2 parts.
In Part~1, which consists of Chapters~\ref{sec:Schur}--\ref{sec:basis}, we study in depth the structures of affine Schur algebras. We start by presenting the affine Weyl group of type C as a group of permutations of $\mathbb Z$ satisfying suitable conditions.
We introduce the affine Schur algebra $\Sjj$ as the endomorphism algebra of a suitable sum of permutation modules of the corresponding affine Hecke algebra. After a detailed combinatorial study of double cosets of affine Weyl groups, we construct a standard basis, a monomial basis, and a canonical basis for $\Sjj$. We further identify $\Sjj$ and its canonical basis with the geometric counterparts introduced  in \cite{FLLLW}.

At the heart of this paper lie the multiplication formulas of the affine Hecke algebra and then of the affine Schur algebra. As the computational and combinatorial details here are rather challenging and tedious, we have made a serious effort to organize them in a (hopefully) orderly manner, and present the multiplication formulas in compact forms.

Part~2, which consists of Chapters~\ref{sec:coideal}--\ref{sec:coideal2}, studies the stabilization properties of the family of affine Schur algebras, as $d$ varies and goes to infinity. The aforementioned multiplication formula of $\Sjj$ plays a fundamental role in the formulation of the stabilization properties of the family of algebras $\Sjj$ and their bar involutions.
Such stabilization properties allow us to define a stablization (or limiting) algebra $\dKjj$ equipped with a monomial basis, a canonical basis, as well as  a multiplication formula. We further identify the algebra $\dKjj$ here with its counterpart constructed geometrically in \cite{FLLLW}.

There is a standard way to construct an algebra $\K^{\fc}_n$, which is a ``non-idempotented" variant of $\dKjj$. The analogous affine type A constructions give us similar algebras $\dKn$ and $\K_n$, which are identified with the idempotented/usual quantum affine $\gl_n$ in \cite{DF14, DF15}.
By studying the comultiplication formulated in \cite{FLLLW} using the multipication formula for $\dKjj$ in this paper, we show that $\K^{\fc}_n$ is a subalgebra $\K_n$ and $(\K_n, \K^{\fc}_n)$ forms a quantum symmetric pair of affine type in the sense of Kolb \cite{Ko14}.

\vspace{2mm}
Below we present a detailed overview of the approach and main results of this paper.

%====
\section{Main results}
\subsection{Multiplication formulas}

Let $\bH$ be the Hecke algebra associated to the Weyl group $W=W(d)$ of affine flag variety of type C,
with generators $T_w$ for $w\in W$.
We caution that while the algebra $\bH$ is naturally associated to affine flag variety of type C,
it is often viewed in literature from Langlands dual viewpoint,
and referred  as affine Hecke algebra of type B.

Following Lusztig's presentation in affine type A, there have been two presentations of  affine Weyl groups of type C as permutation groups on $\ZZ$;
cf. \cite{B86, Shi94, EE98, BB05}, and their length function formulas are given in \cite{EE98, BB05}; see \eqref{eq:old l(g)}.
In this paper, we choose the presentation of $W$ as a permutation group on $\ZZ$ with two fixed points in each period, which makes the symmetries of $W$ more transparent (cf. \cite{Shi94}).
In particular, this leads to a new and simple length function formula; see Lemma~\ref{lem:l(g)}.

We define the affine Schur algebra $\Sjj$ as the endomorphism algebra of a sum of
permutation modules of  the affine Hecke algebra $\bH$.
Like in \cite{DF14, DF15}, we shall develop a new multiplication formula for the affine Schur algebra $\Sjj$.
However there is a major difference here from affine type A.
The semisimple generators {\em loc. cit.} are bar invariant (and correspond to closed orbits geometrically \cite{FL15}),
and the multiplication formula therein does not require nontrivial multiplication on the Hecke algebra level.
In contrast, we shall see that the generators in our setting correspond to tridiagonal matrices, and
they are not bar invariant in general (neither do they correspond to closed orbits on partial affine flag varieties \cite{FLLLW}).

To develop a multiplication formula for affine Schur algebra,
we are led to first establish a multiplication formula at the affine Hecke algebra level,
which is technically challenging and combinatorially involved.
For notations we refer to the sentence above Theorem~\ref{thm:heckemult}.
\begin{thmA}   [Theorem~\ref{thm:heckemult}]
 Let $B = \kappa(\ld, g_1, \mu)$ be a tridiagonal matrix.
Then, for any $g_2 \in \D_{\mu \nu}$ and $w \in \D_{\delta(B)} \cap W_{\mu}$, we have
\[
T_{g_1}T_{wg_2}=\sum_{\sig\in K_w}
(v^2-1)^{n(\sig)}v^{2h(w,\sig)}T_{g_1\sig wg_2}.
\]
\end{thmA}

By construction, the affine Schur algebra $\Sjj$ admits a basis $\{e_A~|~A\in \Xi_{n,d} \}$
parametrized by the set $\Xi_{n,d}$ \eqref{eq:Xind} of $n$-periodic centro-symmetric $\ZZ\times\ZZ$ matrices over
$\NN$ of size $d$. (We sometimes normalize $e_A$ to become the standard basis $[A]$.)
We have the following multiplication formula for the affine Schur algebra $\Sjj$,
and shall refer to the paragraph above Theorem~\ref{thm:multformula} for notations.

\begin{thmB} [Theorems~\ref{thm:multformula} and ~\ref{thm:multformula2}]
Let $A, B \in \Xi_{n,d}$ with $B$  tridiagonal and $\roC(A)=\coC(B)$.
Then we have
\[ e_B   e_A =
\sum\limits_{\substack{T \in \Tt_{B,A}\\ S \in \Gamma_T}}
(v^2-1)^{n(S)}
v^{2( \ell(A,B,S,T)-n(S)-h(S,T) )}
\LR{A;S;T}~ e_{A^{(T-S)}}.
\]
%Moreover, the set $\{e_A ~|~ A\in \Xi_{n,d}\textup{ is tridiagonal}\}$ is a generating set for the Schur algebra $\Sjj$.
\end{thmB}
Multiplication formulas in different form can also be derived from the geometric approach in \cite{FL16}.
A special case of Theorem~B gives a formula which closely resembles the multiplication formula in affine type A in \cite{DF14},
and another special case closely resembles the multiplication formula in type B/C in \cite{BKLW}.
In these special cases, we only need some of the summands whose powers $n(S)$ of $v^2-1$ are 0, which simplify computations dramatically.
While the multiplication formula in Theorem~B is explicit yet complicated, one can read off useful and essential information about the algebra $\Sjj$.
For example, it allows us to show that the algebra $\Sjj$ has a generating set given by $[A]$ for
tridiagonal matrices $A \in \Xi_{n,d}$; see Theorem~\ref{thm:min gen}(c).

%====
\subsection{Monomial and canonical bases}

By developing further multiplicative properties of $\Sjj$ from the multiplication formula in Theorem~B,
we produce an algorithm (see Algorithm \ref{alg:mono})
to construct a monomial basis $\{m_A~|~A\in \Xi_{n,d} \}$ which is bar invariant such that
$m_A =[A] +$ lower terms with respect to a natural partial ordering.
(A version of this algorithm produces a monomial basis in affine type A \cite{LL17}.)
%As a consequence we establish the following.
The construction of the canonical basis of $\Sjj$ via the Kazhdan-Lusztig basis of $\bH$ is rather routine
(cf. \cite{Du92}).

\begin{thmC}  [Theorem~\ref{thm:CB-Sjj},  Theorem~\ref{thm:mono}]
The Schur algebra $\Sjj$ admits both monomial and canonical bases.
%Moreover, the algebra $\Sjj$ is generated by the $[A]$, for tridiagonal matrices $A \in \Xi_{n,d}$.
\end{thmC}

The multiplication formula in Theorem~B also allows us to establish a stabilization property of
the family of algebras $\Sjj$ as $d \mapsto \infty$.
We remark that the stabilization procedure in the paper relies heavily on the multiplication formula with tridiagonal generators, in contrast to the constructions in \cite{FLLLW} where such a formula was not available. The stabilization property leads to the construction of a stabilization algebra $\dKjj$. While largely following the strategy of \cite{BLM90} (which has been applied also to \cite{BKLW, DF15}), our setting is technically more involved. The bar involution on $\dKjj$ is determined by its nontrivial action on the tridiagonal generators; see the proof of Proposition \ref{prop:stab2}.

A variant of the multiplication formula in Theorem~B is valid for the algebra $\dKjj$; see Theorem~\ref{thm:multK}.
The algorithm for monomial basis for the affine Schur algebra $\Sjj$
is adapted here to construct a monomial basis for the stabilization algebra $\dKjj$.
The stably canonical basis for $\dKjj$ follows from the existence of its monomial basis.

\begin{thmD}[Theorem~\ref{thm:dKjj CB}]
We have an algebra $\dKjj$ arising from stabilization on the family of Schur algebras $\Sjj$ (as $d$ varies).
Moreover, $\dKjj$ admits both monomial and stably canonical bases.
\end{thmD}
We emphasize that the constructions in Chapter~\ref{sec:MF-Hecke} throughout Chapter~\ref{sec:coidealK} are
entirely independent of the geometric approach developed in \cite{FLLLW}.
Our constructions and proofs for Theorems~C and D are built on the multiplication formulas in Theorems~A and B,
while Theorems~A and B are new. We show in Propositions~\ref{Schur-iso} and \ref{Knc} that
the algebras $\Sjj$ and $\dKjj$ here are isomorphic to the geometrically constructed algebras in the same notation  in \cite{FLLLW},
and their monomial and canonical bases also match with their counterparts {\em loc. cit.}.

%====
\subsection{Affine quantum symmetric pairs}

Recall a quantum symmetric pair $(\mathbf U, \mathbf U^\imath)$ in the sense of \cite{Le02, Ko14}
consists of a quantum group $\mathbf U$ and its coideal subalgebra $\mathbf U^\imath$.
The stabilization algebra for Schur algebras of
finite type B/C is the (modified) coideal subalgebra of the quantum group of finite type A \cite{BKLW, LW15}.
There is a standard procedure (which goes back to \cite{BLM90})
to construct an algebra $\Kjj$ for which $\dKjj$ is the modified (or idempotented) version.
Similarly, in the affine type A setting we have a stabilization algebra $\dot \K_n$ which is the modified version of $\K_n$,
and moreover, $\K_n$ is isomorphic to the quantum affine $\gl_n$; cf. \cite{DF15}.

\begin{thmE}[Theorem ~\ref{thm:aqsp}]
The pair $(\K_n, \Kjj)$ forms a quantum symmetric pair associated to an involution on the
Dynkin diagram of affine type A depicted in Figure~\ref{figure:jj}.
\end{thmE}

%===============================================================
\begin{figure}[ht!]
\caption{Dynkin diagram of type $A^{(1)}_{2r+1}$ with involution of type $\jmath\jmath \equiv \fc$.}
 \label{figure:jj}
\[
\begin{tikzpicture}
\matrix [column sep={0.6cm}, row sep={0.5 cm,between origins}, nodes={draw = none,  inner sep = 3pt}]
{
	\node(U1) [draw, circle, fill=white, scale=0.6, label = 0] {};
	&\node(U2)[draw, circle, fill=white, scale=0.6, label =1] {};
	&\node(U3) {$\cdots$};
	&\node(U4)[draw, circle, fill=white, scale=0.6, label =$r-1$] {};
	&\node(U5)[draw, circle, fill=white, scale=0.6, label =$r$] {};
\\
	&&&&&
\\
	\node(L1) [draw, circle, fill=white, scale=0.6, label =below:$2r+1$] {};
	&\node(L2)[draw, circle, fill=white, scale=0.6, label =below:$2r$] {};
	&\node(L3) {$\cdots$};
	&\node(L4)[draw, circle, fill=white, scale=0.6, label =below:$r+2$] {};
	&\node(L5)[draw, circle, fill=white, scale=0.6, label =below:$r+1$] {};
\\
};
\begin{scope}
\draw (L1) -- node  {} (U1);
\draw (U1) -- node  {} (U2);
\draw (U2) -- node  {} (U3);
\draw (U3) -- node  {} (U4);
\draw (U4) -- node  {} (U5);
\draw (U5) -- node  {} (L5);
\draw (L1) -- node  {} (L2);
\draw (L2) -- node  {} (L3);
\draw (L3) -- node  {} (L4);
\draw (L4) -- node  {} (L5);
\draw (L1) edge [color = blue,<->, bend right, shorten >=4pt, shorten <=4pt] node  {} (U1);
\draw (L2) edge [color = blue,<->, bend right, shorten >=4pt, shorten <=4pt] node  {} (U2);
\draw (L4) edge [color = blue,<->, bend left, shorten >=4pt, shorten <=4pt] node  {} (U4);
\draw (L5) edge [color = blue,<->, bend left, shorten >=4pt, shorten <=4pt] node  {} (U5);
\end{scope}
\end{tikzpicture}
\]
\end{figure}

Theorem~E does not appear in the framework of \cite{FLLLW} (only an idempotented quantum symmetric pair was established therein).
The passage from the idempotented quantum symmetric pair to the quantum symmetric pair statement here is nontrivial, and
its proof takes advantage of some constructions {\em loc. cit.} together with the multiplication formula for $\dKjj$ in Theorem~B.

We remark that in a very interesting paper \cite{CGM} Chen, Guay and Ma also studied
interactions between affine Hecke algebra and quantum symmetric pair from a very different perspective from ours.

%====
\subsection{Variants}

There are several variants (called type $\imath\jmath, \jmath\imath,$ and $\imath\imath$) of the algebras $\Sjj$ and $\dKjj$
(here we can view $\fc =\jmath\jmath$).
The Schur algebras $\Sij, \Sji, \Sii$ are subalgebras of $\Sjj$ by construction.
We also construct their respective stabilization algebras $\dKji, \dKij, \dKii$ together with their canonical bases.

\begin{thmF} [Theorem~\ref{thm:monoij}, Theorem~\ref{thm:subqij}, Remark~\ref{rem:QSPij}]
The Schur algebra $\Sij$ admits a canonical basis compatible with the one in $\Sjj$ under the inclusion $\Sij \subset \Sjj$.
The algebra, $\dKij$ is isomorphic to a subquotient of $\dKjj$, with compatible standard, monomial, and stably canonical bases.
\end{thmF}
Similar results for types $\jmath\imath$ and $\imath\imath$ also hold and can be found in
Theorem~\ref{thm:main-ji} and Theorem~\ref{thm:subqii}.
%====
\subsection{New length formulas for Weyl groups}

The Weyl groups of finite and affine types as (infinite) permutations have been studied by Lusztig \cite{Lu83}, B\'edard \cite{B86}, Shi \cite{Shi94}, Bj{\"o}rner-Brenti \cite{BB96, BB05} and Eriksson-Eriksson \cite{EE98}.
The length functions, which are realized by counting (affine) inversions, admit rather involved formulas in terms of sums of floored quotients; see Proposition~\ref{prop:A2}.

In the appendix by four authors, we provide new formulas that describe the dimensions of generalized Schubert varieties of certain finite and affine types in a symmetrized fashion; see Theorem \ref{thm:A1}.
We further obtain new length formulas (see Theorem~\ref{thm:A2}) for all finite and affine Weyl groups in a similar symmetrized fashion.

%The formulas \eqref{eq:lAA} -- \eqref{eq:lCA} we obtain a symmetrized sum in certain matrix entries. When such matrices are permutation matrices, our formulas specialized to the length formulas for the Weyl groups of finite and affine classical type.
%The new length formulas are given explicitly in Theorem~\ref{thm:A2}, and they can be expressed uniformly through the symmetrized inversions we introduced in \eqref{def:sinv};
% in contrast the length formulas for affine Weyl groups are known to involve a rather implicit sum of floored quotients; see Proposition~\ref{prop:A2}.

%====
\section{The organization}

This paper is organized as follows.
In Chapter~\ref{sec:Schur},
the affine Schur algebra $\Sjj$ is defined via the affine Hecke algebra $\bH$, and then identified
with the geometrically-defined affine Schur algebra from ~\cite{FLLLW} with compatible standard bases.

In Chapter~\ref{sec:MF-Hecke},  a multiplication formula for affine Hecke algebra $\bH$ is formulated.

In Chapter~\ref{sec:MF-Schur}, we establish a key multiplication formula for the affine Schur algebra $\Sjj$
with tridiagonal generators.
%We make explicit 2 special cases which are similar to the multiplication formulas in finite type B/C and affine type A.

In Chapter~\ref{sec:basis},  canonical basis for the affine Schur algebra $\Sjj$ is constructed and
identified with the one defined geometrically in \cite{FLLLW}. Using the multiplication formula in
Chapter~\ref{sec:MF-Schur}, we construct a monomial basis for $\Sjj$.

In Chapter~\ref{sec:coideal}, we shall establish a stabilization property for the family of affine Schur algebras $\Sjj$ as $d$ varies,
which leads to a quantum algebra $\dKjj$.
A monomial basis and a stably canonical basis for $\dKjj$ are constructed.

In Chapter~\ref{sec:coidealK}, we construct an algebra $\Kjj$ for which $\dKjj$ is the modified (idempotented) version.
We show that  $(\Kn, \Kjj)$ forms a quantum symmetric pair.

In Chapter~\ref{sec:coideal2},
three more variants of
affine Schur algebras and their corresponding stabilization algebras are introduced. We establish various
results for these new variants analogous to those for the algebras $\Sjj$ and $\dKjj$ obtained in earlier chapters.

In Appendix \ref{chap:lengthformula} we establish a formula for the dimension of generalized Schubert varieties of certain finite and affine types in a symmetrized fashion. We further deduce new length formulas of finite and affine classical Weyl groups as an application.

\vspace{3mm}
\noindent {\bf Notations.}
We shall denote $\NN = \{ 0,1,2,\ldots \}$.
For $a,b\in\ZZ$, we let
$$[a..b] = [a,b] \cap \ZZ, \qquad  [a..b) = [a,b) \cap \ZZ
$$
be the integer intervals. We define similarly integer intervals $(a..b]$ and $(a..b)$.

\vspace{.3cm}
\noindent
{\bf Acknowledgments.}
We thank the following institutions whose support and hospitality help to facilitate the progress and completion of this project:
East China Normal University, Institute of Mathematics at Academia Sinica, and University of Virginia.
Z. Fan is partially supported by the NSF of China grant 11671108, the NSF of Heilongjiang province grand LC2017001 and the Fundamental Research Funds for the central universities GK2110260131.
L. Luo is supported by Science and Technology Commission of Shanghai Municipality grant 18dz2271000 and the NSF of China grant 11871214. 
Y. Li is partially supported by the NSF grant DMS-1801915.
W. Wang is partially supported by the NSF grant DMS-1405131 and DMS-1702254.

\newpage
%
%
%
%
%%%%%%%%%%%%%%%%%%%%%%%%%%%%%%%%%%%%%%%%%%%%%%%%%%%%%%%%%%%%%%%
\part{Affine Schur algebras}  %%%{The quantum groups behind Schur algebras}
  \label{part1}
%%\include{FLpart2.5}

%%%%%%%%%%%%%%%%%%%%%%%%%%%%%%%%%%%%%%%
%=========================================================
\chapter{Affine Schur algebras via affine Hecke algebras}
  \label{sec:Schur}

In this chapter, we define the affine Schur algebra $\Sjj$ as the endomorphism algebra of a sum of
permutation modules of the affine Hecke algebra associated to the affine Weyl group of type $\~{\mathrm C}_d$,
and show that its basis is parametrized by periodic and centro-symmetric integer matrices.
We then reformulate some combinatorics of the affine Weyl group and associated Hecke algebra in terms of such integer matrices.
The algebra $\Sjj$ is identified with the geometrically-defined affine Schur algebra from ~\cite{FLLLW}.

%=================
\section{Affine Weyl groups}

Let  $r ,d \in \NN$ be such that $d \ge 2$, and set
\begin{equation}
\label{nD}
n=2r+2, \qquad D=2d+2.
\end{equation}

Let $W$ be the Weyl group of type $\~{\mathrm C}_d$
generated by $S = \{s_0, s_1, \ldots, s_d\}$ with the affine Dynkin diagram
\[
\begin{tikzpicture}[start chain]
\dnode{$0$}
\dnodenj{$1$}
\dydots
\dnode{$d-1$}
\dnodenj{$d$}
\path (chain-1) -- node[anchor=mid] {\(\Longrightarrow\)} (chain-2);
\path (chain-4) -- node[anchor=mid] {\(\Longleftarrow\)} (chain-5);
\end{tikzpicture}
\]
Then $(W,S)$ is a Coxeter group. We denote the identity of $W$ by $\id$. It is known (c.f. \cite{B86}) that  $W$ can be identified as a subgroup
of the permutation group $\mrm{Perm} (\ZZ)$ satisfying certain natural conditions (such a description for Weyl groups of affine type $A$ goes back to Lusztig).
For our purpose, we shall introduce a variant of such a description, which has been given in \cite{Shi94}. We identify $W$ with the subgroup $\mrm{Perm}^\fc (\ZZ)$
which consists of $g\in \mrm{Perm} (\ZZ)$ such that
\begin{equation}
  \label{eq:gc}
g(i+D) = g(i)+D,  g(-i) = -g(i)\quad\tfor i \in \ZZ.
\end{equation}
In particular, we always have
\[
g(0) = 0 \quad\tand\quad g(d+1) = d+1.
\]
Any element $w$ in $W\equiv \mrm{Perm}^\fc (\ZZ)$ is uniquely determined by its value on $\{1, 2, \ldots, d\}$, and we shall denote
\eq\label{def:perm}
w = \left(\ba{{ccccccc}1&2&\ldots&d\\a_1&a_2&\ldots&a_d}\right)_\fc = [a_1 , a_2, \ldots , a_d]_\fc
\endeq
to mean that $w(i) = a_i$ for $1\leq i \leq d$.
We define the {\em transposition}
\eq \label{eq:jk}
(i,j)_\fc \in W, \quad \text{ for } i \ne j,
\endeq
as the element
which swaps $kD\pm i$ and $kD\pm j$ $(k\in\ZZ)$ while fixing $\ZZ \backslash \{kD\pm i,kD\pm j|k\in\ZZ\}$ pointwise.

Let us  establish an explicit isomorphism between $W$ and the Weyl group of type $\~{\mathrm C}_d$
denoted by $\~{S}^{C}_{d}$  \cite[\S8.4]{BB05}.
(Roughly speaking, the difference is that the permutations in $W$  fix $d+1+D\ZZ$.)
The identification $\~{S}^{C}_{d} \rw W$ is given by
\eq
g' \mapsto g = [\iota( g'(1)) , \ldots , \iota( g'(d))]_\fc,
\endeq
where
\eq
 \label{iota}
\iota:\ZZ\longrightarrow \ZZ \backslash \{ d+1+D\ZZ \},
\qquad i \mapsto i + \left\lceil\frac{i-d}{D-1}\right\rceil
\endeq
is an
order-preserving bijection.
Here the floor and ceiling functions are defined as usual by
$\lfloor a\rfloor = \max\{ m\in\ZZ ~|~ m\leq a\}$  and
$\left\lceil a \right\rceil = \min \{ m\in \ZZ ~|~ m \geq a\}$ for $a\in\RR$.
In other words, we have
\eq\label{eq:iota}
\iota \big(g'(i+k(D-1)) \big) = g(i+kD),
\quad
-d\leq i\leq d,
k\in\ZZ.
\endeq
This identification shows $W$ is indeed the Weyl group of type $\~{\mathrm C}_d$.

We denote by $|X|$ the cardinality of a finite set $X$.
The length function $\ell(\cdot)$ on $W$ affords the following simple formula
(compare with the more involved formula in \cite[(8.44)]{BB05}, which is recalled in \eqref{eq:old l(g)} below).

%------------------------------------------------------------------------------------------------------------
\begin{lemma}
   \label{lem:l(g)}
The length of $g\in W$ is given by
\eq \label{eq:l(g)}
\ell(g)
=\frac{1}{2}
\left|
\left\{
(i,j)\in[1..d]\times \mathbb{Z}
~\big |~ \substack{i>j \\ g(i)<g(j)} \textup{ or }\substack{i<j \\ g(i)>g(j)}
\right\}
\right|.
\endeq
\end{lemma}
%------------------------------------------------------------------------------------------------------------
\proof
Let $g' \in \~{S}^{C}_d$ be the element identified with $g$.
It is known  \cite[(8.44), (8.45)]{BB05} that
\eq\label{eq:old l(g)}
\ell(g')=\mbox{inv}_B(g'(1),\ldots,g'(d))+
\sum_{1\leq i\leq j\leq d}\left(\left\lfloor\frac{|g'(i)-g'(j)|}{D-1}\right\rfloor
+\left\lfloor\frac{|g'(i)+g'(j)|}{D-1}\right\rfloor\right).
\endeq

Since $\iota$ is order preserving, we have $g'(i) < g'(j) \Leftrightarrow g(i) < g(j)$ for all $i,j \in [1..d]$. Hence
by \cite[(8.2)]{BB05} we have
\begin{align*}
\textup{inv}_B(g'(1),\ldots,g'(d))
&=
\Big |\{ 	(i,j)\in[1..d]^2	~|~
	\substack{ i<j \\  g'(i)>g'(j)}
\} \Big |
+
\Big |\{ 	(i,j)\in[1..d]^2	~|~
	\substack{ i\leq j \\  g'(-i)>g'(j)}
\} \Big |
\\
%&=|\{(i,j)\in[-d..d]^2~|~0<|i|\leq j,g(i)>g(j)\}|.
%
&=
\Big|\{ 	(i,j)\in[1..d]^2	~|~
	\substack{ i<j \\  g(i)>g(j)}
\} \Big|
+
\Big |\{ 	(i,j)\in[1..d]^2	~|~
	\substack{ i\leq j \\  g(-i)>g(j)}
\} \Big|
\\
&=\frac{1}{2} \Big|
\{(i,j)\in[1..d]\times [-d..d] ~|~
%(j-i) (g(j)-g(i)) < 0
\substack{i>j\\g(i)<g(j)}\textup{ or }\substack{ i<j\\ g(i)>g(j)}\} \Big|.
\end{align*}

By \eqref{iota}--\eqref{eq:iota}, we obtain that
\[
\left\lfloor\frac{|g'(i) \pm  g'(j)|}{D-1}\right\rfloor = \left\lfloor\frac{|g(i) \pm  g(j)|}{D}\right\rfloor
\qquad \text{ for } i,j \in [1..d].
\]

%Note $\~{S}^C_d \subset \~{S}_{D-1}$, where $\~{S}_{D-1}$ denotes the Weyl group of affine type A  in \cite[\S8.3]{BB05}.
%It follows by \cite[(8.31)]{BB05} that, for $i,j \in [1..d]$,
%\[ \bA{
%\left\lfloor\frac{|g'(i) - g'(j)|}{D-1}\right\rfloor
%&=
%\Big|\{k\in\ZZ~|~ g'(j) > g'(i+k(D-1)) \} \Big|
%\\
%&\qquad + \Big|\{k\in\ZZ~|~ g'(i) > g'(j+ k(D-1))\} \Big|.
%}
%\]

We regard $W \subset \~{S}_D$,  the Weyl group of affine type A  in \cite[\S8.3]{BB05}.
\cite[(8.31)]{BB05} can be rephrased using $g$ (instead of $g'$) as follows:
for $g \in W \subset \~{S}_D$ and for $i,j \in [1..d]$, we have
\[\bA{
\left\lfloor\frac{|g(i) - g(j)|}{D}\right\rfloor
= |\{k\in\ZZ~|~ g(j) > g(i+k(D)) \}|  +|\{k\in\ZZ~|~ g(i) > g(j+ k(D))\}|.  }
\]

A detailed calculation shows that
\[
\ba{{ll}
   \sum\limits_{1\leq i\leq j\leq d}
   &\left\lfloor\frac{|g(i)+g(j)|}{D}\right\rfloor
 \\
&
    =\sum\limits_{1\leq i<j\leq d}
    \bp{
    |\{k\geq 1~|~g(j)>g(-i+kD)\}|
    +|\{k\geq 1|g(i)<g(-j-kD)\}|
    }
\\
&+
\sum\limits_{1\leq i\leq d}
\bp{
|\{k\geq 1~|~g(i)>g(-i+kD)\}|
+|\{k\geq 1|g(i)<g(-i-kD)\}|
}
\\
%&=&\sum_{1\leq i<j\leq d}\frac{1}{2}(\sharp\{k\geq 1|g(j)>g(-i+kD)\}+\sharp\{k\geq 1|g(i)>g(-j+kD)\})\\
%&&+\sum_{1\leq i<j\leq d}\frac{1}{2}(\sharp\{k\geq 1|g(i)<g(-j-kD)\}+\sharp\{k\geq 1|g(j)<g(-i-kD)\})\\
%&&+\sum_{1\leq i\leq d}\frac{1}{2}(\sharp\{k\geq 1|g(i)>g(-i+kD)\}+\sharp\{k\geq 1|g(i)>g(\frac{kD}{2})\})\\
%&&+\sum_{1\leq i\leq d}\frac{1}{2}(\sharp\{k\geq 1|g(i)<g(-i-kD)\}+\sharp\{k\geq 1|g(i)<g(\frac{-kD}{2})\})\\
&=\frac{1}{2}\sum\limits_{k=1}^{\infty}
\left(
\ba{{ll}
&|\{(j,-i+kD)~|~1\leq i<j\leq d,g(j)>g(-i+kD)\}|\\
+&|\{(i,-j+kD)~|~1\leq i<j\leq d,g(i)>g(-j+kD)\}|\\
+&|\{(i,-j-kD)~|~1\leq i<j\leq d,g(i)<g(-j-kD)\}|\\
+&|\{(j,-i-kD)~|~1\leq i<j\leq d,g(j)<g(-i-kD)\}|\\
+&|\{(i,-i+kD)~|~1\leq i\leq d,g(i)>g(-i+kD)\}|\\
+&|\{(i,\frac{kD}{2})~|~1\leq i\leq d,g(i)>g(\frac{kD}{2})\}|\\
+&|\{(i,-i-kD)~|~1\leq i\leq d,g(i)<g(-i-kD)\}|\\
+&|\{(i,-\frac{kD}{2})~|~1\leq i\leq d,g(i)<g(-\frac{kD}{2})|\}
}\right).
}
\]
Similarly we have
\begin{align*}
 \ba{{ll}
 \sum\limits_{1\leq i\leq j\leq d}
 & \left\lfloor\frac{|g(i)-g(j)|}{D}\right\rfloor
\\
%&=&\sum_{1\leq i<j\leq d}\left\lfloor\frac{|g(i)-g(j)|}{D}\right\rfloor\\
&
=
\sum\limits_{1\leq i<j\leq d}
\bp{
|\{k\geq 1~|~g(i)>g(j+kD)\}|
+|\{k\geq 1~|~g(j)>g(i+kD)\}|
}
\\
%&=&\sum_{k=1}^{\infty}(\sharp\{(i,j+kD)|1\leq i<j\leq d,g(i)>g(j+kD)\}+\\
%&&\quad\quad\quad\sharp\{(j,i+kD)|1\leq i<j\leq d,g(j)>g(i+kD)\})\\
&=\frac{1}{2}\sum\limits_{k=1}^{\infty}
\left(
\ba{{ll}
&|\{(i,j+kD)~|~1\leq i<j\leq d,g(i)>g(j+kD)\}|\\
+&|\{(j,i-kD)~|~1\leq i<j\leq d,g(i-kD)>g(j)\})|\\
+&|\{(j,i+kD)~|~1\leq i<j\leq d,g(j)>g(i+kD)\}|\\
+&|\{(i,j-kD)~|~1\leq i<j\leq d,g(j-kD)>g(i)\}|)
}
\right).
}
\end{align*}
The lemma then follows by simplifying the summation of the relevant terms above.
\endproof
%------------------------------------------------------------------------------------------------------------
%For $a,b \in \ZZ$, define periodic permutations and periodic symmetric permutations, respectively, by
%\eq
%\cyaA{a, b} = \prod_{k\in\ZZ} (a+kD, b+kD)
%,
%\quad
%\cyaB{a,b} = \cyaA{a,b}\cyaA{-a,-b}
%.
%\endeq
%Under the identification the generators  are mapped as below:
Via the notation   \eqref{def:perm}, the generators of $W$ are given by
\eqnarray
\bA{
s_0 & %= \cyaA{1, -1}
= [-1,2,3,\ldots, d-1, d]_\fc,
\\
s_d & %= \cyaA{d, d+2}
= [1,2,3,\ldots,d-1,d+2]_\fc,
\\
s_i & %= \cyaB{i, i+1}
=[1,\ldots,i-1,i+1,i,i+2,\ldots, d]_\fc =(i,i+1)_\fc,  \quad\text{ for } i =1, \ldots, d-1.
}
\endeqnarray

%------------------------------------------------
\section{Parabolic subgroups and cosets}

Denote the set of (weak) compositions of $d$ into $r+2$ parts (where ``weak'' means a possible zero part is allowed) by
\eq \label{def:Ld}
\Ld = \Ld_{r,d} = \Big \{ \ld = (\ld_0, \ld_1, \ldots, \ld_{r+1})\in\NN^{r+2} ~\big |~ \sum_{i=0}^{r+1} \ld_i = d \Big \}.
\endeq
For  $\ld \in \Ld$, we shall denote by $W_\ld$  the parabolic (finite) subgroup of $W$
generated by $S\backslash\{s_{\ld_0}, s_{\ld_{0,1}},\ldots, s_{\ld_{0,r}}\}$,
where $\ld_{0,i} =\ld_0 + \ld_1 + \ldots + \ld_i$ for $0 \leq i \leq r$;
note $\ld_{0,0} = \ld_0$ and $\ld_{0,r} =d -\ld_{r+1}$.
We define the integral intervals $R_i^\ld$ by
\eqnarray  \label{eq:Ri}
R_i^\ld = \bc{
[-\ld_0.. \ld_0]
&\tif i=0,
\\
(\ld_{0, i-1}.. \ld_{0,i}]
&\tif i \in [1..r],
\\
[d+1-\ld_{r+1}.. d+1+ \ld_{r+1}]
&\tif i=r+1.
}
\endeqnarray
Recall that $n=2r+2$ and $D=2d+2$ in \eqref{nD}. We further extend the definition of $R^\ld_i$ for all $i\in\ZZ$ recursively by letting
\eq   \label{eq:Ri2}
R^\ld_{-i} = \{-x ~|~ x \in R_i^\ld\}, \quad
R^\ld_{i+n} = \{x+D ~|~ x \in R_i^\ld\}.
\endeq
Then the sets $\{R^\ld_i\}_{i\in \ZZ}$ partition the set $\ZZ$, that is,
\[
R^\ld_i \cap R^\ld_i =\emptyset \; \text{ for } i\neq j,
\qquad \ZZ =\bigsqcup_{i\in \ZZ} R^\ld_i.
\]

Denote by $\Stab(X)$ the stabilizer subgroup of the action of $W \equiv \text{Perm}^\fc (\ZZ)$ on $\ZZ$, for
any subset $X \subset \ZZ$.
%-----------------------------------------------------------------------------------------------------------
\begin{lem}\label{lem:stabR}
For any $\ld \in \Lambda$, we have
$
W_\ld = \bigcap_{i = 0}^{r+1} \Stab (R_i^\ld).
$
\end{lem}
%-----------------------------------------------------------------------------------------------------------
\proof
By \cite[Proposition 8.4.4]{BB05} we have, for each $0\le i \le r$,
\[
W_{S\backslash\{s_{\ld_{0,i}}\}} =
\textup{Stab}([-\ld_{0,i}..\ld_{0,i}]) \cap \textup{Stab}([\ld_{0,i}+1..D-\ld_{0,i}-1]).
\]
The lemma follows by taking the intersection $W_\ld = \bigcap_{i = 0}^{r} W_{S\backslash\{s_{\ld_{0,i}}\}}$.
\endproof

Let
\eq  \label{eq:Dmin}
\D_\ld = \big\{g\in \W ~|~ \ell(wg) = \ell(w) + \ell(g), \forall  w\in \W_\ld \big\}.
\endeq
Then $\D_\ld$ (respectively, $\D_\ld^{-1}$) is the set of minimal length right (respectively, left)
coset representatives of $\W_\ld$ in $\W$.
Denote by
\eq  \label{eq:Dmin2}
\D_{\ld\mu} = \D_\ld \cap \D_\mu^{-1}
\endeq
the set of  minimal length double coset representatives
for $W_\ld \backslash W /W_\mu$.
%-----------------------------------------------------------------------------------------------------------
\begin{lemma}\label{lem:Dld}
For any $g \in \W$ and $\ld \in \Ld$, the following are equivalent:
\enua
\item $g\in\D_\ld$;

\item
$g^{-1}$ is order-preserving on $R^\ld_i$, for all $i \in [0..r+1]$;

\item
$g^{-1}$ is order-preserving on $R^\ld_i$, for all $i \in \ZZ$.
\endenua
\end{lemma}
%-----------------------------------------------------------------------------------------------------------
\proof
By the argument following \cite[Proposition 8.4.4]{BB05}, $\D_\ld$ in \eqref{eq:Dmin} can be written as
\[
\D_\ld = \left\{ g \in \W ~\middle|~ \ba{{ll} g^{-1}(0) < \ldots < g^{-1}(\ld_0),\\ g^{-1}(1+\ld_{0,i}) < \ldots < g^{-1}(\ld_{0,i+1}),
\forall i\in [1..r-1] \\ g^{-1}(1+\ld_{0,r}) < \ldots < g^{-1}(d+1)}\right\}.
\]
Note that $g(-i) = -g(i)$ and $g(0) = 0$, so the condition ``$g^{-1}(0) < \ldots < g^{-1}(\ld_{0})$'' is equivalent to
$g^{-1}(-\ld_{0}) < \ldots < g^{-1}(0) = 0 < \ldots < g^{-1}(\ld_{0})$. Similarly, for $g \in \D_\ld$, we have
$g^{-1}(d+1-\ld_{r+1}) < \ldots < g^{-1}(d+1) = d+1 < \ldots < g^{-1}(d+1+\ld_{r+1})$.
Hence (a) is equivalent to (b).

The equivalence of (b) and (c) follows from the periodicity condition \eqref{eq:gc}.
\endproof

The following proposition is standard and can be found in \cite[Proposition 4.16, Lemma ~4.17 and Theorem~ 4.18]{DDPW};
for Part (a) also see Proposition~\ref{prop:delta} below.

\begin{prop}\label{prop:doublecoset}
Let $\ld,\mu \in \Ld$ and $g \in \D_{\ld\mu}$.
\enua
\item There is a weak composition $\delta = \delta(\ld, g, \mu) \in \Ld_{r',d}$ for some $r'$ such that
\[
W_{\delta} = g^{-1} W_\ld g \cap W_\mu.
\]
\item The map $W_\ld \times (\D_\delta \cap W_\mu) \rw W_\ld g W_\mu$ sending $(x,y)$ to $xgy$ is a bijection;
moreover, we have $\ell(xgy) = \ell(x) + \ell(g) + \ell(y)$.
\item The map $(\D_\delta\cap \W_\mu) \times \W_\delta \rw \W_\mu$ sending $(x,y)$ to $xy$ is a bijection;
moreover, we have $\ell(x) + \ell(y) = \ell(xy)$.
\endenua
\end{prop}

%=======================
\section{Affine Schur algebra via Hecke}
 \label{sec:Sj}

Let $\mA =\ZZ[v,v^{-1}]$. The Hecke algebra $\bH = \bH(W)$  %of type $\~{\MATHRM C}_d$
is an $\mA$-algebra with a basis $\{T_g ~|~ g\in W\}$ satisfying
\[
\ba{{lllll}
T_w T_{w'} = T_{ww'}
&
\tif
\ell(ww') = \ell(w) + \ell(w'),
\\
(T_s+1) (T_s - v^2) = 0,
&\tfor
s \in S.
}
\]
For any finite subset $X \subset \W$ and for $\ld\in\Ld$ \eqref{def:Ld}, set
\eq  \label{eq:x}
T_X = \sum_{w\in X} T_w
\quad\textup{and}\quad
x_\ld = T_{\W_\ld}.
\endeq
For $\ld,\mu\in\Ld$ and $g\in \D_{\ld\mu}$, we consider a right $\bH$-linear map
$
\phi_{\ld\mu}^g \in \Hom_\bH(x_\mu \bH, \bH)$,
sending $x_\mu$ to $T_{W_\ld g W_\mu}.$
Thanks to Proposition~ \ref{prop:doublecoset}(b),
we have $T_{W_\ld g W_\mu} = x_\ld T_g T_{\D_\delta \cap W_\mu}$ for some $\delta\in \Ld_{r',d}$,
and hence we have constructed a right $\bH$-linear map
\eq  \label{phi}
\phi_{\ld\mu}^g \in \Hom_\bH(x_\mu\bH, x_\ld\bH),
\qquad x_\mu \mapsto T_{W_\ld g W_\mu} = x_\ld T_g T_{\D_\delta \cap W_\mu}.
\endeq
The {\em affine Schur algebra} $\Sjj$ is defined as the following $\mA$-algebra
\eq
\label{def:Sjj}
\Sjj = \textup{End}_{\bH}
\Bp{
\mathop{\oplus}_{\ld\in\Ld} x_\ld \bH
}
= \bigoplus_{\ld,\mu \in \Ld} \Hom_{\bH} (x_\mu \bH, x_\ld \bH)
.
%= \Span \left\{ \phi_{\ld\mu}^g %~|~ \ld,\mu \in \Ld, g \in \D_{\ld\mu}
%\right\}.
\endeq
Introduce the following subset of $\Ld \times W \times \Ld$:
\eq  \label{Dnd}
 \D_{n,d} %=\{ (\ld, g, \mu) \in \Ld \times W \times \Ld~|~ g \in \D_{\ld\mu} \}
 =\bigsqcup_{\ld, \mu \in \Ld} \{\ld\} \times \D_{\ld,\mu} \times \{\mu\}.
\endeq

A  formal argument as in \cite{Du92, Gr97} is applicable to our setting and gives us the following.
%-----------------------------------------------------------------------------------------------------------
\begin{lemma}
\label{lem:basis}
The set $\{  \phi_{\ld\mu}^g ~|~ (\ld,g, \mu) \in \D_{n,d} \}$
forms an $\mA$-basis of $\Sjj$. %, which is called the \textit{standard basis} of $\Sjj$.
\end{lemma}

We denote
\[
\Tt_n = \big\{A=(a_{ij}) \in \text{Mat}_{\ZZ\times\ZZ}(\NN)~|~ a_{ij}=a_{i+n, j+n}, \forall i, j\in \ZZ \big \}.
\]
We further consider the following subset of $\Tt_n$:
\begin{align}
\label{eq:Xind}
\begin{split}
\Xi_{n,d} = \Big\{A=(a_{ij}) \in \text{Mat}_{\ZZ\times\ZZ}(\NN)~\big |~
& a_{-i,-j} =a_{ij}=a_{i+n, j+n}, \forall i, j\in \ZZ;
\\
&
a_{00}, a_{r+1,r+1} \text{ are odd};
\sum_{1 \leq i \leq n} \sum_{j\in\ZZ} a_{ij}= D
 \Big \}.
 \end{split}
\end{align}
%Set $\Xi_{n,d}$ to be the subset of $\Tt_n$ in which each element $A = (a_{ij})$ satisfies additionally that
%\itm
%\item $a_{-i,-j} = a_{ij}$ for all $i,j \in \ZZ$;
%\item $a_{00}$ and $a_{r+1,r+1}$ are odd;
%\item $\sum_{1 \leq i \leq n} \sum_{j\in\ZZ} a_{ij}= D$.
%\enditm
For $k,\ell\in\ZZ$, we define a matrix  $E^{k\ell} \in \Tt_n$ by
\eq  \label{Ekl}
E^{k\ell} = (E^{k\ell}_{ij})_{i,j\in\ZZ},
\quad
E^{k\ell}_{ij} =\bc{
1&\tif (i,j) = (k+tn,  \ell+tn) \textup{ for some } t\in\ZZ,
\\
0&\otw.
}
\endeq
For any $T= (t_{ij}) \in \Tt_n$,  set
\eq  \label{Ttt}
T_\tt = (t_{\tt,ij}), \quad t_{\tt,ij} = t_{ij}+t_{-i,-j}.
\endeq
For $T = (t_{ij})\in \Tt_n$, define its type A row sum vector $\roA(T)= (\roA(T)_k)_{k\in\ZZ}$
and type A column sum vector $\coA(T)= (\coA(T)_k)_{k\in\ZZ}$
%and full column sum $\coA(T)$
by
\eq  \label{eq:rowsum:a}
\roA(T)_k = \sum\limits_{j\in\ZZ} t_{kj}
\quad\textup{and}\quad
\coA(T)_k = \sum\limits_{i\in\ZZ} t_{ik}.
\endeq
Let
\eq \label{eq:Xin}
\Xi_n= \bigcup_{d\in\NN} \Xi_{n,d}.
\endeq
For $A = (a_{ij})\in \Xi_n$, we set
\eq\label{def:a'}
%a'_{kk} =  \frac{1}{2}(a_{kk}-1), \qquad \text{for } k \in \ZZ(r+1).
%
a'_{ij} =
\bc{\frac{1}{2}(a_{ii}-1)&\tif i = j \in \ZZ(r+1),
\\ a_{ij}
&\textup{otherwise}.  }
\endeq
For any $A \in \Xi_n$,
we define its type C row sum vector $\roC(A)= (\roC(A)_0, \ldots, \roC(A)_{r+1} )$
and type C column sum vector $\coC(A)= (\coC(A)_0, \ldots, \coC(A)_{r+1} )$ by
\eqnarray
\ba{{l}
\roC(A)_k =
\bc{
    a'_{00} + \sum\limits_{j \ge 1}a_{0j} &\tif k=0,
    \\
   a'_{r+1,r+1}  +  \sum\limits_{j\leq r}a_{r+1,j} &\tif k=r+1,
    \\
     \roA(T)_k &\tif 1\le k \le r.
    }
\\
\coC(A)_k =
\bc{
   a'_{00}  +   \sum\limits_{i\geq 1}a_{i0} &\tif k=0,
    \\
  a'_{r+1,r+1}  +    \sum\limits_{i \leq r}a_{i,r+1} &\tif k=r+1,
    \\
     \coA(A)_k & \tif 1\le k \le r.
    }
}
\endeqnarray
%and let
%\[
%\Xi^{(r)} = \{ A=(a_{ij}) \in \Xi_n~|~ a_{ij} = a_{i+N,j+N} \textup{ for all } i,j\}.
%\]
Each $A \in \Xi_n$ is uniquely determined by $\{a_{ij}~|~ (i,j) \in I^+\}$, where
\eq
 \label{plus}
I^+ = \big(\{0\}\times \NN \big)\sqcup \big([1..r]\times\ZZ \big) \sqcup \big(\{r+1\}\times \ZZ_{\leq r+1} \big)
\endeq
is the index set corresponding to the ``first half-period''.
%With these notations, $\Xi_n$ can be expressed as follows:
%\[
%\Xi_n =
%(2\NN+1) E^{00} + (2\NN+1) E^{r+1,r+1} + \sum_{(i,j)\in I^+} \NN \Ett^{ij}.
%\]

%====================================
\section{Set-valued matrices}

We introduce a ``higher-level'' structure of $\Xi_n$, which
will facilitate the proof of the multiplication formula in next chapter.
Let $\Xi_{n,d}^\cP$ be the set of $\ZZ\times\ZZ$ matrices  $\bbA = (\bbA_{ij})$  with entries being finite subsets of $\ZZ$
which satisfy Conditions (P1)--(P5) below:
\enu
\item[(P1)] $\bigsqcup\bbA_{ij}=\ZZ$;
\item[(P2)] $\bbA_{i+n,j+n} = \{x+D~|~ x \in \bbA_{ij}\}$ for all $i,j\in\ZZ$;
\item[(P3)] $\bbA_{-i,-j} = \{-x~|~ x \in \bbA_{ij}\}$ for all $i,j\in\ZZ$;
\item[(P4)] $0 \in \bbA_{00}$ and $d+1 \in \bbA_{r+1,r+1}$;
\item[(P5)] $\sum_{1\leq i\leq n}\sum_{j\in\ZZ} |\bbA_{ij}| = D$.
\endenu
Set
$$
\Xi^\cP_n = \bigsqcup_{d\in\NN} \Xi_{n,d}^\cP.
$$
Again, each $\bbA \in \Xi^\cP_n$ is uniquely determined by $\{\bbA_{ij}~|~ (i,j) \in I^+\}$.

Recalling $I^+$ in \eqref{plus}, we denote
\eq\label{def:Ia}
\Ia = I^+ \big \backslash \{(0,0), (r+1,r+1)\}.
%, \qquad I^+ =\Ia  \cup  \{(0,0), (r+1,r+1)\}.
\endeq

%\Def\label{alg:higherlevel}
We define a map
\eq  \label{stP}
\cP: \Xi_{n,d} \longrightarrow \Xi_{n,d}^\cP,
\qquad A\mapsto  A\stdP
\endeq
as follows.
For each $A= (a_{ij}) \in\Xi_{n,d}$,  a matrix $A\stdP \in \Xi_{n,d}^\cP$  is determined  by (a1)--(a2) below
via  ``row-reading''  (thanks to  Conditions (P2)--(P3) for $\Xi_{n,d}^\cP$):
%(see \eqref{def:a'} for $a'_{ij}$):
\enu
\item[(a1)]
 Set
$(A\stdP)_{00} = \Big{[}-a'_{00}~..~a'_{00} \Big{]}$, \quad
$(A\stdP)_{r+1,r+1} = \Big{[}d+1-a'_{r+1,r+1}~..~d+1+a'_{r+1,r+1} \Big{]}.$

\item[(a2)]
 For $(i,j) \in \Ia$,
set
\[
(A\stdP)_{ij} = \Big{(}\sum\limits_{l=0}^{i-1} \roC(A)_{l} + \sum\limits_{k<j}a_{ik}~..~ \sum\limits_{l=0}^{i-1} \roC(A)_{l} + \sum\limits_{k\leq j}a_{ik}\Big{]}.
\]
\endenu
%Note the above conditions (1)--(2) determine uniquely the matrix $A\stdP$  thanks to Conditions (P2)--(P3) for $\Xi_{n,d}^\cP$.
%\endDef
By definition we have a map
\eq  \label{norm}
|\cdot | : \Xi_{n,d}^\cP \longrightarrow \Xi_{n,d},
\qquad |\bbA|=(|\bbA_{ij}|)_{i,j\in\ZZ}.
\endeq
%For any $\bbA\in\Xi_{n,d}^\cP$, it is obvious that $|\bbA|:=(|\bbA_{ij}|)\in\Xi_{n,d}$.
Moreover, $|A\stdP|=A$.

%\Alg\label{alg:gstd}
We further define a map
\eq  \label{stW}
\gstd:  \Xi_{n,d}^\cP \longrightarrow W,
\qquad \bbA \mapsto \gstd_\bbA
\endeq
as follows.
On $I^+$, let $<_{\text{lex}}$ be the lexicographical order such that $(i,j) <_{\text{lex}} (x,y)$ if and only if
$i < x$ or ($i = x$ and $j < y$).
Let $\bbA = (\bbA_{ij}) \in \Xi^\cP_{n,d}$, and set $A = (a_{ij}) = |\bbA|$.
The Weyl group element $\gstd_\bbA\in W$ is determined
by  (g1)--(g2) below via ``column-reading''  (thanks to the periodicity condition~\eqref{eq:gc}):
\enu
\item[(g1)]
For $(i,j)\in I^+$, we set
\[
I^{(j,i)} =
\bc{
~[- a'_{ii} .. a'_{ii}],
&\tif (j,i) = (0,0) \textup{ or }(r+1,r+1),
\\
~[1 .. a_{ji}],
&\otw.
}
\]
Then set $\bbA_{j,i} = \{a^{(j,i)}_l ~|~l \in I^{(j,i)}\}$ such that $a^{(j,i)}_l < a^{(j,i)}_{l+1}$ for admissible $l$.

\item[(g2)]
Let $1\le k \le d$. If $k\le a'_{00}$, set $\gstd_\bbA(k) = a^{(0,0)}_k=k$;
if $k>a'_{00}$, then
 find the unique $(i,j) \in I^+\backslash\{(0,0)\}$ and $m\in [1.. a'_{ji}]$ such that
\[
k = a'_{00}+ \sum_{(x,y) \in I^+\backslash\{(0,0)\}, (x,y) <_{\text{lex}} (i,j)} a_{yx} + m,
\]
and then set $\gstd_\bbA(k) = a^{(j,i)}_m$.
\endenu
%\endAlg

%------------------------------------------------------------------------------------------------------------
\exa
Let $d=3, D=8, r=0, n=2,$ and let
\[
A
=\bM{[ccc:c:cccccc]
\ddots&&&&&
\\
2&0&1&0&2&&
\\
\hdashline
&&&3&&
\\
\hdashline
&&2&0&1&0&2
\\
&&&&&3&
\\
&&&&2&0&1
\\
&&&&&&\ddots
\\} = 3E^{00} + 2E^{1,-1}_\tt + E^{11}.
\]
Here the dashed stripes indicate the 0th column/row.
We have
\[
A\stdP
=\bM{[ccc:c:cccccc]
\ddots&&&&&
\\
\{-6,-5\}&\varnothing&\{-4\}&\varnothing&\{-3,-2\}&&
\\
\hdashline
&&&[-1..1]&&
\\
\hdashline
&&\{2,3\}&\varnothing&\{4\}&\varnothing&\{5,6\}
\\
&&&&&[7..9]&
\\
&&&&\{10,11\}&&\{12\}
\\
&&&&&&\ddots
\\}.
\]
The reading
$$\cdots,-1,0,1,-3,-2,4,10,11,7,8,9,5,6,12,\cdots$$
gives $\gstd_{A\stdP} = [1,-3,-2]_\fc$; see \eqref{def:perm}.
%
%On the other hand, for $\bbA = A\stdP$, we have $\bbA_{0,0} = [-1..1], \bbA_{-1,1} = %\{-3,-2\}, \bbA_{1,1} = \{4\}$. Hence for $(i,j) \in I^+$, $a^{(j,i)}_{m}$ are defined by
%\[
%a^{(0,0)}_{-1} = -1
%<
%a^{(0,0)}_{0} = 0
%<
%a^{(0,0)}_{1} =1,
%\quad
%a^{(-1,1)}_{1} = -3
%<
%a^{(-1,1)}_{2} = -2,
%\quad
%a^{(1,1)}_0 = 4.
%\]
%Therefore, $\gstd_{A\stdP} = [1,-3,-2]_\fc$; see \eqref{def:perm}.
\endexa
%-----------------------------------
%%
%%
%%

%=========================================================
\section{A bijection $\kappa$}
\label{sec:kappa}

For $\ld, \mu \in \Ld$, we denote
\begin{align*}
\Xi_{n,d}(\ld,\mu) &= \{ A \in \Xi_{n,d} ~|~ \roC(A) = \ld, \coC(A) = \mu\}.
\end{align*}
Then we have $\Xi_{n,d} =\bigsqcup_{\ld,\mu \in \Ld} \Xi_{n,d}(\ld,\mu).$
We define a map
%\[
%\ba{{cccccccccc}
%\kappa_{\ld\mu}^\cP&:\W &\longrightarrow &\Xi^\cP_{n,d}(\ld,\mu)&
%\textup{and}&\kappa_{\ld\mu}&:\W &\longrightarrow &\Xi_{n,d}(\ld,\mu)\\
%&g &\mapsto &(R_i^\ld \cap gR_j^\mu), &&& g &\mapsto &(\big | R_i^\ld \cap gR_j^\mu \big |).
%}
%\]
\eq
  \label{kalm}
\kappa_{\ld\mu}:\W \longrightarrow \Xi_{n,d}(\ld,\mu),
\qquad
g \mapsto (\big | R_i^\ld \cap gR_j^\mu \big |).
\endeq
%If we identify $W \equiv \{\ld\} \times W \times \{\mu \}$ and r
Recalling $\Ld$ from \eqref{def:Ld},
we can assemble the maps $\kappa_{\ld\mu}$  above for various
$\ld, \mu$ to a map
$$
\kappa: \Ld \times \W \times \Ld \longrightarrow \Xi_{n,d},
\qquad \kappa(\ld, g, \mu) =\kappa_{\ld,\mu}(g).
$$
By restriction to $\D_{n,d}$ \eqref{Dnd}, we obtain a map
\eq
 \label{eq:ka}
\kappa: \D_{n,d} \longrightarrow \Xi_{n,d},
\qquad \kappa(\ld, g, \mu) =\kappa_{\ld,\mu}(g).
\endeq

%------------------------------------------------------------------------------------------------------------
\begin{lemma}\label{lem:kappa}
The map $\kappa:\D_{n,d} \longrightarrow \Xi_{n,d}$
with $\kappa(\ld, g, \mu) =(\big | R_i^\ld \cap gR_j^\mu \big |)$  is a bijection.
Moreover, if $(\ld,g,\mu)=\kappa^{-1}(A)$ for $A \in \Xi_{n,d}$,
then $\ld = \roC(A)$, $\mu = \coC(A)$, and $g = \gstd_{A\stdP}$.
\end{lemma}
%------------------------------------------------------------------------------------------------------------
\begin{proof}
We shall show equivalently that for fixed $\ld,\mu \in \Ld$, the restriction of $\kappa_{\ld\mu}$ to $\D_{\ld\mu}$,
$\kappa_{\ld\mu}|_{\D_{\ld\mu}}: \D_{\ld\mu} \longrightarrow \Xi_{n,d}(\ld,\mu)$,
 is a bijection.

Let $\Xi^\cP_d(\ld,\mu) = \{\bbA \in \Xi_{n,d}^\cP ~|~ \roC(|\bbA|) = \ld, \coC(|\bbA|) = \mu \}$. Consider the map
\[
\kappa_{\ld\mu}^\cP:\W \longrightarrow \Xi^\cP_{n,d}(\ld,\mu),
\qquad
g \mapsto (R_i^\ld \cap gR_j^\mu),
\]
so we have $\kappa_{\ld\mu}  =|\cdot | \circ \kappa_{\ld\mu}^\cP$.
The map $\kappa^\cP_{\ld\mu}$ is a surjection since $\kappa^\cP_{\ld\mu}(\gstd_\bbA) = \bbA$ for $\bbA \in \Xi_{n,d}^\cP(\ld,\mu)$.
Furthermore, given $g \in (\kappa^\cP_{\ld\mu})^{-1}(\bbA)$, by Lemma \ref{lem:Dld} we have
$%\[
g = \gstd_{\bbA}\textup{ if and only if }g \in \D_{\ld\mu}.
$ %\]
Thus the restriction $\kappa_{\ld\mu}^\cP|_{\D_{\ld\mu}}$ is a bijection.
On the other hand, for each $A\in \Xi_{n,d}(\ld,\mu)$ we have $|A\stdP| = A$,
and hence $\kappa_{\ld\mu} =(|\cdot | \circ \kappa_{\ld\mu}^\cP)_{\ld\mu}$ is a surjection.

Moreover, given $\bbA \in \Xi^\cP_d(\ld,\mu)$ with $|\bbA| =A$, applying Lemma~ \ref{lem:Dld} again we have
$ %\[
\bbA = A\stdP\textup{ if and only if }\gstd_{\bbA} \in \D_{\ld\mu}.
$ %\]
Therefore $(|\cdot| \circ\kappa_{\ld\mu}^\cP)|_{\D_{\ld\mu}} = \kappa_{\ld\mu}|_{\D_{\ld\mu}}$ is a bijection.

The second statement of the lemma can be read off from the above argument.
\end{proof}
%------------------------------------------------------------------------------------------------------------
For each $A = \kappa(\ld,g,\mu) \in \Xi_{n,d}$, we use the bijection $\kappa$ in Lemma~\ref{lem:kappa} to introduce new notation
\eq\label{def:eA}
e_A = \phi_{\ld\mu}^g.
\endeq
Hence Lemma~\ref{lem:basis} can be rephrased that $\{e_A ~|~ A\in\Xi_{n,d}\}$ forms an $\mA$-basis of $\Sjj$.

Let $A\in\Xi_{n,d}$.  We choose $k_j\ge 0$ for $0\le j \le r+1$ such that $a_{ij}=0$ unless $|i-j|\le k_j$.
We define a weak composition
\eq  \label{eq:delta}
\delta(A) \in \Ld_{\mathfrak{r},d},
\endeq
with
$
\mathfrak{r} =k_0 +\sum_{j=1}^r (2k_j+1) + k_{r+1}$ as follows.
The composition $\delta(A)$ starts with the $(k_0+1)$ entries in $0$th column,
$a'_{00}, a_{10}, a_{20},  \ldots, a_{k_0 0},$
followed by the $(2k_j+1)$ entries $a_{ij}$ in $j$th column when $i$ runs up the interval $[j-k_j ..  j+k_j]$ for $j=1, \ldots, r$,
and finally followed by the $(k_{r+1}+1)$ entries in the $(r+1)$st column,
%a_{1-k_1,1}, a_{2-k_1,1}, \ldots, a_{1+k_1,1},
%a_{2-k_2,2}, a_{2-k_2,2}, \ldots, a_{2+k_2,2}],
%\ldots \ldots,
%a_{r-k_r,r}, a_{r+1-k_r,r}, \ldots a_{r+k_r,r},
$a_{r+1-k_{r+1}, r+1}, \ldots, a_{r,r+1}, a'_{r+1,r+1}.$

%%%%%%%%%%%%%%%%%%%%%%
\iffalse
%%%%%%%%%%%%%%%%%%%%%%
we define
$$
\delta(A) = (\delta(A)_0, \ldots, \delta(A)_{r'+1} ) \in \Ld_{r',d}
$$
for some $r'$ by the following procedure:
\enu
\item Set $\delta _0  = \frac{1}{2} (a_{00}-1)$ (possibly zero);
set $(\delta^{(0)}_{1}, \ldots, \delta^{(0)}_{k_0})$ for some $k_0 \in \NN$ to be the composition of $\coC(A)_0-\delta_0$ obtained from $( a_{10}, a_{20},  \ldots)$ by deleting all zero entries.
\item For each $j = 1, \ldots, r$, set $(\delta^{(j)}_{1}, \ldots, \delta^{(j)}_{k_j}) \in \Ld_{k_j,\ld_j}$ for some $k_j \in \NN$ to be the composition of $\coC(A)_j$   obtained from
$(\ldots,~a_{-1,j}, a_{0j}, a_{1j},  \ldots)$ by deleting all zero entries.
\item
Set $(\delta^{(r+1)}_{1}, \ldots, \delta^{(r+1)}_{k_{r+1}})$ for some $k_{r+1} \in \NN$ to be the composition  of $\coC(A)_{r+1}-\delta_{r'+1}$ obtained by deleting all zero entries from
$(\ldots,a_{r-1,r+1}, a_{r,r+1})$, where $\delta_{r'+1} = \frac{1}{2}(a_{r+1,r+1}-1)$ (possibly zero) and $r' = k_0 + \ldots + k_{r+1}$.
\item Finally, set
\eq %\label{eq:delta}
\delta(A) = (\delta_0, \delta^{(0)}_{1}, \ldots, \delta^{(0)}_{k_0}, \delta^{(1)}_1, \ldots, \delta^{(1)}_{k_1},\ldots, \delta^{(r+1)}_1, \ldots, \delta^{(r+1)}_{k_{r+1}}, \delta_{r'+1}).
\endeq
\endenu
%%%%%%%%%%%%%%%%%%%%%%%
\fi
%%%%%%%%%%%%%%%%%%%%%%%

\rmk
We can remove any zeroes that are not in the first or the last place in a weak composition $\ld \in \Ld_{\mathfrak{r},d}$
without changing the parabolic subgroup $W_\ld$,
e.g.,
$W_{(2,0,2)} = W_{(2,2)}$.
(However, removing zeroes in the first or the last place in a weak composition changes the corresponding parabolic subgroup.)
%For example, $W_{(0,1,1,0)}, W_{(0,1,1)}, W_{(1,1,0)}$ and $W_{(1,1)}$ are four distinct parabolic subgroups.
Therefore, while the composition $\delta(A)$ in \eqref{eq:delta} depends on $A$ as well as on the choices of $k_j$,
the parabolic subgroup $W_{\delta(A)}$ only depends on $A$.
\endrmk
%-----------------------------------------------------------------------------------------------------------

We can now make the construction in Proposition \ref{prop:doublecoset}(a) explicit.
\begin{prop}\label{prop:delta}
Let $A=\kappa(\ld,g,\mu)$ for  $\ld,\mu \in \Ld, g\in \D_{\ld\mu}$. Then $W_{\delta(A)} = g^{-1} W_\ld g \cap W_\mu.$
%Namely, $\delta(A)$ is one (possible weak) composition $\delta$ described in Proposition \ref{prop:doublecoset}.
\end{prop}

\proof
By Lemma \ref{lem:stabR}, we have
\[
\bA{
g^{-1} \W_\ld g \cap \W_\mu
&=
\Bp{
\bigcap\limits_{i=0}^{r+1} \Stab(g^{-1}R_i^\ld)
}
\cap
\Bp{
\bigcap\limits_{j=0}^{r+1} \Stab(R_j^\mu)
}
\\
&=\bigcap\limits_{(i,j)\in I^+} \Stab(g^{-1}R_i^\ld \cap R_j^\mu)
= \bigcap\limits_{(i,j)\in I^+} \Stab( g^{-1}(A\stdP)_{ij}  )
= \W_{\delta(A)}.
}
\]
The proposition is proved.
\endproof

%=========================================================
\section{Computation in affine Schur algebra $\Sjj$}
\label{sec:H+S}

We denote the (type A) quantum $v$-number and quantum $v$-factorial  by, for $m\in \ZZ, n \in \NN$,
\begin{align}
  \label{factorial}
[m] = & \frac{v^{2m}-1}{v^2-1},
\qquad
\lr{n}^! =  \lr{n}\lr{n-1} \cdots \lr{1},
\\
\qbinom{m}{n}  &= \frac{[m][m-1]\cdots [m-n+1]}{[n]^!}.
\end{align}
(It is understood that $\lr{0}^! = 1$.)
For $T= (t_{ij}) \in \Tt_n$, %recalling the type-$\fa$ quantum integers/factorials \eqref{factorial}
we define
\eq  \label{Tfactorial}
[T]^! = \prod\limits_{i=1}^n \prod\limits_{j\in\ZZ} [t_{ij}]^!.
\endeq

%For $n \in \ZZ$, we denote the (type C/B) quantum $v$-number  by
%\begin{align}
%  \label{factorial-c}
%\lr{n}_\fc = \frac{v^{4n}-1}{v^2-1} =\lr{2n}.
%\end{align}
For $A= (a_{ij}) \in \Xi_n$, define (see \eqref{def:a'} for $a'_{ij}$)
\eq
[A]^!_\fc =  [a'_{00}]^!_\fc \, [a'_{r+1,r+1}]^!_\fc \cdot \prod\limits_{(i,j)\in \Ia} [a_{ij}]^!,
\endeq
where
$
[m]^!_\fc = \prod\limits_{k=1}^{ m } \lr{2k}.
$
In particular, we have
$[a'_{ii}]^!_\fc = [2] \cdot [4] \cdot \ldots \cdot [a_{ii}-1]$ for $i = 0, r+1$.
%Alternatively, we have (see \eqref{def:Ia} for $\Ia$):
%\[
%[A]^!_\fc = [a_{00}]^!_\fc[a_{r+1,r+1}]^!_\fc \prod_{(i,j)\in \Ia} [a_{ij}]^!.
%\]
%-----------------------------------------------------------------------------------------------------------
\lem  \label{lem:Ac}
For any $A \in \Xi_n$, we have
$[A]^!_\fc = \sum_{w \in W_{\delta(A)}} v^{2\ell(w)}$.
\endlem
%-----------------------------------------------------------------------------------------------------------
\proof
Denote the Weyl group of type $A_{m-1}$ (respectively,  $C_{m}$) by $S_m$ (respectively,  $W_{C_m}$).
It is well known that the Poincare polynomial for $S_m$ and $W_{C_m}$ are, respectively,
\[
\sum_{w\in S_m}v^{2\ell(w)}=\prod_{k=1}^m[k]=[m]^!
\quad
\tand
\quad
\sum_{w\in W_{C_m}}v^{2\ell(w)}=\prod_{k=1}^m[2k]=[m]^!_\fc.
\]
Since
$
W_{\delta(A)}\simeq
W_{C_{\delta_0}}\times
S_{\delta_1}\times S_{\delta_2}\times\cdots\times S_{\delta_{\mathfrak{r}}}
\times W_{C_{\delta_{\mathfrak{r}+1}}},
$
we obtain
\[
\sum_{w \in W_{\delta(A)}} v^{2\ell(w)}
=
\prod_{i\in\{0,\mathfrak{r}+1\}}
\Bp{\sum_{w\in W_{C_{\delta_i}}} v^{2\ell(w)}}
\prod_{i=1}^{\mathfrak{r}}
\Bp{\sum_{w\in S_{\delta_i}}v^{2\ell(w)}}
=[A]^!_\fc.
\]
The lemma is proved.
\endproof

%-----------------------------------------------------------------------------------------------------------
\lem  \label{lem:xTx}
Let $A=\kappa(\ld,g,\mu)$ for   $\ld,\mu \in \Ld, g\in \D_{\ld\mu}$.  Then
$x_\ld T_{g} x_\mu = [A]^!_\fc \, e_A(x_\mu).$
\endlem

\proof
Let $\delta = \delta(A)$. By Proposition \ref{prop:doublecoset}(c), we have
\[
x_\mu
= \sum_{x \in \W_\mu} T_x
= \sum_{\substack{w \in \D_\delta \cap W_\mu \\ y\in \W_\delta}} T_{wy}
= \sum_{w \in \D_\delta \cap W_\mu} T_w \sum_{y\in \W_\delta} T_{y}
= T_{\D_\delta \cap \W_\mu} x_\delta.
\]
Note that $x_\mu T_w = v^{2\ell(w)} x_\mu$ for any $w \in \W_\mu$, and thus it follows  by Lemma~\ref{lem:Ac}
and $\W_\delta \subset \W_\mu$ that
$x_\mu x_\delta = \sum\limits_{w \in W_{\delta}} v^{2\ell(w)} x_\mu = [A]^!_\fc x_\mu$.
Therefore by \eqref{phi} and \eqref{def:eA} we have
\[
x_\ld T_{g} x_\mu
= x_\ld T_{g} T_{\D_\delta \cap \W_\mu} x_\delta
= e_A(x_\mu) x_\delta
= e_A(x_\mu x_\delta)
= [A]^!_\fc e_A(x_\mu).
\]
The lemma is proved.
\endproof

%%%%%%%%%
%from Chapter 3
%%%%%%%%%

%---------------\section{Structure constants --move to Chapter 2 ?}
\lem\label{lem:mult1}
Let $B = \kappa(\ld,g_1,\mu)$ and $A = \kappa(\mu,g_2,\nu)$, where
$\ld,\mu, \nu \in \Ld$, $g_1 \in \D_{\ld\mu}$, and $g_2 \in \D_{\mu\nu}$.
Let  $\delta = \delta(B)$; see \eqref{eq:delta}. Then we have
\[
e_B   e_A(x_\nu) = \frac{1}{[A]^!_\fc} x_\ld T_{g_1} T_{(\D_{\delta} \cap W_\mu)g_2} x_\nu.
\]
\endlem
%-----------------------------------------------------------------------------------------------------------
\proof
It follows by Lemma \ref{lem:xTx},  \eqref{phi} and \eqref{def:eA} that
\[
e_B e_A(x_\nu)
= \frac{1}{[A]^!_\fc} e_B (x_\mu T_{g_2} x_\nu)
= \frac{1}{[A]^!_\fc} e_B (x_\mu) T_{g_2} x_\nu
= \frac{1}{[A]^!_\fc} x_\ld T_{g_1} T_{\D_{\delta} \cap W_\mu} T_{g_2} x_\nu.
\]
Since $g_2 \in \D_{\mu}^{-1}$, we have $ T_w T_{g_2} = T_{wg_2}$ for all $w \in \D_{\delta} \cap W_\mu$. Therefore
$T_{\D_{\delta} \cap W_\mu} T_{g_2}
= T_{(\D_{\delta} \cap W_\mu)g_2}$ and we are done.
\endproof
%-----------------------------------------------------------------------------------------------------------
%\rmk %\label{rmk:nontriv}
For $w\in W_\mu$, although $T_{g_1} T_w = T_{g_1 w}$ and $T_{w} T_{g_2} = T_{w g_2}$,
it is not true that $T_{g_1} T_w T_{g_2} = T_{g_1 w g_2}$ in general.
Therefore we need to write out $T_{g_1} T_{w g_2}$ in order to have a useful multiplication formula.
%\endrmk
For $w \in \D_\delta \cap \W_\mu$, we write
\eq\label{eq:Delta(w)}
T_{g_1} T_{wg_2} = \sum_{\sig \in \Delta(w)} c^{(w,\sig)} T_{ g_1 \sig  w g_2}, \quad
\text{ for  } c^{(w,\sig)} \in \ZZ[v^2] \text{ and finite subset } \Delta(w) {\subset} \W.
\endeq
For $\sig\in\Delta(w)$,
%denote the  minimal length representative in the double coset $W_\ld ( g_1 \sig w g_2) W_\nu$ by $y^{(w,\sig)} \in \D_{\ld\nu}$.
we have
\eq  \label{eq:T12}
T_{g_1\sig w g_2} = T_{w^{(\sig)}_\ld}  T_{y^{(w,\sig)}}  T_{w^{(\sig)}_\nu}
\endeq
if we write
\eq \label{g1wg2}
g_1\sig w g_2 = w^{(\sig)}_\ld y^{(w,\sig)} w^{(\sig)}_\nu,
\quad \text{ for some }
y^{(w,\sig)} \in \D_{\ld\nu}, \;
w^{(\sig)}_\ld \in W_\ld, \, w^{(\sig)}_\nu \in W_\nu.
\endeq
We further denote
\eq  \label{Aw}
A^{(w,\sig)} = (a^{(w,\sig)}_{ij}) = \kappa(\ld,y^{(w,\sig)},\nu).
\endeq

%-----------------------------------------------------------------------------------------------------------
\prop\label{prop:mult1}
Let $\delta = \delta(B)$; see \eqref{eq:delta}. Let $c^{(w,\sig)}, \Delta(w)$ and $A^{(w,\sig)}$
be defined as in \eqref{eq:Delta(w)} and \eqref{Aw}. Then we have
\eq\label{eq:mult1}
 e_B  e_A
=  \sum_{\substack{ w\in \D_\delta \cap \W_\mu \\ \sig\in\Delta(w) }}
 c^{(w,\sig)} v^{2\ell(g_1\sig w g_2) - 2\ell(y^{(w,\sig)})}  \frac{[A^{(w,\sig)}]^!_\fc}{[A]^!_\fc} e_{A^{(w,\sig)}}.
\endeq
\endprop
%-----------------------------------------------------------------------------------------------------------
\proof
Combining Lemma \ref{lem:mult1} and \eqref{eq:Delta(w)}, we have
\[
 e_B   e_A(x_\nu)
=  \frac{1}{[A]^!_\fc} \sum_{\substack{ w\in \D_\delta \cap \W_\mu \\ \sig\in\Delta(w) }} c^{(w,\sig)} x_\ld T_{g_1 \sig w g_2} x_\nu.
\]
For $\sig \in \Delta(w)$,  by \eqref{eq:T12} we have
$x_\ld T_{g_1 \sig w g_2 }x_\nu
= v^{2\ell(w^{(\sig)}_\ld)+ 2\ell(w^{(\sig)}_\nu)} x_\ld T_{y^{(w,\sig)}}  x_\nu$.
The proposition now follows by applying Lemma \ref{lem:xTx}.
\endproof

It is unrealistic to obtain an explicit description for $c^{(w,\sig)}$ in \eqref{eq:Delta(w)} for general $g_1$
since it amounts to obtaining explicitly all the structure constants for affine Hecke algebras.
%Also, it is not clear for which pairs $(w,\sig)$ the matrices $A^{(w,\sig)}$ are the same.
Later on we shall treat the special case when $B$ is tridiagonal, when the structure constants
will be computed explicitly.

%=========================================================
\section{Isomorphism $\Sjjg \cong  \Sjj$}
\label{sec:SjC}

In this section we show that the Schur algebra $\Sjj$ defined in \S\ref{sec:Sj}
can be identified with the affine Schur algebra in \cite{FLLLW} which was defined as a convolution algebra in a geometric setting.
We shall add superscript ``$\text{geo}$" (a shorthand for ``geometric") to the notations used {\em loc. cit.}.

Let $F = \mathbb F_{q^2} ((\epsilon))$ be the field of formal Laurent series over a finite field $\mathbb F_{q^2}$ of $q^2$ elements,
and let $\SP_F(2d)$ be the symplectic group with coefficients in $F$.
Set $\Sjjg $ to be the Schur algebra in \cite[\S4.2]{FLLLW} (denoted by $\bS_{n,d;\mA}$ therein) which specializes at $v=q$ as
\eq
\Sjjg|_{v=q} = {\mA}_{\SP_F(2d)} (\Xjj \times \Xjj),
\endeq
the convolution algebra on the set
$\Xjj$  of $n$-step periodic lattices of affine type $C$ in an $F$-vector space $V$ of dimension $2d$.
Recall that $\Xi_{n,d}$ parameterizes $\SP_F(2d)$-orbits of $\Xjj \times \Xjj$.
For any $A\in \Xi_{n,d}$, denote by $\egeo_A$ the characteristic function on the orbit $\cO_A$.
The Hecke algebra $\bH$  which specializes at $v=q$ as a convolution algebra
\eq
\bH|_{q} = {\mA}_{\SP_F(2d)} (\Yjj \times \Yjj),
\endeq
where $\Yjj$ is the set of `complete' periodic lattices of affine type $C$ in $V$.
%-----------------------------------------------------------------------------------------------------------
\begin{prop}\label{Schur-iso}
There is an algebra isomorphism $\psi: \Sjjg \stackrel{\simeq}{\longrightarrow}  \Sjj$, which sends
$e_A^\geo$ to $e_A$ for each $A \in \Xi_{n,d}$.
\end{prop}
%-----------------------------------------------------------------------------------------------------------
\proof
We only need to prove the statement when $v$ is specialized to $q$ for various prime powers $q$.
Let $\psi: \Sjjg \rightarrow \Sjj$ be the $\mA$-linear map sending $e_A^{\rm geo}$ to $e_A$ for all $A\in \Xi_{n,d}$.
Since $\Xi_{n,d}$ parameterizes the basis of both algebras, $\psi$ is an isomorphism (see also \cite{Vi03}).

It remains to show that $\psi$ is an algebra homomorphism.
Fix $A,B,C \in \Xi_{d}$, and let $\ld, \mu, \nu \in \Ld$ and $g_1, g_2, g_3 \in W$ be such that
\[
A = \kappa(\ld, g_1, \mu),
\quad
B = \kappa(\mu, g_2, \nu),
\quad
C = \kappa(\ld, g_3, \nu).
\]
Set $g^C_{AB}(v) \in \mA$ to be such that
\eq
g^C_{AB} (q)
=\left | \left\{ \~{L} \in \Xjj ~\big |~ (L, \~{L}) \in \cO_A, (\~{L}, L') \in \cO_B, (L, L') \in \cO_C \right\} \right |,
\endeq
for some fixed $L, L' \in \Xjj$.
Here $\cO_A$ is the $\mrm{Sp}_F(2d)$-orbit in $\Xjj\times \Xjj$ indexed by $A$.
Therefore, by definition we have
\eq
\egeo_A   \egeo_B = \sum_{C} g^C_{AB}(v)~ \egeo_C.
\endeq
For $x,y,z\in W$,
set $B^z_{xy}(v) \in \mA$ to be such that
\eq  \label{eq:TTy}
T_x T_y = \sum_z B^z_{xy}(v) T_z.
\endeq
For $g \in W$,  let $\cO_g$ be the orbit $\cO_{\kappa(\omega, g, \omega)}$ where $\omega = (0,1,1,\ldots,1,0) \in \Ld_{d,d}$.
It is well known that
\[
B^z_{xy}(q)
=\left | \left\{ \~{L} \in \Yjj ~\big |~ (L, \~{L}) \in \cO_x, (\~{L}, L') \in \cO_y, (L, L') \in \cO_z \right\} \right |,
\]
for some fixed $L, L' \in \Yjj$.
%By slightly abuse of notation, for $x\in W$ and $A = \kappa(\ld, g, \mu)$,  we write $x \in A$ if and only if
%\eq
%|R^\ld_i \cap x R^\mu_j| = a_{ij}
%\quad
%\textup{for all}
%\quad
%i,j.
%\endeq
Then (as in \cite[Proposition~3.4]{Du92})   for any $z\in W_\ld g_3 W_\nu$ we have
\eq  \label{eq:g=B}
g^C_{AB}(q) = \pi_\mu(q)^{-1} \sum_{\substack{x\in W_\ld g_1 W_\mu \\ y \in W_\mu g_2 W_\nu}} B^z_{xy}(q),
\endeq
where %$\pi_\mu(\cdot )$ is given by
$%\[
\pi_\mu(v) = \sum_{x \in W_\mu} v^{2\ell(x)}.
$%\]

On the other hand, it is well known that (in particular it follows as a special case of Lemma~\ref{lem:xTx} and its proof
with $A=\kappa(\mu,1,\mu)$ and $W_\delta =W_\mu$)
\eq  \label{eq:x2}
x_\mu^2 = \pi_\mu(v) x_\mu.
\endeq
Therefore by \eqref{phi}, \eqref{def:eA} and \eqref{eq:x2} we have
\begin{align*}
e_A   e_B(x_\nu)
&= e_A(x_\mu T_{g_2} T_{\D_{\delta(B)} \cap W_\nu})
 = e_A(x_\mu) T_{g_2} T_{\D_{\delta(B)} \cap W_\nu}
\\
&= \pi_\mu(v)^{-1} e_A(x_\mu) x_\mu T_{g_2} T_{\D_{\delta(B)} \cap W_\nu}
= \pi_\mu(v)^{-1} T_{W_\ld g_1 W_\mu} T_{W_\mu g_2 W_\nu},
\end{align*}
which, by \eqref{eq:TTy} and \eqref{eq:g=B}, can be rewritten as
\[
e_A   e_B(x_\nu)
= \pi_\mu(v)^{-1} \sum_{z\in W_\ld g_3 W_\nu}  \sum_{\substack{x\in W_\ld g_1 W_\mu \\ y \in W_\mu g_2 W_\nu}} B^z_{xy}(v) T_z
=\sum_{z\in W_\ld g_3 W_\nu} g^C_{AB}(v) T_z.
\]
This implies that
$e_A   e_B  = \sum_{C} g^C_{AB}(v) ~e_C$, and hence $\psi$ is an algebra homomorphism.
\endproof

%=========================================================
\chapter{Multiplication formula for affine Hecke algebra}
 \label{sec:MF-Hecke}

In this chapter, we establish a multiplication formula in affine Hecke algebra $\bH$ with an element of the form
$T_g$, where $g \in \D_{\ld\mu}$ for some $\ld,\mu \in \Ld$ and $\kappa(\ld,g,\mu)$ is tridiagonal.
This formula is used to obtain a corresponding multiplication formula
for affine Schur algebra $\Sjj$ in Chapter~\ref{sec:MF-Schur}.

%----------------------------------------------------------------------------------------------------------
\section{Minimal length representatives}
%----------------------------------------------------------------------------------------------------------

Fix $\ld,\mu \in \Ld$, $g_1 \in \D_{\ld\mu}$. %, and $g_2 \in \D_{\mu\nu}$.
Let $B = \kappa(\ld,g_1,\mu)$. % and $A = \kappa(\mu,g_2,\nu)$.
We assume that $B =(b_{ij}) = \kappa(\ld, g_1, \mu)$ is tridiagonal, i.e., $b_{ij}=0$ unless $|i-j|\le 1$.
We choose  \eqref{eq:delta} to be
\eq  \label{eq:delta3}
\delta =\delta(B)
= (b'_{00}, b_{10}; \; b_{01}, b_{11}, b_{21};  \ldots,  %b_{i-1,i}, b_{i,i}, b_{i+1,i}; \ldots
 b_{r-1,r}, b_{r,r}, b_{r+1,r}; \;
  b_{r,r+1}, b'_{r+1,r+1}) \in \Ld_{3r+2, d}. % \subset \NN^{3r+4}.
\endeq
Recalling the convention of indexing in \eqref{def:Ld}, we understand that $\delta$ has $3r+4$ components indexed by $0,1, \ldots, 3r+2, 3r+3$.
%That is,  unlike $\delta(B)$ defined in \eqref{eq:delta}, here the intermediate terms $\delta_1$,
%$\ldots$, $\delta_{3r+2}$, $\delta'_1$, $\ldots$, $\delta'_{3r+2}$ can be zero.
%Note that the two conventions coincide (i.e., $\delta = \delta(B)$) when the intermediate terms are all nonzero.

%-----------------------------------------------------------------------------------------------------------
\lem\label{lem:ld-mu-delta}
We have, for $i \in \ZZ$, %1\le i \le r$,
\[
R_i^\mu =  R^{\delta}_{3i-1} \cup R^{\delta}_{3i} \cup R^{\delta}_{3i+1}, \quad
g_{1}^{-1}R_i^\ld =  R^{\delta}_{3i-2} \cup R^{\delta}_{3i} \cup R^{\delta}_{3i+2}.
\]
\endlem
%-----------------------------------------------------------------------------------------------------------
\proof
Follows by the construction \eqref{eq:Ri}--\eqref{eq:Ri2} and the definition of $\delta$ in \eqref{eq:delta3}.
\endproof

Again, let $B = \kappa(\ld, g_1, \mu)$ be tridiagonal. By Lemma~\ref{lem:kappa} we have $g_1 = \gstd_{B\stdP}$.
%In this special case, $g_1$ is the permutation ``swapping'' $R^{\delta}_{3i-2}$ and $R^{\delta}_{3i-1}$ for any $i$, and hence it
Unraveling the definition \eqref{stW} for $\gstd_{B\stdP}$, we
can write $g_1=\prod_{i=1}^{r+1} g_1^{(i)}$, where $g^{(i)}_1\in W$ is specified by the following recipe:
\eq\label{g1}
g^{(i)}_1(x) =
\bc{
 x + b_{i-1,i} &\tif x\in R^{\delta}_{3i-2} \subset R_{i-1}^\mu,
\\
 x-b_{i,i-1} &\tif x\in R^{\delta}_{3i-1} \subset R_{i}^\mu,
\\
x &\tif x\in [1..d] \big \backslash (R^{\delta}_{3i-2} \cup R^{\delta}_{3i-1}).
}
\endeq
%-------------------------------------------------------------------------------------------------------------
\lem\label{lem:g1'}
Fix $1\le i \le  r+1$. Write $R_{3i-2}^\delta = [m+1 .. m+\af]$ and $R_{3i-1}^\delta = [m+\af +1 .. m+\af+\beta]$
for some $m,\af,\beta\in \NN$.
Then $g_1^{(i)}$ has a reduced expression
\eq\label{g1i}
g_1^{(i)}
=
(s_{m+\beta}\cdots s_{m+2}s_{m+1})
(s_{m+\beta+1}\cdots s_{m+2})
\cdots
(s_{m+\beta+\alpha-1}\cdots s_{m+\alpha}).
\endeq
\endlem
%-----------------------------------------------------------------------------------------------------------
\proof
We note that $s_i (\forall i\in (d+1)\ZZ)$ does not appear in \eqref{g1i} since all elements in $(d+1)\ZZ$ are on the main diagonal of $B\stdP$ (and hence in $R^{\delta}_{3j}$ for some $j\in\ZZ$). For $1 \leq t \leq \alpha$, the product $s_{m+\beta+t-1}\cdots s_{m+t}$ is the cyclic permutation on $[m+t .. m+t+\beta]$, which sends
$
%m+\beta +t \mapsto m+t,
m+t \mapsto m+t+1 \mapsto \ldots \mapsto
%m+\beta +t -1 \mapsto
m+\beta +t  \mapsto m+t.
%\quad \ldots, \quad m+t \mapsto m+t+1.
$
The lemma then follows from \eqref{g1}.
\endproof
%-----------------------------------------------------------------------------------------------------------
For any $1\le i \le r+1$, $w\in\D_\delta \cap \W_\mu$ and $g_2\in\D_{\mu\nu}$, recalling \eqref{eq:jk}
we introduce the following subset of $W$:
\begin{align}  \label{def:Kw(i)}
K_w^{(i)} &=\Big \{
\text{products of disjoint transpositions of the form } (j,k)_\fc~\Big |~
\notag
\\
& \qquad \qquad \qquad \quad  j \in R_{3i-2}^\delta, k \in R_{3i-1}^\delta, (wg_2)^{-1}(k) < (wg_2)^{-1}(j)
\Big \}.
\end{align}
We then define
\eq\label{def:Kw}
K_w:=\Big \{
\prod_{i=1}^{r+1} \sig^{(i)}
~\big|~
\sig^{(i)}\in K^{(i)}_w
\Big \}.
\endeq
%Throughout this paper, when we take an element  $(j_1^{(i)},k_1^{(i)})_\fc(j_2^{(i)},k_2^{(i)})_\fc\cdots(j_s^{(i)},k_s^{(i)})_\fc\in K^{(i)}_w\subset W$, we always means that this expression is a product of disjoint transpositions for Weyl group $W$. Moreover, we also always assume that $j_1^{(i)}<j_2^{(i)}<\cdots<j_s^{(i)}$ so that the above expression is uniquely determined.
%-----------------------------------------------------------------------------------------------------------
Clearly we have $\sig^{2}=1$ for any $\sig\in K_w$.
For $w\in \D_\delta \cap W_\mu$ and $\sig \in K_w$ we denote
\eq\label{def:nsig}
n(\sig)=\mbox{the number of disjoint transpositions appeared in } \sig.
\endeq
%For $w\in\D_\delta \cap \W_\mu$ and $\sigma\in K_w$, w
We also set
\eq  \label{eq:hw}
h(w,\sig)= \big|H(w,\sig) \big|,
\endeq
where
\eq\label{def:Q(wsig)1}
H(w,\sig)
=
\bigcup_{i=1}^{r+1}
\left\{ (j,k)\in R_{3i-2}^\delta\times R_{3i-1}^\delta
~\middle|~
\ba{{c}
(wg_2)^{-1}\sig(j)>(wg_2)^{-1}(k),
\\
(wg_2)^{-1}(j)>(wg_2)^{-1}\sig(k)
}
\right\}.
\endeq

%%%%%%%%%%%%%%%%%%%%%
\section{Multiplication formula for affine Hecke algebra}

The goal of this section is to establish
the following formula for structure constants of the affine Hecke algebra $\bH$.
The formula will be applied in Chapter~\ref{sec:MF-Schur} to compute structure constants for affine Schur algebra $\Sjj$.
Recall $\D_\ld$ from \eqref{eq:Dmin}, $\D_{\ld\mu}$ from \eqref{eq:Dmin2},
$\kappa$ from \eqref{eq:ka}, $\delta(B)$ from \eqref{eq:delta3},
$K_w$ from \eqref{def:Kw}, $n(\sigma)$ from \eqref{def:nsig}, and $h(w,\sig)$ from \eqref{eq:hw}.

\thm
  \label{thm:heckemult}
Let $B = \kappa(\ld, g_1, \mu)$ be a tridiagonal matrix.
Then, for any $g_2 \in \D_{\mu \nu}$ and $w \in \D_{\delta(B)} \cap W_{\mu}$, we have
\eq\label{heckemult}
T_{g_1}T_{wg_2}=\sum_{\sig\in K_w}
(v^2-1)^{n(\sig)}v^{2h(w,\sig)}T_{g_1\sig wg_2}.
\endeq
In particular, $\Delta(w)$ and $c^{w,\sig}$ from \eqref{eq:Delta(w)} are $K_w$ and $(v^2-1)^{n(\sig)}v^{2h(w,\sig)}$, respectively.
\endthm

\proof
Recall $g_1=\prod_{i=1}^{r+1} g_1^{(i)}$; see \eqref{g1}.
It suffices to  show that
\eq  \label{heckemult(i)}
T_{g_1^{(i)}}T_{wg_2}=\sum_{\sig\in K_w^{(i)}}
(v^2-1)^{n(\sig)}v^{2h(w,\sig)}T_{g_1^{(i)}\sig wg_2}
\endeq
for any $1\le i \le r+1$. We shall continue to use the notations in Lemma \ref{lem:g1'} in this proof.

By Lemma \ref{lem:g1'}, we have
\[T_{g_1^{(i)}}T_{wg_2}=(T_{s_{m+\beta}}\cdots T_{s_{m+2}}T_{s_{m+1}})
(T_{s_{m+\beta+1}}\cdots T_{s_{m+2}})\cdots (T_{s_{m+\beta+\alpha-1}}\cdots T_{s_{m+\alpha}})T_{wg_2}.
\]
Let us denote
\[
\{k\in R^\delta_{3i-1}|(wg_2)^{-1}(m+\alpha)>(wg_2)^{-1}(k)\}=\{m+\alpha+1,m+\alpha+2,\ldots,m+\alpha+p\}.
\]
We claim that
\begin{align}  \label{heckemult1}
(& T_{s_{m+\beta+\alpha-1}}   \cdots T_{s_{m+\alpha}})T_{wg_2}
\notag
\\
&=v^{2p}T_{s_{m+\beta+\alpha-1}\cdots s_{m+\alpha}wg_2}+\sum_{x=1}^p(v^2-1)v^{2(x-1)}T_{s_{m+\beta+\alpha-1}\cdots s_{m+\alpha}(m+\alpha,m+\alpha+x)_\fc wg_2}.
\end{align}

We prove Equation \eqref{heckemult1} recursively.
Firstly,
\[
T_{s_{m+\alpha}}T_{wg_2}=
\left\{
\begin{array}{ll}
T_{s_{m+\alpha}wg_2},& \mbox{if $(wg_2)^{-1}(m+\alpha)<(wg_2)^{-1}(m+\alpha+1)$};\\
v^2T_{s_{m+\alpha}wg_2}+(v^2-1)T_{wg_2},& \mbox{if $(wg_2)^{-1}(m+\alpha+1)<(wg_2)^{-1}(m+\alpha)$}.
\end{array}
\right.
\]
Suppose that we have shown, for $c<p$,
\begin{align*}
(& T_{s_{m+\alpha+c-1}}\cdots T_{s_{m+\alpha}})T_{wg_2}
\\
&=v^{2c}T_{s_{m+\alpha+c-1}\cdots s_{m+\alpha}wg_2}
+\sum_{x=1}^c(v^2-1)v^{2(x-1)}T_{s_{m+\alpha+c-1}\cdots s_{m+\alpha}(m+\alpha,m+\alpha+x)_\fc wg_2}.
\end{align*}
Note that
\[
T_{s_{m+\alpha+c}}T_{s_{m+\alpha+c-1}\cdots s_{m+\alpha}wg_2}
=v^2T_{s_{m+\alpha+c}\cdots s_{m+\alpha}wg_2}+(v^2-1)
 T_{s_{m+\alpha+c}\cdots s_{m+\alpha}(m+\alpha,m+\alpha+c+1)_\fc wg_2}
\]
thanks to
\[(s_{m+\alpha+c-1}\cdots s_{m+\alpha}wg_2)^{-1}(m+\alpha+c)>(s_{m+\alpha+c-1}\cdots s_{m+\alpha}wg_2)^{-1}(m+\alpha+c+1),
\]
while
\[
T_{s_{m+\alpha+c}}T_{s_{m+\alpha+c-1}\cdots s_{m+\alpha}(m+\alpha,m+\alpha+x)_\fc wg_2}=T_{s_{m+\alpha+c}\cdots s_{m+\alpha}(m+\alpha,m+\alpha+x)_\fc wg_2}
\]
for any $1\le x \le c$, thanks to
\begin{align*}
(s_{m+\alpha+c-1} & \cdots s_{m+\alpha}(m+\alpha,m+\alpha+x)_\fc wg_2)^{-1}(m+\alpha+c)\\
&<(s_{m+\alpha+c-1}\cdots s_{m+\alpha}(m+\alpha,m+\alpha+x)_\fc wg_2)^{-1}(m+\alpha+c+1).
\end{align*}
Therefore we obtain
\begin{align*}
( &T_{s_{m+\alpha+c}}\cdots T_{s_{m+\alpha}})T_{wg_2}
\\
&=v^{2(c+1)}T_{s_{m+\alpha+c}\cdots s_{m+\alpha}wg_2}
+\sum_{x=1}^{c+1}(v^2-1)v^{2(x-1)}T_{s_{m+\alpha+c}\cdots s_{m+\alpha}(m+\alpha,m+\alpha+x)_\fc wg_2}.
\end{align*}
Suppose that we have obtained, for $c\geq p$,
\begin{align*}
( & T_{s_{m+\alpha+c-1}}\cdots T_{s_{m+\alpha}})T_{wg_2}
\\
&=v^{2p}T_{s_{m+\alpha+c-1}\cdots s_{m+\alpha}wg_2}
+\sum_{x=1}^p(v^2-1)v^{2(x-1)}T_{s_{m+\alpha+c-1}\cdots s_{m+\alpha}(m+\alpha,m+\alpha+x)_\fc wg_2}.
\end{align*}
Since
\[(s_{m+\alpha+c-1}\cdots s_{m+\alpha}wg_2)^{-1}(m+\alpha)<(s_{m+\alpha+c-1}\cdots s_{m+\alpha}wg_2)^{-1}(m+\alpha+c+1)
\] and
\begin{align*}
 (s_{m+\alpha+c-1}& \cdots s_{m+\alpha}(m+\alpha,m+\alpha+x)_\fc)^{-1}(m+\alpha)
\\
&<(s_{m+\alpha+c-1}\cdots s_{m+\alpha}(m+\alpha,m+\alpha+x)_\fc)^{-1}(m+\alpha+c+1),
\end{align*}
we have
\[
(T_{s_{m+\alpha+c}}\cdots T_{s_{m+\alpha}})T_{wg_2}=v^{2p}T_{s_{m+\alpha+c}\cdots s_{m+\alpha}wg_2}+\sum_{x=1}^p(v^2-1)v^{2(x-1)}T_{s_{m+\alpha+c}\cdots s_{m+\alpha}(m+\alpha,m+\alpha+x)_\fc wg_2}.
\]
Thus \eqref{heckemult1} holds.

The power of $v^2$ for each term on the right hand side of \eqref{heckemult1} is the number of elements $k\in R^\delta_{3i-1}$ satisfying $(wg_2)^{-1}\eta(m+\alpha)>(wg_2)^{-1}(k)$ where $\eta=\id$ or $(m+\alpha,m+\alpha+x)_\fc$. The power of $v^2-1$ for each term on the right hand side of \eqref{heckemult1} is the number of disjoint transpositions for $\eta$. In other words, \eqref{heckemult1} can be rewritten as follows
\eq\label{heckemult2}
(T_{s_{m+\beta+\alpha-1}}\cdots T_{s_{m+\alpha}})T_{wg_2}=\sum_{\substack{\eta=\id,(m+\alpha,m+\alpha+x)_\fc\\x=1,2,\ldots,p}}(v^2-1)^{n(\eta)}v^{2h'(\eta)}T_{s_{m+\beta+\alpha-1}\cdots s_{m+\alpha}\eta wg_2}
\endeq where
$h'(\eta):= \Big |\{(m+\alpha,k)\in R_{3i-2}^\delta\times R_{3i-1}^\delta|
(wg_2)^{-1}\eta(m+\alpha)>(wg_2)^{-1}(k)\} \Big|$.

Replacing $wg_2$ by $\eta wg_2$, we can obtain a similar formula for $(T_{s_{m+\beta+\alpha-2}}\cdots T_{s_{m+\alpha-1}})T_{\eta wg_2}$. However, in case of $\eta=(m+\alpha,m+\alpha+x)_\fc$, we note that
\begin{eqnarray*}
&&\{k\in R^\delta_{3i-1}|(\eta wg_2)^{-1}(k)<(\eta wg_2)^{-1}(m+\alpha-1)\}\\
&=&\{k\in R^\delta_{3i-1}|(wg_2)^{-1}(k)<(wg_2)^{-1}(m+\alpha-1)\}\setminus\{m+\alpha+x\}\\
&=&\{k\in R^\delta_{3i-1}|(wg_2)^{-1}(k)<(wg_2)^{-1}\eta(m+\alpha-1),(wg_2)^{-1}\eta(k)<(wg_2)^{-1}(m+\alpha-1)\}.
\end{eqnarray*}
Hence
\[
\bA{
(T_{s_{m+\beta+\alpha-2}} & \cdots T_{s_{m+\alpha-1}})(T_{s_{m+\beta+\alpha-1}}\cdots T_{s_{m+\alpha}})T_{wg_2}\\
&=\sum_{\zeta}(v^2-1)^{n(\zeta)}v^{2h'(\zeta)}T_{s_{m+\beta+\alpha-2}\cdots s_{m+\alpha-1}s_{m+\beta+\alpha-1}\cdots s_{m+\alpha}\zeta wg_2},
}\]
where $\zeta$ runs over $\id,(m+\alpha-1,k_1)_\fc,(m+\alpha,k_2)_\fc,(m+\alpha-1,k_1)_\fc(m+\alpha,k_2)_\fc,(k_1\neq k_2)$ with ${wg_2}^{-1}(m+\alpha-1)>{wg_2}^{-1}(k_1)$ and ${wg_2}^{-1}(m+\alpha)>{wg_2}^{-1}(k_2)$, and
\[
h'(\zeta)=
\left|\left\{
(j,k)\in R_{3i-2}^\delta\times R_{3i-1}^\delta
\Big |
\begin{array}{c}
j=m+\alpha\mbox{ or }m+\alpha-1,\\
(wg_2)^{-1}\zeta(j)>(wg_2)^{-1}(k),(wg_2)^{-1}(j)>(wg_2)^{-1}\zeta(k)\end{array}
\right\}\right|.
\]
Repeating this procedure, we have proved the formula \eqref{heckemult(i)}. The theorem is proved.
\endproof

%------------------------------------------------------------------------------------------------
\rem\label{countqpower}
There is an intuitive explanation for $h(w,\sig)$ as follows. For each $j\in R^\delta_{3i-2}$, where $1\leq i\leq r+1$, count the number of those elements $k\in R^\delta_{3i-1}$ subject to Conditions (1)--(3) below:
(1) there is no transposition $(j',k)$, $(j<j'\in R^\delta_{3i-2})$, appearing in $\sig$;
(2) the entry of $A\stdP$ in which $w^{-1}(k)$ lying is to the left of entry in which $w^{-1}(j)$
lying (i.e. $(wg_2)^{-1}(j)>(wg_2)^{-1}(k)$);
(3) $k<k'$ if  $(j,k')$ is a transposition appearing in $\sig$.
The sum of these numbers for all $j\in R^\delta_{3i-2}$ and all $1\le i\le r+1$ is $h(w,\sig)$.
\endrem
%------------------------------------------------------------------------------------------------
%The following lemma gives another description for $n(\sig)$ and $h(w,\sig)$.
\cor
For any $g_1\in \D_{\ld \mu}$, $g_2 \in \D_{\mu\nu}$,  $w\in \D_\delta \cap W_\mu$ and $\sig \in K_w$, we have
\eq\label{eq:l(g_1etawg_2)}
\ell(g_1) + \ell(w) + \ell(g_2) = \ell(g_1 \sig w g_2) + n(\sig) + 2 h(w,\sig).
\endeq
\endcor
\proof
Follows from \eqref{heckemult}.% and the definition of multiplication on $\bH$.
\endproof

%%%%%%%%%%%%%
\section{An example}

\exa\label{exa:tri}
%(Example 3.6)
Let $r = 2, n = 6, d = 8$ and $D = 18$. Let $B = E^{00} + 2\sum\limits_{1\leq i,j \leq 2} E^{ij}_\tt + E^{33}$ and $A = E^{00} + \sum\limits_{i=1}^2\sum\limits_{j=1}^4 E^{ij}_\tt + E^{33}$. Namely,
\[
B =
\bM{[c:c:cccc]
\ddots&&&&&
\\
\hdashline
&1&&&&
\\
\hdashline
&&2&2&&
\\
&&2&2&&
\\
&&&&1&
\\
}
\quad
\textup{and}
\quad
A =
\bM{[c:c:cccc]
\ddots&&&&&
\\
\hdashline
&1&&&&
\\
\hdashline
&&1&1&1&1
\\
&&1&1&1&1
\\
&&&&1&
\\
},
\]
where dashed stripes denote the 0th column/row.
Here we use a two-by-four submatrix  for short when there is no ambiguity. That is,
\[
B = \young(2200,2200)~,
\quad
A = \young(1111,1111)~.
\]
Thus, $g_1 = \gstd \Bp{\bY{12&34&&\cr 56&78&&\cr}} = [1,2,5,6,3,4,7,8]_\fc$; see \eqref{def:perm}.
On the other hand, we have
\[
\ba{{c|ccccccccccccccc}
i & 0&1&2&3&4&5&6&7&8&9
\\
\hline
R^\delta_i
&\{0\}
&&&[1..2]&[3..4]&[5..6]&[7..8]&&&\{9\}
}
\]
That is, $g_1 = g_1^{(2)} = (s_4 s_3) (s_5 s_4)$ with $m =\af=\beta =2$.

Also, $g_2 = \gstd{\Bp{ \raisebox{-10pt}{\young(1234,5678)} }} = [1,5,2,6,3,7,4,8]_\fc$.

Now write $T_{ \young(abcd,efgh) } = T_x$ where $x = \gstd\Bp{ \raisebox{-10pt}{\young(abcd,efgh)} }$ for short. We have
\[
\bA{
T_4 T_{\young(1234,5678)}
&= v^2 T_{\young(1235,4678)} + (v^2-1)T_{\young(1234,5678)},
\\
T_5T_4 T_{\young(1234,5678)}
&= v^2 T_{\young(1236,4578)} + (v^2-1) \Bp{  v^2 T_{\young(1235,4678)}  + T_{\young(1234,6578)} }.
}
\]
If $w = \id$,  we have
%\[
%\ba{{c|cccccc}
%T_{s_{5}s_{4} \sig wg_2}&T_{\young(~~36,45~~)} & T_{\young(~~35,46~~)} & T_{\young(~~34,65~~)}
%\\
%\\
%\sig & \id &(4,6)& (4,5)
%\\
%\\
%\sig (4)&4 & 6 &5
%\\
%\\
%h'(\sig)&2&1&0
%\\
%\\
%wg_2 \leftrightarrow \young(~~34,56~~)&\young(~~~~,56~~)&\young(~~~~,5\times~~)&\young(~~~~,\times\times~~)
%}
%\]
\[
K_w^{(1)} = K_w^{(3)} = \{\id\},
\quad
K_w^{(2)} = K_w= \{\id, (3,5)_\fc, (3,6)_\fc, (4,5)_\fc, (4,6)_\fc, (3,5)_\fc(4,6)_\fc, (3,6)_\fc(4,5)_\fc\},
\]
and $T_{g_1} T_{w g_2} = \sum (v^2-1)^{n(\sig)} v^{2h(\sig)} T_{g_1\sig w g_2}$ with
\[
\ba{{ccccccccc}
T_{g_1 \sig w g_2}& \sig & n(\sig) & h(\sig)
\\
\hline
T_{\young(~~56,34~~)}&\id &0& 4
\\
T_{\young(~~46,35~~)}&(3,6)_\fc &1& 3
\\
T_{\young(~~36,54~~)}&(3,5)_\fc &1& 2
\\
T_{\young(~~54,36~~)}&(4,6)_\fc &1& 2
\\
T_{\young(~~34,56~~)}&(3,5)_\fc(4,6)_\fc &2&1
\\
T_{\young(~~53,64~~)}&(4,5)_\fc &1&1
\\
T_{\young(~~43,65~~)}&(3,6)_\fc(4,5)_\fc &2&0
\\

}
\]
\endexa
\vspace{.3cm}

%=========================================================
\chapter{Multiplication formula for affine Schur algebra}
  \label{sec:MF-Schur}

This chapter is devoted to the multiplication formula with tridiagonal generators. An essential idea in the proof is to identify the basis element $e_A$ with its corresponding ``higher-level'' matrix with entries being subsets of $\ZZ$ by \eqref{stP}.
We also provide two special cases of the multiplication formula that are analogous to the multiplication formulas with semisimple generators in affine type A, and with Chevalley generators as in finite type B/C.

%%%%%%%%%%%%%%
\section{A map $\varphi$}

Recall
$T_\tt$ from \eqref{Ttt}.
We introduce an {\em entry-wise partial order}, denoted by $\leq_e$, on $\Tt_n$ %(as well as on $\Xi_n$)
 by a matrix-entry-wise comparison:
\eq  \label{eq:leqe}
(a_{ij})\leq_e (b_{ij})\Leftrightarrow
a_{ij}\leq b_{ij} (\forall i,j).
\endeq
For $A, B \in \Xi_n$,
we set
\eq\label{def:ThetaBA}
\Tt_{B, A} = \big\{T\in\Tt_n ~|~T_\tt \leq_e A, \;  \roA(T)_i = b_{i-1,i} \textup{ for all } i \big\}.
\endeq

Fix $\ld,\mu, \nu \in \Ld$, $g_1 \in \D_{\ld\mu}$, and $g_2 \in \D_{\mu\nu}$.
Let $B = \kappa(\ld,g_1,\mu)$ and $A = \kappa(\mu,g_2,\nu)$.
Recall $\delta =\delta(B)$ from \eqref{eq:delta3}.
%Recall $e_A$ from \eqref{def:eA}.
For $w \in W$, we associate a matrix $\varphi(w)$ whose
$(i,j)$th entry of the matrix
$\varphi(w)$ is given by
\eq  \label{def:T}
\varphi(w)_{ij} = |R^\delta_{3i-1}\cap wg_2 R^\nu_j|.
\endeq

\lem\label{def:wtoT}
We have a surjective map $\varphi : \D_\delta \cap W_\mu \longrightarrow \Tt_{B,A}, \; w \mapsto \varphi(w)$.
\endlem

\proof
We first note that $\varphi(w) \in \Tt_n$.
For any $w \in \D_\delta \cap W_\mu$ and any $i$, we have
\[
\roA(\varphi(w))_i=|R^\delta_{3i-1}|=b_{i-1,i}.
\]
Next we check that $\varphi(w) \leq_e A$. Indeed, by Lemma \ref{lem:ld-mu-delta}, the $(i,j)$-th entry of $\varphi(w)_{\theta}$
\[
\bA{
\varphi(w)_{\theta,ij} &=|R^\delta_{3i-1}\cap wg_2 R^\nu_j|+|R^\delta_{3(-i)-1}\cap wg_2 R^\nu_{-j}|
\\
&=|(R^\delta_{3i-1}\cup R^\delta_{3i+1})\cap wg_2 R^\nu_{j}|
\\
&\leq|R^\mu_{i}\cap wg_2 R^\nu_j|
\\
&=|w^{-1}R^\mu_{i}\cap g_2 R^\nu_j|=|R^\mu_{i}\cap g_2 R^\nu_j|=a_{ij}.
}
\]
Hence we have proved $\varphi (w) \in \Tt_{B,A}$.

To show that $\varphi$ is surjective, for any $T =(t_{ij}) \in\Tt_{B, A}$ we shall construct an element $w_{A,T}\in \varphi^{-1}(T)$ as follows.
Recall $A\stdP$ from \eqref{stP}. For all $i,j$, we set
\begin{eqnarray*}
\mathcal{T}_{ij}^- &=& \textup{subset of } (A\stdP)_{ij}\textup{ consisting of the smallest }t_{ij}\textup{ elements},
\\
\mathcal{T}_{ij}^+ &=& \textup{subset of }(A\stdP)_{ij}\textup{ consisting of the largest }t_{-i,-j}\textup{ elements},
\\
\mathcal{T}_{ij}^0 &=& (A\stdP)_{ij} - \mathcal{T}_{ij}^+  - \mathcal{T}_{ij}^-.
\end{eqnarray*}
Note that $\sum_{j\in\ZZ} |\mathcal{T}_{ij}^-| = \roA(T)_i = |R^\delta_{3i-1}|$,  $\sum_{j\in\ZZ} |\mathcal{T}_{ij}^+| = \roA(T)_{-i} = |R^\delta_{3i+1}|$.
There is a unique element
$w_{A,T}= \prod\limits_{i=0}^{r+1} w_{A,T}^{(i)} \in\D_\delta \cap \W_\mu$ where $w_{A,T}^{(i)}\in \textup{Perm}(R_i^\mu)$
%(recall that $R^\mu_{i}=R^\delta_{3i-1}\cup R^\delta_{3i}\cup R^\delta_{3i+1}$)
is determined by
\eq\label{def:wAT}
w_{A,T}^{(i)}(x)\in \left\{
\begin{array}{ll}
R^\delta_{3i-1},& \mbox{if $x\in\bigcup_{j}\mathcal{T}_{ij}^-$}\\
R^\delta_{3i+1},& \mbox{if $x\in\bigcup_{j}\mathcal{T}_{ij}^+$}\\
R^\delta_{3i},& \mbox{otherwise.}
\end{array}
\right.
\endeq
The uniqueness follows from that $w_{A,T}^{-1}$ is  order-preserving on each $R^\delta_j$ (cf. Lemma \ref{lem:Dld}).

One verifies by construction that $\varphi(w_{A,T}) =T$. The lemma is proved.
\endproof

\lem
For $T\in\Tt_{B, A}$, the element $w_{A,T}$ determined by \eqref{def:wAT}
is the minimal length element in $\varphi^{-1}(T)$. Moreover, its length is
\begin{align}   \label{wATlength}
\ell(w_{A,T})
&=\sum\limits_{\substack{1\leq i \leq r \\ j\in\ZZ}}
\Bp{
  t_{ij}\sum\limits_{k < j}(A-T)_{ik}
  +
  {t}_{-i,-j}\sum\limits_{k>j}(A-T_\tt)_{ik})
}
\notag
\\
%2
& +\sum\limits_{\substack{k<j\leq 0\\ \text{or} \\  -k\geq j>0}} t_{0j}(A-T)_{0k}
+  \sum\limits_{|k|<j}    t_{0j} (A-T_\tt)_{0k}
-\sum\limits_{j>0}
     {t_{0j}\choose 2}
\\
%3
& +\sum\limits_{\substack{k<j\leq r+1 \\ \text{or} \\ n-k \geq j >r+1}} t_{r+1,j}(A-T)_{r+1,k}
+\sum\limits_{\substack{j>r+1, \\ |k-r-1|<j}}   t_{r+1,j} (A-T_\tt)_{r+1,k}
-\sum\limits_{j>r+1}  {t_{r+1,j}\choose 2}.
\notag
\end{align}
\endlem

\proof
It follows by construction that $w_{A,T}$ is the minimal length element in $\varphi^{-1}(T)$.

The permutation $w_{A,T}$ shifts elements in $\bigcup_{j\in\mathbb{Z}}\mathcal{T}^{-}_{ij}$
to the front of elements in the union  $\bigcup_{j\in\mathbb{Z}}(A\stdP)_{ij}\setminus\mathcal{T}^{-}_{ij}$,
and shifts elements in $\bigcup_{j\in\mathbb{Z}}\mathcal{T}^{+}_{ij}$ to the back of
$\bigcup_{j\in\mathbb{Z}}(A\stdP)_{ij}\setminus\mathcal{T}^{+}_{ij}$.
Note that $|\mathcal{T}^{-}_{ij}|=t_{ij}$ and $|\mathcal{T}^{+}_{ij}|=t_{-i,-j}$.

For $1\le i \le r$, we first count that there are $\sum_{j\in\mathbb{Z}}(t_{ij}\sum_{k<j}(a_{ik}-t_{ik}))$ elements in total
crossed by $\bigcup_{j\in\mathbb{Z}}\mathcal{T}^{-}_{ij}$, and then there are
$\sum_{j\in\mathbb{Z}}(t_{-i,-j}\sum_{k>j}(a_{ik}-t_{ik}-t_{-i,-k}))$ elements in total
crossed by $\bigcup_{j\in\mathbb{Z}}\mathcal{T}^{+}_{ij}$. This accounts for the sum on the first line of the RHS of \eqref{wATlength}.

When we move elements in $\mathcal{T}^{-}_{0j}$, their opposite elements in $\mathcal{T}^{+}_{0,-j}$
are moved automatically in a symmetric way.
Thus there are $\sum_{k<j\leq0}t_{0j}(a_{0k}-t_{0k})$ elements in total
crossed by $\bigcup_{j\leq 0}\mathcal{T}^{-}_{0j}$.
 And for each $j>0$,  there are
 $t_{0j} \sum_{k<-j}(a_{0k}-t_{0k})+ \big( t_{0j}(a_{0,-j}-t_{0,-j})- \sum_{s=1}^{t_{0j}-1} s % (1+2+\ldots+(t_{0j}-1))
 \big)
 +t_{0j} \sum_{|k|<j}(a_{0k}-t_{0k}-t_{0,-k})$
 elements in total crossed by $\mathcal{T}^{-}_{0j}$.
Adding these contributions gives us the second line on the RHS of \eqref{wATlength}.
Similarly, we obtain the third line of the RHS of  \eqref{wATlength} by considering $\mathcal{T}^{\pm}_{r+1,j}$.
\endproof

\lem\label{lem:sumwAT}
Let $T \in \Tt_{B,A}$, we have
\eq\label{eq:sumwAT}
\sum_{w\in\varphi^{-1}(T)}v^{2\ell(w)}=v^{2\ell(w_{A,T})}\frac{[A]^!_\fc}{[A-T_\tt]^!_\fc [T]^!}.
\endeq
\endlem

\proof
Any permutation $w \in \varphi^{-1}(T)$ is determined by which $t_{ij}$ and $t_{-i.-j}$
elements in $(A\stdP)_{i,j}$ is moved to left and right for all $(i,j)\in I^+$ (see \eqref{plus}), respectively.
Hence $\sum_{w\in\varphi^{-1}(T)}v^{2\ell(w)}$,
which is the generating function in $v$ of counting inversions for $\varphi^{-1}(T)$,
is a product of the generating functions of counting inversions locally
for each $(i,j)\in I^+$.

Recall $I^+ =\Ia  \cup  \{(0,0), (r+1,r+1)\}$ from \eqref{def:Ia}.
Let us determine these generating functions of local inversion countings
%local $v$-counts of inversions
one by one.
For each entry $(i,j)\in \Ia$, it is given by a standard recipe
$\dfrac{[a_{ij}]^!}{[t_{ij}]^![t_{-i,-j}]^![a_{ij}-(t_\tt)_{ij}]^!}$.

Let $k=0$ or $r+1$. Recall $a'_{kk}=\frac{a_{kk}-1}2$ from \eqref{def:a'}.
The generating function of counting inversions locally at the entry $(k,k) \in I^+$ is given by
\[
\bA{
\sum_{x+y=t_{kk}}
\lrb{a'_{kk}}{x}
&
\lrb{a'_{kk}-x}{y}
(v^2)^{\frac{x(x+1)}{2}+x(a'_{kk}-t_{kk})}
\\
&=\lrb{ a'_{kk}}{t_{kk}}
\sum_{x=0}^{t_{kk}}
\lrb{t_{kk}}{x}
v^{x(x-1)} (v^{a'_{kk}+1-t_{kk}})^{2x}
\\
&\stackrel{(\diamondsuit)}{=}\lrb{ a'_{kk}}{t_{kk}}
\prod_{i=1}^{t_{kk}}(1+v^{2(i-1)}v^{2(a'_{kk}-t_{kk}+1)})
%\\
%&= \lrb{ a'_{kk}}{t_{kk}} \prod_{i=1}^{t_{kk}} \frac{[a_{kk}+1 - 2i]}{[a'_{kk}+1 - i]}
%\\ &= \prod_{i=1}^{t_{kk}} \frac{[a_{kk}+1-2i]}{[t_{kk}+1-i]}
%\\
= \frac{[a_{kk}]^!_\fc}{[a_{kk} - 2t_{kk}]^!_\fc [t_{kk}]^!},
}
\]
where ($\diamondsuit$) uses the quantum binomial theorem
$
\sum_{r=0}^n
\lrb{n}{r}
v^{r(r-1)}x^r
=\prod_{k=0}^{n-1}(1+v^{2k}x).
$

Summarizing, we have obtained
\[
\bA{
\sum_{w\in\varphi^{-1}(T)}v^{2\ell(w)}
&=
v^{2\ell(w_{A,T})}
\prod_{(i,j)\in \Ia}
\frac{[a_{ij}]^!}{[t_{ij}]^![t_{-i,-j}]^![a_{ij}-(t_\tt)_{ij}]^!}
\prod_{k=0,r+1}\frac{[a_{kk}]^!_\fc}{[a_{kk} - 2t_{kk}]^!_\fc [t_{kk}]^!}
\\
&=v^{2\ell(w_{A,T})}\frac{[A]^!_\fc}{[A-T_\tt]^!_\fc [T]^!}.
}\]
The lemma follows.
\endproof

Recall the subset $K_w \subset W$ from \eqref{def:Kw(i)}--\eqref{def:Kw} for $w\in\D_\delta \cap \W_\mu$.
Also recall $h(w,\sig)$ from \eqref{eq:hw} and $H(w,\sig)$ from \eqref{def:Q(wsig)1},
for $w\in \D_\delta \cap W_\mu$ and $\sig \in K_w$.
Recall further from \eqref{def:wtoT} the map $\varphi : \D_\delta \cap W_\mu \rw \Tt_{B,A}$.

%---------------------------------------------------------------------------------------------------------
\lem\label{KwtoKT}
Let $w_1,w_2\in\D_\delta \cap W_\mu$. If $\varphi(w_1)=\varphi(w_2)$,  then $K_{w_1}=K_{w_2}$
and $H(w_1,\sig)=H(w_2,\sig)$ for any $\sig\in K_{w_1} (=K_{w_2})$;
in particular, we have $h(w_1,\sig)=h(w_2,\sig)$.
\endlem

\proof
It follows from $\varphi(w_1)=\varphi(w_2)$ that
$w_1^{-1}(x)$ and $w_2^{-1}(x)$ lie in the same entry of $A\stdP$, for any $x\in R^\delta_{3i-2}\cup R^\delta_{3i-1}$.
Thus for any $j\in R^\delta_{3i-2}$ and $k\in R^\delta_{3i-1}$,
we have $g_2^{-1}w_1^{-1}(k)<g_2^{-1}w_1^{-1}(j)$ if and only if $g_2^{-1}w_2^{-1}(k)<g_2^{-1}w_2^{-1}(j)$.
So $K_{w_1}=K_{w_2}$ and $H(w_1,\sig)=H(w_2,\sig)$ by the definition \eqref{def:Kw} and \eqref{def:Q(wsig)1}.
\endproof

%-----------------------------------------------------------------------------------------------------------
Thanks to Lemma~\ref{KwtoKT}, for $T\in \Tt_{B,A}$, we can define
\eq\label{def:KT}
K(T) = K_w
\quad
\textup{for some (or for all) }
w\in\varphi^{-1}(T),
\endeq
and further define, for $\sig\in K(T)$,
\eq\label{def:hTsig}
h(T,\sig) = h(w,\sig)
\quad
\textup{for some (or for all) }
w\in\varphi^{-1}(T).
\endeq

%%%%%%%%%%%%%%
\section{Algebraic combinatorics for $\Sjj$}
%{Multiplication formulas with tridiagonal generators}

For any $w\in\D_\delta \cap W_\mu$ and $\sig \in K_w$,
we recall from \eqref{g1wg2} that
$g_1\sig w g_2 = w^{(\sig)}_\ld y^{(w,\sig)} w^{(\sig)}_\nu$ and from \eqref{Aw} that
%denote the shortest representative in the double coset $W_\ld g_1 \sig w g_2 W_\nu$ by $y^{(w,\sig)} \in \D_{\ld\nu}$.
$A^{(w,\sig)} = (a^{(w,\sig)}_{ij}) = \kappa(\ld,y^{(w,\sig)},\nu)$.

Recall $\sig = \prod_{i=1}^{r+1} \sig^{(i)}\in K_w$ with $\sig^{(i)} \in K_w^{(i)}$. Let us fix the product
$\sig^{(i)} = \prod_{l=1}^{s_i} (j_l^{(i)}, k_l^{(i)})_\fc$
by requiring
\eq\label{def:sigconv}
j_1^{(i)} < j_2^{(i)} < \cdots < j_{s_i}^{(i)}.
\endeq
%Assume $\sig=\sig^{(1)}\sig^{(2)}\cdots\sig^{(r+1)} \in K_w$ where $\sig^{(i)}=(j_1^{(i)},k_1^{(i)})_\fc(j_2^{(i)},k_2^{(i)})_\fc\cdots(j_{s_i}^{(i)},k_{s_i}^{(i)})_\fc\in K_w^{(i)}$ for $i=1,2,\ldots,r+1$.
We further define $s_{-i}=s_{i+1}$ for $0 \leq i \leq r$ and
\eq\label{extendjk}
j_l^{(-i)}=k_{s_{i+1}-l+1}^{(i+1)},
\quad
k_l^{(-i)}=j_{s_{i+1}-l+1}^{(i+1)}
\quad
\tfor
\quad
0 \leq i\leq r,
1 \leq l \leq s_i.
\endeq
Hence the permutations $\sig^{(-i)} = \prod_{l=1}^{s_{-i}} (j_l^{(-i)}, k_l^{(-i)})_\fc$ for $0 \leq i \leq r$ satisfy \eqref{def:sigconv} as well.
%----------------------------------------------------------------------------------------------------------
For $w \in \D_\delta \cap W_\mu$,  we define a map
\begin{align}  \label{def:S}
\begin{split}
\psi_w & : K_w \longrightarrow \Tt_n,
\\
\psi_w(\sig)_{ij}
&= |R^\delta_{3i-1} \cap \sig(R^\delta_{3i-2})\cap w g_2 R_j^{\nu}|
= |\{k_l^{(i)}\}_{l=1}^{s_i}\cap w g_2 R_j^{\nu}|.
\end{split}
\end{align}
For any $\ZZ\times \ZZ$ matrix $T = (t_{ij})$, we set
\eq\label{def:T^}
\^{T}= (\^{t}_{ij}), \qquad \text{ where } \^{t}_{ij} = t_{i+1,j}.
\endeq
%-----------------------------------------------------------------------------------------------------------
Assume that $S = \psi_w(\sig)$ for some $w\in \D_\delta \cap W_\mu$ and $\sig \in K_w$. By \eqref{extendjk} we have
\eq
\bA{
\^{s}_{ij} &= | \{k_l^{(i+1)}\}_l \cap w g_2 R_{j}^\nu |,
\\
s_{-i,-j}
&
%= | \{k_l^{(-i)}\}_l \cap w g_2 R_{-j}^\nu |
=  | \{j_l^{(i+1)}\}_l \cap w g_2 R_{j}^\nu |,
\\
(\^{s})_{-i,-j} &
%= | \{k_l^{(-i-1)}\}_l \cap w g_2 R_{-j}^\nu |
= | \{j_l^{(i)}\}_l \cap w g_2 R_{j}^\nu |.
}
\endeq
%-----------------------------------------------------------------------------------------------------------
For any matrix $S=(s_{ij})$,
denote by $S^\dagger=(s^\dagger_{ij})$ the matrix obtained by rotating  the matrix $\^{S}$ by 180 degrees, namely,
\eq\label{def:dagger}
s^\dagger_{ij} = s_{1-i,-j} = \^{s}_{-i,-j}.
\endeq
%-----------------------------------------------------------------------------------------------------------
For $T\in\Tt_{B, A}$ (cf.~\eqref{def:ThetaBA}), we set
\eq\label{def:GammaT}
\Gamma_T = \{S \in \Tt_n~|~ S \leq_e T, \roA(S) = \roA(S^\dagger)\}.
\endeq

%-----------------------------------------------------------------------------------------------------------
\lem For $w\in\D_\delta \cap W_\mu$, we have
$\psi_w(K_w)\subset\Gamma_{\varphi(w)}.$
\endlem

\proof
For each $\sig\in K_w$, it  follows from~\eqref{extendjk} that $\roA(\psi_w(\sig)) = \roA(\psi_w(\sig)^\dagger)$.
Also, by Lemma~\ref{def:wtoT} we have
\[
|\{k_1^{(i)},k_2^{(i)},\ldots,k_{s_i}^{(i)}\}\cap wg_2R_j^{\nu}|\leq |R_{3i-1}^\delta\cap wg_2R_j^{\nu}|=\varphi(w)_{ij},
\]
and hence $\psi_w(\sig)\leq T$.
\endproof
%------------------------------------------------------------------------------------------------------------
For $T\in \Tt_{B,A}$ \eqref{def:ThetaBA} and $S\in\Gamma_T$ \eqref{def:GammaT}, we set
\begin{align}\label{def:A^TS}
A^{(T-S)} =A-(T-S)_\theta+(\widehat{T-S})_\theta,
\quad \text{ and }
A^{(T)} =A^{(T-0)}.
\end{align}
Recall $A^{(w,\sig)}$ from \eqref{Aw}.
%------------------------------------------------------------------------------------------------------------
\lem\label{Awsig}
For $w \in \D_\delta \cap W_\mu$ and $\sig \in K_w$, we have
\eq\label{A(T-S)}
A^{(w,\sig)}=A^{(T-S)},
\endeq
where $T=\varphi(w)$ and $S=\psi_w(\sig)$; cf.~\eqref{def:T} and \eqref{def:S}.
\endlem
%------------------------------------------------------------------------------------------------------------
\proof
By the definitions \eqref{def:T} and \eqref{def:S}, we have
\eq  \label{TS4}
\bA{
(T-S)_{ij} &= \big|(R^\delta_{3i-1} -\sig(R^\delta_{3i-2}))\cap w g_2 R_j^{\nu}\big|,
\\
(T-S)_{-i,-j} &= \big|(R^\delta_{3i+1} -\sig(R^\delta_{3i+2}))\cap w g_2 R_j^{\nu}\big|,
\\
(\^{T-S})_{ij} &= \big|(R^\delta_{3i+2} -\sig(R^\delta_{3i+1}))\cap w g_2 R_j^{\nu}\big|,
\\
(\^{T-S})_{-i,-j} &= \big|(R^\delta_{3i-2} -\sig(R^\delta_{3i-1}))\cap w g_2 R_j^{\nu}\big|.
}
\endeq

Recall from Lemma \ref{lem:ld-mu-delta} that
$R_i^\mu=R_{3i-1}^\delta\cup R_{3i}^\delta\cup R_{3i+1}^\delta$, and hence
\eq \label{aij3}
\bA{
a_{ij} &= \big|R_i^\mu \cap wg_2 R_j^\nu\big|
\\
& = \big| R_{3i-1}^\delta\ \cap wg_2 R_j^\nu\big|
+ \big| R_{3i}^\delta  \cap wg_2 R_j^\nu\big|
+ \big| R_{3i+1}^\delta \cap wg_2 R_j^\nu\big|.
%&=a^{(w,\sig)}_{ij} - (T-S)_{\tt,ij} + (\^{T-S})_{\tt,ij}.
}
\endeq

Again by  Lemma \ref{lem:ld-mu-delta}, we have
$
g_1^{-1}R_i^\ld=R_{3i-2}^\delta\cup R_{3i}^\delta\cup R_{3i+2}^\delta.
$
Therefore, by using \eqref{TS4}--\eqref{aij3}, we obtain
\[
\bA{
a^{(w,\sig)}_{ij}
&=\big|R^\ld_i\cap g_1\sig wg_2R^\nu_j\big|
=\big|\sig g_1^{-1}R^\ld_i\cap  wg_2R^\nu_j\big|
\\
&=\big|\sig(R_{3i-2}^\delta)\cap wg_2R^\nu_j\big|
+\big|R_{3i}^\delta\cap wg_2R^\nu_j\big|
+\big|\sig(R_{3i+2}^\delta)\cap wg_2R^\nu_j\big|
\\
&=a_{ij} - (T-S)_{\tt,ij} + (\^{T-S})_{\tt,ij}.
%&=(|R_{3i-2}^\delta\cap wg_2R^\nu_j|-|\{j_1^{(i)},j_2^{(i)},\ldots,j_{s_i}^{(i)}\}\cap w g_2R_j^{\nu}|+|\{k_1^{(i)},k_2^{(i)},\ldots,k_{s_i}^{(i)}\}\cap w g_2
%R_j^{\nu}|)\\
%&\quad +(|R_i^\mu\cap wg_2R^\nu_j|-|R_{3i-1}^\delta\cap wg_2R^\nu_j|-|R_{3i+1}^\delta\cap wg_2R^\nu_j|)\\
%&\quad +(|R_{3i+2}^\delta\cap wg_2R^\nu_j|-|\{k_1^{(i+1)},k_2^{(i+1)},\ldots,k_{s_{i+1}}^{(i+1)}\}\cap w g_2R_j^{\nu}|\\
%&\quad\quad\quad+|\{j_1^{(i+1)},j_2^{(i+1)},\ldots,j_{s_{i+1}}^{(i+1)}\}\cap w g_2
%R_j^{\nu}|)\\
%&=(t_{1-i,-j}-s_{1-i,-j}+s_{ij})+(a_{ij}-t_{ij}-t_{-i,-j})+(t_{i+1,j}-s_{i+1,j}+s_{-i,-j})\\
%&=a_{ij}-(t_{ij}+t_{-i,-j}-s_{ij}-s_{-i,-j})+(\hat{t}_{ij}+\hat{t}_{-i,-j}-\hat{s}_{ij}-\hat{s}_{-i,-j})
}
\]
The lemma is proved.
\endproof

\exa\label{exa:S}
%(Example 4.8)
Retain the notation as in Example \ref{exa:tri}.
The matrices in $\Tt_{B,A}$ have zero rows except for the -1st and 2nd (mod $n$) rows, and hence $\Tt_{B,A} = \{ T^{(-1)}_x + T^{(2)}_y~|~ 1\leq x, y \leq 6\}$, where
\[
\{T_x^{(-1)} | 1 \leq x \leq 6\} = \{E^{-1,j}+E^{-1,k} | -4 \leq j < k \leq -1\},
\]
\[
\{T_y^{(2)} | 1 \leq y \leq 6\} = \{E^{2,j} + E^{2,k} | 1 \leq j < k \leq 4\}.
\]
For any $\ZZ\times \ZZ$ matrix $M = (m_{ij})$, recalling \eqref{Ekl}--\eqref{Ttt}, we introduce another short-hand notation
\[
M_\fa = \sum_{(i,j)\in I^+}m_{\tt,ij} E^{ij}.
\]
In particular, for $T =  E^{-1,-4} + E^{-1,-3} + E^{21} + E^{22} \in \Tt_{B,A}$, we have
\[
T_\fa = \young(0011,1100) .
%\leftrightarrow \young(\varnothing\varnothing\negthree\negfour,56\varnothing\varnothing)
%\quad
%\textup{where}
%\quad R^\delta_{-1} = [-4..-3], R^\delta_2 = [5..6].
\]
In this case we have $\varphi^{-1}(T) = \{\id\}$. Moreover, by Example \ref{exa:tri}, we have
\[
K(T) = \{\id, (3,5)_\fc, (3,6)_\fc, (4,5)_\fc, (4,6)_\fc, (3,5)_\fc(4,6)_\fc, (3,6)_\fc(4,5)_\fc\}.
\]
The complete list of $S \in \Gamma_T$ is given by
\[
\ba{{c|c|cc}
S&S_\fa& \psi_w^{-1}(S)
\\
\hline
0&\young(0000,0000)& \{\id\}
\\
E^{-1,-3} + E^{21}&\young(0010,1000) & \{(3,5)_\fc\}
\\
E^{-1,-3} + E^{22}&\young(0010,0100) & \{(3,6)_\fc\}
\\
E^{-1,-4} + E^{21}&\young(0001,1000) & \{(4,5)_\fc\}
\\
E^{-1,-4} + E^{22}&\young(0001,0100) & \{(4,6)_\fc\}
\\
T &\young(0011,1100)& \{(3,5)_\fc(4,6)_\fc, (3,6)_\fc(4,5)_\fc\}
}
\]
\endexa

\vspace{.3cm}
We define an element
\eq  \label{sws}
\sig_{w,S}= \prod\limits_{i=1}^{r+1} \prod\limits_{l=1}^{s_i} (j_l^{(i)}, k_l^{(i)})_\fc \in \psi_w^{-1}(S)
\endeq
satisfying Conditions (S1)--(S2) below:
\enu
\item[(S1)] $k_1^{(i)} < k_2^{(i)} < \cdots < k_{s_i}^{(i)}$, $\forall i$;
\item[(S2)]
$w^{-1}(\{k_l^{(i)}\}_l)\cap g_2 R_j^\nu$ consists of the largest $s_{ij}$ elements in $w^{-1}R_{3i-1}^\delta\cap g_2R_j^{\nu}$, $\forall i$.
\endenu
It follows from \eqref{def:sigconv} that  Conditions (S1)--(S2) together imply Condition~(S3) below:
\enu
\item[(S3)]
$w^{-1}(\{j_l^{(i)}\}_l)\cap g_2 R_j^\nu$ consists of the smallest $s_{ij}$ elements in $w^{-1}R_{3i-2}^\delta\cap g_2R_j^{\nu}$, $\forall i$.
\endenu
%-----------------------------------------------------------------------------------------------------------
For $S \in \Gamma_T$, recall $S^\dagger$ from \eqref{def:dagger} and set
\eq\label{def:[[S]]}
\LR{S} = \prod_{i=1}^{r+1} \LR{S}_i,
\endeq
where
\eq
\LR{S}_i = \prod_{ j \in \ZZ}
\lrb{\sum\limits_{k\leq j} (S- {S}^\dagger)_{ik}}{{s}^\dagger_{i,j+1}} [{s}^\dagger_{i,j+1}]^!.
\endeq
Each $\LR{S}_i$ counts the ``quantum number'' of pairs $(x,y)$ in the following sense:
\enu
\item The element $x$ contributes to the $i$th row of $S$. That is, $x \in \{k^{(i)}_l\}_{l=1}^{s_i}$;
\item The element $y$ contributes to the $i$th row of $S^\dagger$. That is, $y \in \{j^{(i)}_l\}_{l=1}^{s_i}$;
\item The element $x$ is ``to the left'' of $y$ as elements in $A\stdP$.
\endenu

Recall $\psi_w : K_w \rightarrow \Tt_n$ from \eqref{def:S} and $\sig_{w,S}$ from \eqref{sws}.

\lem  \label{lem:standardf2}
Let $T \in \Tt_{B,A}$ and $S \in \Gamma_T$. For any $w\in\varphi^{-1}(T)$, we have
\eq  \label{eq:standardf2}
\sum_{\sig\in\psi_w^{-1}(S)}v^{-2h(w,\sig)}= v^{-2h(w,\sig_{w,S})} \lrb{T}{S} \LR{S}.
\endeq
\endlem
%-----------------------------------------------------------------------------------------------------------

\proof
Let $s_i=\roA(S)_i=\roA(S^\dagger)_i$ for all $i$.
By definition, each $\sig\in\psi_w^{-1}(S)$ can be reconstructed by the following steps:
\enu
\item
For $1 \leq i \leq r+1, j\in\mathbb{Z}$, choose $s_{i,j}^\dag$ elements from the set $R_{3i-2}^\delta\cap wg_2R_j^{\nu}$.

\item
Let $j_\sig = \{j_1^{(i)},j_2^{(i)},\ldots,j_{s_i}^{(i)}\}$ be the set of elements chosen from
$\bigcup\limits_{j\in\mathbb{Z}}R_{3i-2}^\delta\cap wg_2R_j^{\nu}$ such that
$j_1^{(i)}<j_2^{(i)}<\cdots<j_{s_i}^{(i)}.$

\item
For $1 \leq i \leq r+1, j\in\mathbb{Z}$, choose $s_{i,j}$ elements from the set $R_{3i-1}^\delta\cap wg_2R_j^{\nu}$.

\item
Let $k_\sig = \{k_1^{(i)},k_2^{(i)},\ldots,k_{s_i}^{(i)}\}$ be the set of elements chosen from
$\bigcup\limits_{j\in\mathbb{Z}}R_{3i-1}^\delta\cap wg_2R_j^{\nu}$ such that
\[
(wg_2)^{-1}(j_s^{(i)})>(wg_2)^{-1}(k_s^{(i)}),
\quad
\textup{for}
\quad
s = 1,2,\ldots, s_i.
\]
(Note that we do not impose that $k_1^{(i)}<k_2^{(i)}<\cdots<k_{s_i}^{(i)}$.)

\item Set $\sig=\prod\limits_{i=1}^{r+1}(j_1^{(i)},k_1^{(i)})_\fc\cdots(j_{s_i}^{(i)},k_{s_i}^{(i)})_\fc$.
\endenu

For those $\sig \in \psi_w^{-1}(S)$ having the same $k_\sig$ (say $k_\sig = K$), we pick a representative
\[
\sig^{\lhd K} =\prod_{i=1}^{r+1}(j_1^{(i)},k_1^{(i)})_\fc\cdots(j_{s_i}^{(i)},k_{s_i}^{(i)})_\fc
\]
such that $k_1^{(i)}<k_2^{(i)}<\cdots<k_{s_i}^{(i)}$.
We claim that
\eq\label{sigstd1}
\sum\limits_{\sig \in K_w, k_\sig = K}v^{-2h(w,\sig)}=\LR{S}v^{-2h(w,\sig^{\lhd K})}.
\endeq

We prove \eqref{sigstd1}. Indeed, any $\sig\in K_w$ with $k_\sig = K$ must be of the form
\[
\sig=\prod\limits_{i=1}^{r+1}(j_1^{(i)},k_{\tau_i(1)}^{(i)})_\fc\cdots(j_{s_i}^{(i)},k_{\tau_i(s_i)}^{(i)})_\fc
\]
for some $\tau_i \in K'_i ~(1\leq i \leq r+1)$, where
\[
K'_i =
\Big\{ \tau \in \textup{Perm}([1 .. s_i]) ~|~ (wg_2)^{-1}(j_s^{(i)})>(wg_2)^{-1}(k_{\tau(s)}^{(i)})
\textup{ for } 1\leq s \leq s_i \Big\}.
\]
By a detailed calculation, for each $i$ we have
\[
\sum\limits_{\tau \in K'_i}
v^{2\ell(\tau)}
=\prod\limits_{j\in\mathbb{Z}}
\lrb{\sum\limits_{l\leq j} (S- {S}^\dagger)_{il}}
{{s}^\dagger_{i,j+1}}
[{s}^\dagger_{i,j+1}]^!
= \LR{S}_i.
\]
Now \eqref{sigstd1} follows by computing
\[
\sum\limits_{\sig \in K_w, k_\sig = K}
v^{-2h(w,\sig)}
=\prod\limits_{i=1}^{r+1}
\sum\limits_{\substack{\tau \in K'_i}}
v^{2\ell(\tau)}v^{-2h(w,\sig^{\lhd K})}
= v^{-2h(w,\sig^{\lhd K})} \LR{S}.
\]

By the construction of $\sig_{w,S}$ we have
\eq\label{sigstd2}
\sum\limits_{K}v^{-2h(w,\sig^{\lhd K})}= v^{-2h(w,\sig_{w,S})} \lrb{T}{S}.
\endeq
The lemma follows by combining~\eqref{sigstd1} and~\eqref{sigstd2}.
\endproof

%-----------------------------------------------------------------------------------------------------------
\exa\label{exa:GammaT}
Following Example \ref{exa:S}, we pick the element $S = T \in \Gamma_T$. Thus
\[
S_A = \young(0011,1100),
\quad
S^\dagger_A = \young(1100,0011),
\quad
(\sum_{k\leq j} (S- S^\dagger)_{ik})_{ij} = \young(0121,1210).
\]
Therefore
\[
\lrb{T}{S} = 1,
\quad
\LR{S} = \lrb{\young(~~~~,1210)}{\young(~~~~,0110)} = [2],
\quad
\sig_{\id,S} = (3,5)_\fc (4,6)_\fc,\   h(\id, \sig_{\id,S}) = 1.
\]
Also we have
$\psi_{\id}^{-1}(S) = \{(3,6)_\fc(4,5)_\fc, (3,5)_\fc(4,6)_\fc\}$.
According to Example \ref{exa:tri}, we have
$
\text{LHS of }  \eqref{eq:standardf2} = v^{0} + v^{-2} =  v^{-2}[2]=
\text{ RHS of } \eqref{eq:standardf2}.
$
\endexa
\vspace{.3cm}

%-------------------------------------------------------------------------------------------------------------
For $T\in \Tt_{B,A}$ \eqref{def:ThetaBA} and $S\in\Gamma_T$ \eqref{def:GammaT}, we set
\eq\label{def:n(S)}
n(S) =\sum\limits_{i=1}^{r+1} \roA(S)_i,
\endeq
and
\eq\label{def:h(ST)}
\bA{
h(T,S)
&= \sum\limits_{i=1}^{r+1}\sum\limits_{j=-\infty}^{\infty}s_{ij}(\sum\limits_{k=-\infty}^jt_{ik}-\frac{s_{ij}+1}{2})
\\
&+\sum\limits_{i=1}^{r+1}\sum\limits_{j=-\infty}^{\infty}(t_{1-i,-j}-s_{1-i,-j})(\sum\limits_{k=-\infty}^{j-1}t_{ik}+\sum\limits_{k=j}^{\infty}s_{ik}-\sum\limits_{k=j+1}^{\infty}s_{1-i,-k}).
}
\endeq
%-------------------------------------------------------------------------------------------------------------
\lem
For  $T \in \Tt_{B,A}$ and $S\in \Gamma_T$, we have
\[
n(\sig_{w,S}) = n(S), \qquad
h(T,\sig_{w,S})  = h(T,S).
\]
\endlem
%-------------------------------------------------------------------------------------------------------------
\proof
The first statement is obvious since $n(\sig_{w,S})$ is the number of disjoint transpositions for $\sig_{w,S}$.
To compute $h(T,\sig_{w,S})$, We should count the elements in $H(T,\sig_{w,S})$. There are $\sum_{i=1}^{r+1}\sum_{j=-\infty}^{\infty}(t_{1-i,-j}-s_{1-i,-j})(\sum_{k=-\infty}^{j-1}t_{ik}+\sum_{k=j}^{\infty}s_{ik}-\sum_{k=j+1}^{\infty}s_{1-i,-k})$
elements $(u,v)\in H(T,\sig_{w,S})$ such that $\sig_{w,S}(u)=u$ while there are $\sum_{i=1}^{r+1}\sum_{j=-\infty}^{\infty}s_{ij}(\sum_{k=-\infty}^jt_{ik}-\frac{s_{ij}+1}{2})$ elements $(u,v)\in H(T,\sig_{w,S})$ such that $u$ appears in the disjoint transpositions of $\sig_{w,S}$.
\endproof
%-------------------------------------------------------------------------------------------------------------
%\rem
%It follows from \eqref{eq:f1} (respectively,  \eqref{eq:f2}) that $n(\sig_{w,S})$ (respectively,  $h(T,\sig_{w,S})$) depends on $S$ (respectively,  $S$ and $T$). Thus we use $n(S)$ and $h(S,T)$ instead of $n(\sig_{w,S})$ and $h(T,\sig_{w,S})$, respectively.
%%Moreover, we also denote $\ell(A,B,S,T)=\ell(g_1) + \ell(w_{A,T}) + \ell(g_2)-\ell(y^{(T,S)})$.
%\endrem

%%%%%%%%%%
\section{Multiplication formula for $\Sjj$}

We are now in the position of establishing a crucial multiplication formula for affine Schur algebra $\Sjj$.
For $A,B\in\Xi_n, T \in \Tt_{B,A}$ and $S\in\Gamma_T$,
we recall $\LR{S}, n(S),  h(S,T), A^{(T-S)}$ from \eqref{def:[[S]]}, \eqref{def:n(S)}, \eqref{def:h(ST)}, \eqref{def:A^TS} respectively.
We further denote
\eq\label{def:l3}
\ell(A,B, S, T) = \ell(A)+\ell(B)-\ell(A^{(T-S)})+ \ell(w_{A,T}),
\endeq
where $\ell(C)$ means $\ell(g)$ if $C=\kappa(\ld,g,\mu)$, $g\in\D_{\ld,\mu}$.

%-------------------------------------------------------------------------------------------------------------
\thm\label{thm:multformula}
Let $A, B \in \Xi_{n,d}$ with $B$  tridiagonal and $\roC(A)=\coC(B)$.
Then we have
\eq\label{eq:multformula}
e_B   e_A =
\sum\limits_{\substack{T \in \Tt_{B,A}\\ S \in \Gamma_T}}
(v^2-1)^{n(S)}
v^{2( \ell(A,B,S,T)-n(S)-h(S,T) )}
\LR{A;S;T}~ e_{A^{(T-S)}},
\endeq
where
\eq\label{def:[AST]}
\LR{A;S;T} =
\dfrac{[A^{(T-S)}]^!_\fc}{[T-S]^! [S]^! [A-T_\tt]^!_\fc}
\LR{S}.
\endeq
\endthm
%-------------------------------------------------------------------------------------------------------------
\proof
We have
\[
\ba{{lr}
e_B  e_A \\
\ds=\sum_{\substack{ w\in \D_\delta \cap \W_\mu \\ \sig\in K_w }}
(v^2-1)^{n(\sig)}
(v^2)^{h(w,\sig) + \ell(g_1\sig w g_2) - \ell(y^{(w,\sig)})}
\frac{[A^{(w,\sig)}]^!_\fc}{[A]^!_\fc}
e_{A^{(w,\sig)}}
\quad\quad\quad\textup{(by \eqref{eq:mult1}, \eqref{heckemult})}
\\
\ds=\sum_{\substack{ w\in \D_\delta \cap \W_\mu \\ \sig\in K_w }}
(v^2-1)^{n(\sig)}
(v^2)^{\ell(g_1) + \ell(w) + \ell(g_2)-\ell(y^{(w,\sig)})-n(\sig)-h(w,\sig)}
\frac{[A^{(w,\sig)}]^!_\fc}{[A]^!_\fc} e_{A^{(w,\sig)}}
\quad\quad\textup{(by \eqref{eq:l(g_1etawg_2)})}
\\
\ds=\sum_{\substack{ T\in\Tt_{B,A} \\ S\in\Gamma_T}}
(v^2-1)^{n(S)}
v^{2( \ell(A,B,S,T)-n(S)-h(S,T) )}
\LR{A;S;T}
e_{A^{(T-S)}}.
\quad\textup{(by \eqref{eq:sumwAT}, \eqref{eq:standardf2}, \eqref{A(T-S)})}
}
\]
%Finally, $A^{(T-S)}=A-(T-S)_\theta+(\widehat{T-S})_\theta$ because of Lemma \ref{Awsig}.
The proof is finished.
\endproof

In more concrete term, we can express $\LR{A;S;T}$ as
\eq\label{eq:[AST]}
\bA{
\LR{A;S;T}
&
=
\prod\limits_{(i,j) \in \Ia}
\lrb{ (A-T_\tt)_{i, j} + s_{ij} + s_{-i,-j} + (\^{T-S})_{ij} + (\^{T-S})_{-i,-j}}
{(A-T_\tt)_{i, j} ; s_{ij} ; s_{-i,-j} ; (\^{T-S})_{ij} ; (\^{T-S})_{-i,-j}}
\\
&\qquad \cdot \prod\limits_{k \in \{0, r+1\}}
\frac{
\prod\limits_{i=1}^{s_{kk}+(\^{T-S})_{kk}}
[a_{kk}-2t_{kk} -1 + 2i]
}{[s_{kk}]^! [(\^{T-S})_{kk}]^!}
\cdot \LR{S}.
}
\endeq

\rem
There is also a geometric multiplication formula in different form for $\Sjj$ in \cite{FL16}.
We have checked that the two formulas match well in various examples including the specific case in Proposition \ref{prop:mult(sl)}.
\endrem

%-----------------------------------------------------------------------------------------------------------
\section{Special cases of the multiplication formula}
%\section{quasi-bidiagonal case}

For any matrix $T= (t_{ij}) \in \Tt_n$,  let $\diag(T) = (\delta_{ij} t_{ij}) \in \Tt_n$ and let
\eq\label{def:offd}
T^\offd = T-\diag(T) \in \Tt_n.
\endeq
We shall describe below the special case of the multiplication formula for $\Sjj$ (see Theorem~\ref{thm:multformula})
when $B$ satisfies $B^\offd  = \sum_{i=0}^r b_{i,i+1}\Ett^{i,i+1}$ or $B^\offd = \sum_{i=0}^r b_{i+1,i}\Ett^{i+1,i}$.
(Note that $g_1 = \id$ here.
A direct proof of this special case is much easier than the general case in Theorem~\ref{thm:multformula}.)
The formula is analogous to the multiplication formula in affine type  A in \cite{DF14}.

%-----------------------------------------------------------------------------------------------------------
\prop  \label{prop:mult(sl)}
Let $A, B \in \Xi_{n,d}$. Assume
$B^\offd =\sum_{i=0}^r b_{i,i+1}\Ett^{i,i+1}$ or $B^\offd =\sum_{i=0}^r b_{i+1,i}\Ett^{i+1,i}$
and $\roC(A)=\coC(B)$.
Then
\[
e_B e_A = \sum_{T \in \Tt_{B, A}}
v^{2\left( \ell(w_{A,T}) + \ell(A) -\ell(A^{(T)}) \right)}
\frac{[A^{(T)}]_\fc!}{[A-T^{\theta}]_\fc![T]!}
e_{A^{(T)}}.
\]
\endprop
%-----------------------------------------------------------------------------------------------------------
\proof
This is a special case of \eqref{eq:multformula} where $S$ is always the zero matrix.
\endproof

Let $\epsilon_{ij}^\theta$ be the $(i,j)$-th entry of $\Ett^{i,j}$, that is,
\eq
\epsilon_{ij}^\theta=
\left\{
\begin{array}{ll}
2, & \mbox{if $(i,j)=(0,0)$ or $(r+1,r+1)$;}\\
1, & \mbox{otherwise.}
\end{array}
\right.
\endeq
Below is a a further specialization of Proposition~\ref{prop:mult(sl)} when $B$ is a Chevalley generator.
Another multiplication formula with $B$ being a ``divided power"
(e.g., $B=\diag(B)+R \Ett^{h,h+1}$, for $0 \leq h \leq r$ and $R \in \NN$) can also
be easily available as a specialization  of Proposition~\ref{prop:mult(sl)}, and we shall skip it.
Note the formula below is analogous to the multiplication formula for Schur algebra of finite type B/C \cite{BKLW}
where a geometric approach was used.
(Indeed the Hecke algebraic approach of this paper can be easily adapted
to reproduce the multiplication formulas with divided powers therein.)

%------------------------------------------------------------------------------------------------
\begin{cor}
Let $0 \leq h \leq r$.
\enu
\item[(a)]
If $B=\diag(B)+\Ett^{h,h+1}$ and $\roC(A)=\coC(B)$, then
\[
e_B e_A=
\sum\limits_{\substack{p\in \ZZ\\ a_{h+1,p}\geq\epsilon_{h+1,p}^\theta}}
v^{2\sum_{j>p} a_{hj}}
[a_{hp}+1]
e_{A+\Ett^{hp}-\Ett^{h+1,p}}.
\]
\item[(b)]
If $B=\diag(B)+\Ett^{h+1,h}$ and $\roC(A)=\coC(B)$, then
\[
e_B e_A=
\sum\limits_{ \substack{ p\in\ZZ \\ a_{hp} \geq \epsilon_{hp}^\theta }}
v^{2\sum_{j<p}a_{h+1,j}}
[a_{h+1,p}+1]
e_{A-\Ett^{hp}+\Ett^{h+1,p}}.
\]
\endenu
\end{cor}

Finally, we note the multiplication formula in Theorem~\ref{thm:multformula}
can be greatly simplified at the classical limit $v=1$. Let $(2m)!!=(2m) \cdots 4\cdot 2$, for $m\ge 1$ and $0!!=1$.
We have the following multiplication formula for the (classical) affine Schur algebra (which is defined as for $\Sjj$ with Hecke algebra $\bH$ replaced by Weyl group $W$).

%--------------------------------------------------------------------------------
\cor
Let $A, B \in \Xi_{n,d}$ with $B$ tridiagonal and $\roC(A)=\coC(B)$.
Then we have
\eq\label{MFq=1}
\bA{
e_B e_A|_{v=1}=\sum_{T=(t_{ij}) \in \Tt_{B,A}}
\frac{\prod_{(i,j)\in I^{\fa}}(a_{ij}-t_{ij}-t_{-i,-j}+t_{i+1,j}+t_{-i+1,-j})!}{(\prod_{i=1}^n\prod_{j\in\mathbb{Z}}t_{ij}!)
(\prod_{(i,j)\in I^{\fa}}(a_{ij}-t_{ij}-t_{-i,-j})!)} \qquad{}\\
  \cdot\prod_{i=0,r+1}\frac{(a_{ii}-2t_{ii}+2t_{i+1,i}-1)!!}{(a_{ii}-2t_{ii}-1)!!}e_{A-T_\theta+\widehat{T}_{\theta}}.
}\endeq
\endcor
%--------------------------------------------------------------------------------
\proof
All terms in \eqref{eq:multformula}  having positive powers of $(v^2-1)$ specialize to $0$ at $v=1$.
Thus we only need to consider the special terms for which $S$ is the zero matrix. Then we easily obtain \eqref{MFq=1}
upon replacing all quantum numbers in \eqref{eq:multformula} by ordinary numbers.
\endproof

%%%%%%%%%%%%
%%%%%%%%%%%%
%=========================================================
\chapter{Monomial and canonical bases for affine Schur algebra}
  \label{sec:basis}

In this chapter, we construct canonical and monomial bases of the affine Schur algebra $\Sjj$.
The canonical basis is then identified with the canonical basis defined geometrically in \cite{FLLLW}.
The construction of the monomial basis here
will play a crucial role in the setting of the stabilization algebras in the next chapters.
We shall show that the identification of the algebra $\Sjj$ and the geometrically defined version in \cite{FLLLW}
preserves the canonical bases.

%=========================================================
\section{Bar involution on $\Sjj$}  \label{secbar}

Recall the bar map on $\bH$ is an $\ZZ$-algebra involution $\bar{~}:\bH \rw \bH$,
which sends $v \mapsto v^{-1}, T_w \mapsto T_{w^{-1}}^{-1}$, for all $w\in W$.
Let $\leq$ be the (strong) Bruhat order on $W$.
Following \cite{KL79}, denote by $\{C'_w\}$ the Kazhdan-Lusztig basis of the Hecke algebra $\bH$
characterized by Conditions (C1)--(C2) below:
\itm
\item[(C1)]
$C'_w$ is bar-invariant;

\item[(C2)]
$C'_w = v^{-\ell(w)} \sum_{y \leq w} P_{yw} T_y,$
where $P_{ww}=1$ and $P_{yw}\in \ZZ[v^2]$ for $y<w$ has degree in $v$ $\leq \ell(w)-\ell(y)-1$.
% is the Kazhdan-Lusztig polynomial.
\enditm

For $\ld,\mu \in \Ld$ \eqref{def:Ld},
set $g^+_{\ld\mu}$ to be the {\em longest} element in $W_\ld g W_\mu$ for $g \in \D_{\ld\mu}$,
and set $w_\circ^\mu = \id_{\mu\mu}^+$ to be the longest element in the (finite) parabolic subgroup $W_\mu = W_\mu \id W_\mu$.
Recall $x_\mu$ for $\mu \in \Ld$ from \eqref{eq:x}.

%-----------------------------------------------------------------------------------------------------------
\lemma\label{lem:Cur}
Let $\ld,\mu \in \Ld$, $g\in \D_{\ld\mu}$, and $\delta = \delta(\ld,g,\mu)$; see Proposition~\ref{prop:doublecoset}.
Then we have
\enua
\item $g_{\ld\mu}^+ = w_\circ^\ld g w_\circ^{\delta} w_\circ^\mu$, and
$\ell(g_{\ld\mu}^+) =  \ell(w_\circ^\ld) +  \ell(g)  - \ell(w_\circ^{\delta}) +  \ell(w_\circ^\mu).$

\item $W_\ld g W_\mu = \{w \in W ~|~ g \leq w \leq g^+_{\ld\mu}\}$.

\item $T_{W_\ld g W_\mu}
= v^{\ell(g^+_{\ld\mu})} C'_{g^+_{\ld\mu}} + \sum_{\substack{y\in \D_{\ld\mu}\\
 y < g }} c^{(\ld,\mu)}_{y,g} C'_{y^+_{\ld\mu}}$,
for $c^{(\ld,\mu)}_{y,g}\in\mA.$
Moreover,
$x_\mu = v^{\ell(w_\circ^\mu)} C'_{w_\circ^\mu}$.
\endenua
\endlemma
%-----------------------------------------------------------------------------------------------------------
\proof
See \cite[Theorem 1.2, (1.11)]{Cur85} and \cite[Corollary 4.19]{DDPW}.
\endproof
%-----------------------------------------------------------------------------------------------------------
Following \cite[Proposition~3.2]{Du92},
we define a bar involution $\bar{\phantom{x}}$ on $\Sjj$ as follows:
for each $f \in \textup{Hom}_{\bH}(x_\mu \bH, x_\ld \bH)$, let $\={f}\in\textup{Hom}_{\bH}(x_\mu \bH, x_\ld \bH)$ be the
$\bH$-linear map which sends $C'_{w_\circ^\mu}$ to $\={f(C'_{w_\circ^\mu})}$, that is,
\[
\={f}(x_\mu H) = v^{2\ell(w_\circ^\mu)}\={f(x_\mu)} H,
\quad
\tfor
H \in \bH.
\]
For $A=\kappa(\ld,g,\mu) \in \Xi_n$,
we have $e_A(x_\mu) = \phi_{\ld\mu}^g(x_\mu) =T_{W_\ld g W_\mu}$; see \eqref{phi} and \eqref{def:eA}.
Hence by Lemma~\ref{lem:Cur}(c) we have
\eqnarray
e_A(C'_{w_\circ^\mu})
&\ds = v^{\ell(g^+_{\ld\mu})-\ell(w_\circ^\mu)} C'_{g^+_{\ld\mu}}
 + \sum_{\substack{y\in\D_{\ld\mu}\\y < g}}
v^{-\ell(w_\circ^\mu)} c_{y,g}^{(\ld,\mu)} C'_{y^+_{\ld\mu}},
    \label{eq:eA}
\\
\={e_A}(C'_{w_\circ^\mu})
&\ds = v^{\ell(w_\circ^\mu)-\ell(g^+_{\ld\mu})} C'_{g^+_{\ld\mu}}
+ \sum_{\substack{y\in\D_{\ld\mu}\\y < g}}
v^{\ell(w_\circ^\mu)}\={c_{y,g}^{(\ld,\mu)}} C'_{y^+_{\ld\mu}}.
    \label{eq:eAbar}
\endeqnarray
%-----------------------------------------------------------------------------------------------------------
%\lem\label{lem:eAbar}
%Let $A=\kappa(\ld,g,\mu)$ for some $\ld,\mu \in \Ld, g\in \D_{\ld\mu}$. Then
%\[
%\={e_A} = v^{\ell(w_\circ^\mu) - \ell(g_{\ld\mu}^+)}e_A +\textup{lower terms}.
%\]
%\endlem
%%-----------------------------------------------------------------------------------------------------------
%\proof
%By definition, $\={e_A}(C'_{w_\circ^\mu}) = v^{\ell(w_\circ^\mu)}\={e_A(x_\mu)} = v^{\ell(w_\circ^\mu)}\={T_{W_\ld g W_\mu}}$, and hence by Lemma~\ref{lem:Cur}(c) we have
%\[
%\={e_A}(C'_{w_\circ^\mu}) = v^{\ell(w_\circ^\mu)-\ell(g^+_{\ld\mu})} C'_{g^+_{\ld\mu}} + \textup{lower terms}.
%\]
%On the other hand, we have
%\[
%e_A(C'_{w_\circ^\mu}) = v^{\ell(g^+_{\ld\mu})-\ell(w_\circ^\mu)} C'_{g^+_{\ld\mu}} + \textup{lower terms}.
%\]
%It then follows by an easy induction.
%\endproof
%-----------------------------------------------------------------------------------------------------------

%=========================================================
\section{A standard basis in $\Sjj$}

\label{sjj}

For each $A = \kappa(\ld,g,\mu) \in \Xi_{n,d}$,
we set
\[
\ell(A) =\ell(g),
\]
the length of the corresponding Weyl group element $g$. By rephrasing Lemma~\ref{lem:l(g)}, we obtain the following
expression of $\ell(A)$ as a polynomial in the entries of $A$.
%-----------------------------------------------------------------------------------------------------------
\begin{prop} \label{prop:l(A)}
Let $A=(a_{ij}) \in \Xi_n$ and recall $a'_{ij}$  in \eqref{def:a'}. Then we have
\eq\label{eq:l(A)}
\ba{{l}
\ell(A) =
\dfrac{1}{2}
\bigg(
\sum\limits_{(i,j) \in I^+ }
\Bp{
\sum\limits_{\substack{x < i\\ y > j}}
+
\sum\limits_{\substack{x > i\\ y < j}}
}
a'_{ij} a_{xy}
\bigg).
}
\endeq
\end{prop}

We define
\eq\label{def:dA}
d_A =
\dfrac{1}{2}
\bigg(
\sum_{(i,j) \in I^+ }
\Big(
\sum_{x \leq i, y > j}
+
\sum_{x \geq i, y < j}
\Big)
a'_{ij} a_{xy}
\bigg),
\qquad
\text{ for } A \in \Xi_{n,d}.
\endeq
We shall see in Proposition~\ref{prop:Abar}(1)  that $d_A\in\NN$. Set
\eq \label{eq:std}
[A] = v^{-d_A} e_A.
\endeq
Then it follows by Lemma~\ref{lem:basis} that
$\{[A] ~|~ A \in \Xi_{n,d}\}$ forms an $\mA$-basis of $\Sjj$, which we call the \textit{standard basis}.

Note that it was defined in \cite{FLLLW} that
\begin{align} \label{eqda}
d_A
& = \frac{1}{2} \Big(\sum_{\substack{i\geq k, j<l \\ i\in [0, n-1 ]} } a_{ij} a_{kl}
-  \sum_{i\geq 0  >j } a_{ij}
- \sum_{i\geq  r + 1 >j } a_{ij}
\Big),
\qquad
\text{ for } A \in \Xi_{n,d}.
\end{align}
As shown in \cite[(4.1.1), Lemma~4.1.1]{FLLLW},
$d_A$ \eqref{eqda} is the dimension of a $generalized$ affine Schubert variety $X^L_A$.
One can show the two different definitions of $d_A$ in \eqref{def:dA} and \eqref{eqda}
coincide. (This fact is not used in this paper, a proof of which can be found in Proposition~\ref{eq:l^Csym} in the Appendix.)

\prop\label{prop:Abar}
Let $A =\kappa(\ld,g,\mu)\in \Xi_{n,d}$. Then we have
\enu
\item
$d_A =\ell(g_{\ld\mu}^+) - \ell(w_\circ^\mu)$;

\item
$\={[A]}  \in [A] + \sum_{\substack{y\in \D_{\ld\mu} \\ y< g}} \mA\, [\kappa(\ld,y,\mu)].$
\endenu
\endprop
%-----------------------------------------------------------------------------------------------------------
\proof
Set $\delta = \delta(A) =(\delta_0, \delta_1, \ldots, \delta_{\mathfrak{r}+1})$; see \eqref{eq:delta}.
Recall $\ell(A) =\ell(g)$ and \eqref{eq:l(A)}.
Hence  by applying Lemma~\ref{lem:Cur}(a) and then Proposition~\ref{prop:l(A)} we have
\[
\bA{
\ell(g_{\ld\mu}^+) - \ell(w_\circ^\mu) -\ell(g)
&=\ell(w_\circ^\ld) - \ell(w_\circ^{\delta})
\\
&= \ld_0^2
+ \sum_{i=1}^r {\ld_i \choose 2}
+ \ld_{r+1}^2
- \left(
	(\delta_0)^2
	+ \sum_{i=1}^{r'} {\delta_i \choose 2}
	+ (\delta_{r'+1})^2
\right)
\\
&= 2\sum\limits_{0\leq j<x} a'_{0j}a_{0x}
	+ \sum\limits_{j>0} {a_{0j} +1 \choose 2}
	+ \sum\limits_{\substack{j<y\\ 1\leq i \leq r}} a_{ij} a_{iy}
\\
&
\quad	+2 \sum\limits_{j<x \leq r+1}a_{r+1,j}a'_{r+1,x}
	+ \sum\limits_{j < r+1} {a_{r+1,j} +1 \choose 2}
\\
&=\frac{1}{2}
\bigg(
\sum\limits_{(i,j) \in I^+ }
\bp{
\sum\limits_{\substack{x = i\\ y > j}}
+
\sum\limits_{\substack{x = i\\ y < j}}
}
a'_{ij} a_{xy}
\bigg)
\\
&= d_A - \ell(A)  = d_A - \ell(g).
}
\]
Hence we have
$d_A = \ell(g^+_{\ld\mu}) - \ell(w_\circ^\mu)$, proving (1).

Equations \eqref{eq:eA}--\eqref{eq:eAbar} can be rewritten as
\begin{eqnarray}  \label{eq:AAbar}
[A](C'_{w_\circ^\mu})
&\ds =  C'_{g^+_{\ld\mu}}
 + \sum_{\substack{y\in\D_{\ld\mu}\\y < g}}
v^{-\ell(g^+_{\ld\mu})} c_{y,g}^{(\ld,\mu)} C'_{y^+_{\ld\mu}}, \label{eq:[A](C)}
\\
\={[A]}(C'_{w_\circ^\mu})
&\ds =  C'_{g^+_{\ld\mu}}
+ \sum_{\substack{y\in\D_{\ld\mu}\\y < g}}
v^{\ell(g^+_{\ld\mu})}\={c_{y,g}^{(\ld,\mu)}} C'_{y^+_{\ld\mu}}. \label{eq:=[A](C)}
\end{eqnarray}
Inverting Equation \eqref{eq:AAbar}, we see that
$C'_{g^+_{\ld\mu}} \in
[A](C'_{w_\circ^\mu}) + \sum_{\substack{y\in \D_{\ld\mu} \\ y< g}} v^{-1}\mA\, [\kappa(\ld,y,\mu)] (C'_{w_\circ^\mu}).$
Now Part~(2) %for arbitrary $g$
follows by from this and \eqref{eq:=[A](C)}.
\endproof

For $A\in \Xi_n$, we let
\eq\label{def:sig}
\sigma_{ij}(A) =
\ds\sum_{x\leq i, y\geq j} a_{xy}.
\endeq
Now we define a partial order $\leq_\alg$ on $\Xi_n$ by letting, for $A, B \in \Xi_n$,
\eq \label{eq:order}
A \leq_\alg B \Leftrightarrow
\roC(A) = \roC(B),\; \coC(A)=\coC(B), \; \text{and } \sigma_{ij}(A) \leq \sigma_{ij}(B), \forall i < j.
\endeq
We denote $A<_\alg B$ if $A\leq_\alg B$ and $A\neq B$.
In the following the expression ``lower terms'' represents a linear combination of smaller elements with respect to $\leq_\alg$.
%-----------------------------------------------------------------------------------------------------------
\lem\label{lem:PO}
Assume that $A = \kappa(\ld, g, \mu)$ and $B = \kappa(\ld, h, \mu)$. If $h \leq g$  then $B \leq_\alg A$.
\endlem
%-----------------------------------------------------------------------------------------------------------
\proof
By \cite[Theorem~8.4.8]{BB05}, the Bruhat ordering $h \leq g$ is equivalent to that
$
h[s,t] \leq g[s,t] \textup{ for all } s,t \in \ZZ,
$
where $g[s,t] = |\{a \in \ZZ \ |\ a \leq s, g(a) \geq t\}|$.
Thanks to the bijections
$R_x^\ld \cap gR_y^\mu \leftrightarrow \{a \in R_y^\mu \ |\ g(a) \in R_x^\ld \}$ for $x,y \in\ZZ$,
we have, for $i<j$,
\[
\sigma_{ij}(A)
= \sum_{x\leq  i, y \geq j} a_{xy}
= \sum_{x\geq -i, y \leq -j}|R_x^\ld \cap gR_y^\mu|
=g[s,t],
\]
where $s$ is the largest element in $R^\ld_{-i}$ and $t$ is the smallest element in $R^\mu_{-j}$.
Therefore, $\sig_{ij}(B) = h[s,t] \leq g[s,t] = \sig_{ij}(A)$, for all $s,t$, and so $B \leq_\alg A$.
\endproof

%-----------------------------------------------------------------------------------------------------------
\begin{cor}\label{cor:Abar}
For any $A \in \Xi_{n,d}$, we have
%there exists $\gamma_{A',A}  \in \mA$ for each $A' <_a A$ such that
$
\={[A]} = [A] + \textup{lower terms}.%\sum_{A' <_a A} \gamma_{A',A} [A'].
$
\end{cor}
%-----------------------------------------------------------------------------------------------------------
\proof
It follows by combining Proposition~\ref{prop:Abar} and Lemma~\ref{lem:PO}.
\endproof

%By Theorem 0.28 in \cite{DDPW}, we have the following corollary.
%\begin{cor}
%   There exists a basis $\{\{A\}\ |\  A \in \Xi_{n,d}\}$ satisfies that
%
%   (1) $\{A\} = [A] + \sum_{B < A} P_{B,A}[B]$, where $P_{B,A} \in v^{-1}\ZZ[v^{-1}]$ for any $B<_{\rm alg}A$;
%
%   (2) $\overline{\{ A \}} = \{A\}$.
%\end{cor}
%
%The basis $\{\{A\}\ |\ A\in \Xi_{n,d}\}$ is called the canonical basis.
%Moreover, by Proposition 4.2.3 in \cite{FLLLW}, the structure constants of $\mathbf S_{n,d}^{\fc}$ with respect to the canonical basis are in $\mathbb N[v,v^{-1}]$.

%=========================================================
\section{Multiplication formula using $[A]$}

Let us reformulate the multiplication formula for $\Sjj$ (Theorem \ref{thm:multformula})  in
terms of the standard basis. % (see \cite[Lemma~3.4(a2)]{BLM90}). %, \cite[Proposition~3.7]{BKLW}, \cite{DF14}).
\begin{thm} \label{thm:multformula2}
Let $A, B \in \Xi_{n,d}$ with $B$  tridiagonal and $\roC(A)=\coC(B)$.
Then we have
\eq\label{eq:mult3}
[B]\, [A] = \sum_{\substack{T \in \Tt_{B,A}\\ S \in \Gamma_T}} v^{\beta(A,S,T)} (v^2-1)^{n(S)} \={\LR{A;S;T}}~ [A^{(T - S)}],
\endeq
where $\beta(A,S,T)$ is given in \eqref{eq:beta(A,S,T)} below.
\end{thm}

The explicit formula for $\beta(A,S,T)$ can be derived as follows.
Let $\gamma(A,S,T)$  be the integer such that
\eq\label{def:gamma(A,S,T)}
\={\LR{A;S;T}} = v^{\gamma(A,S,T)} \LR{A;S;T}.
\endeq
%-----------------------------------------------------------------------------------------------------------
\lem \label{lem:beta}
%Let $A \in \Xi, T \in \Tt_{B,A}$ for some tridiagonal matrix $B$, and $S \in \Gamma_T$
%(see \eqref{def:Ia} for $\Ia$, \eqref{def:a'} for $a'_{ij}$, \eqref{def:dagger} for $\dagger$).
%Then:
%\enua
%\item $\={[A]^!} = v^{-\sum\limits_{i=1}^N \sum\limits_{j\in\ZZ} {a_{ij} \choose 2}} [A]^! $.
%\item
%$\={[A]^!_\fc} = v^{-(a'_{00})^2 - (a'_{r+1,r+1})^2 - \sum\limits_{(i,j) \in \Ia} {a_{ij} \choose 2}} [A]^!_\fc $,
%%\item ${x+y \choose 2} - {x \choose 2} = \dfrac{1}{2} y(y+2x-1)$.
%\item $\={\lrb{x}{y}} = v^{y(y-x)}\lrb{x}{y}$.
%\item $\={\LR{S}} = v^{\sum\limits_{i=1}^{r+1} \sum\limits_{j\in\ZZ} s^\dagger_{i,j+1} (s^\dagger_{i,j+1} - \sum\limits_{k\leq j} (S-S^\dagger)_{ik}) -{ s^\dagger_{i,j+1} \choose 2}}\LR{S}$.
%\item $\={\LR{A;S;T}} = v^{\gamma(A,S,T)} \LR{A;S;T}$, where
%\[
%\bA{
%\gamma(A,S,T)
%&
%= -\sum_{(i,j)\in \Ia} (s_{\tt,ij} + (\^{T-S})_{\tt,ij})(s_{\tt,ij} + (\^{T-S})_{\tt,ij} + 2a_{ij} - 2 t_{\tt,ij} - 1 )
%	\\
%&-\sum_{k\in\{0,r+1\}}(a_{kk}-1 -  (T-S)_{\tt,kk} + (\^{T-S})_{\tt,kk} )(s_{\tt,kk} +(\^{T-S})_{\tt,kk})
%	\\
%& + 2\sum_{i=1}^N \sum_{j\in \ZZ} \left( {(T-S) \choose 2} + {s_{ij} \choose 2} \right)
%	\\
%& +2\sum_{i=1}^{r+1} \sum_{j\in\ZZ} s^\dagger_{i,j+1} (s^\dagger_{i,j+1} - \sum_{k\leq j} (S-S^\dagger)_{ik}) -{ s^\dagger_{i,j+1} \choose 2}.
%}
%\]
%\endenua
%In particular, let $d'(A) = 2\ell(A) - d_A$, we have
%\[
%\beta(A,S,T)
%= d'_B +d'_A - d'_{A^{(T - S)}} +\ell(w_{A,T})+ \gamma(A,S,T).
%\]
Let $A, B \in \Xi_{n,d}$ with $B$ tridiagonal. Let $T \in \Tt_{B,A}$ and $S \in \Gamma_T$.
Then
\eq\label{eq:gamma(A,S,T)}
\bA{
\gamma(A,S,T)
&
= -\sum_{(i,j)\in \Ia} \Big(s_{\tt,ij} + (\^{T-S})_{\tt,ij} \Big) \Big(s_{\tt,ij} + (\^{T-S})_{\tt,ij} + 2a_{ij} - 2 t_{\tt,ij} - 1 \Big)
	\\
&-\sum_{k\in\{0,r+1\}} \Big(a_{kk}-1 -  (T-S)_{\tt,kk} + (\^{T-S})_{\tt,kk} \Big) \Big(s_{\tt,kk} +(\^{T-S})_{\tt,kk} \Big)
	\\
& + 2\sum_{i=1}^n \sum_{j\in \ZZ} \left( {(T-S) \choose 2} + {s_{ij} \choose 2} \right)
	\\
& +2\sum_{i=1}^{r+1} \sum_{j\in\ZZ} s^\dagger_{i,j+1}
 \big(s^\dagger_{i,j+1} - \sum_{k\leq j} (S-S^\dagger)_{ik} \big) -{ s^\dagger_{i,j+1} \choose 2}.
}
\endeq
In particular, by setting $d'(A) = 2\ell(A) - d_A$, we have
\eq\label{eq:beta(A,S,T)}
\beta(A,S,T)
= d'(B) +d'(A) - d'({A^{(T - S)}}) + 2 \ell(w_{A,T}) -2n(S) -2 h(S,T) + \gamma(A,S,T).
\endeq
\endlem
%-----------------------------------------------------------------------------------------------------------
\proof
By a direct computation using \eqref{Tfactorial}, we have
\[
\bA{
\={[A]^!}
&= v^{-2\sum\limits_{i=1}^n \sum\limits_{j\in\ZZ} {a_{ij} \choose 2}} [A]^!,
\\
\={[A]^!_\fc}
&= v^{-2(a'_{00})^2 - 2(a'_{r+1,r+1})^2 - 2\sum\limits_{(i,j) \in \Ia} {a_{ij} \choose 2}} [A]^!_\fc,
\\
%\item ${x+y \choose 2} - {x \choose 2} = \dfrac{1}{2} y(y+2x-1)$.
%\={\lrb{x}{y}}
%& = v^{y(y-x)}\lrb{x}{y},
%\\
\={\LR{S}}
&= v^{2\sum\limits_{i=1}^{r+1} \sum\limits_{j\in\ZZ} s^\dagger_{i,j+1}
 \big(s^\dagger_{i,j+1} - \sum\limits_{k\leq j} (S-S^\dagger)_{ik} \big)
 - 2 { s^\dagger_{i,j+1} \choose 2}}\LR{S}.
}
\]
The lemma follows by putting them together.
\endproof

\section{The canonical basis for $\Sjj$}

Recall $C'_{w^{\mu}_0}$ and $C'_{g^+_{\lambda\mu}}$ from Chapter \ref{secbar}.
For any $A = \kappa(\ld, g, \mu) \in \Xi_{n,d}$, following \cite{Du92} we define
 \[
   \{A\} \in  \Hom_\bH(x_\mu\bH, x_\ld\bH), \text{  and hence } \{A\} \in \Sjj,
 \]
   by requiring
\begin{equation}
  \label{canonicalbasis}
  \{A\} (C'_{w^{\mu}_0}) = C'_{g^+_{\lambda\mu}}.
\end{equation}
  By the definition of $\ \bar{}\ $,  we have $ \overline{ \{A\}} (C'_{w^{\mu}_0}) = C'_{g^+_{\lambda\mu}}$, and hence
$ %   \begin{equation}\label{barA}
     \overline{\{A\}} = \{A\}.
$ %   \end{equation}

%It follows by \eqref{eq:[A](C)} that $ [A] \in \{A\} +
%\sum_{w< g} v^{-1} \mA  \, \{\kappa(\lambda, w, \mu)\},$  for any $A = \kappa(\ld, g, \mu) \in \Xi_{n,d}$.
%Inverting the relations, we have
Following \cite[(2.c), Lemma~3.8]{Du92}, we have
\[
\{A\} =  [A] + \sum_{y< g}  v^{\ell(y^+_{\ld\mu}) -\ell(g^+_{\ld\mu})} P_{y^+_{\ld\mu},g^+_{\ld\mu}}  \, [\kappa(\lambda, y, \mu)].
\]
Recalling (C2) in \S\ref{secbar}, we have
\eq \label{eq:CB:AA}
  \{A\}  \in [A] + \sum_{y< g} v^{-1}\ZZ[v^{-1}]   \, [\kappa(\lambda, y, \mu)].
  %\sum_{w< g} v^{l(g^+_{\lambda\mu})} \alpha_{w,g} [\kappa(\lambda, w, \mu)],
\endeq
%where $\alpha_{w,g}$ are entries of the matrix $ (C^{\lambda, \mu}_{w,g})^{-1}$.
By \eqref{eq:[A](C)}, $\{\{A\}\ |\ A \in \Xi_{n,d}\}$ forms a basis of $\mathbf S^{\fc}_{n,d}$, which is called {\it the canonical basis}.
We summarize this as follows.

\begin{thm} \label{thm:CB-Sjj}
There exists a canonical basis $\mathfrak B_{n,d}^{\fc} := \{\{A\}\ |\ A \in \Xi_{n,d}\}$ for $\Sjj$,
which is characterized by the bar-invariance and the property \eqref{eq:CB:AA}.
\end{thm}

%Inverting Lemma~\ref{lem:Cur}(c),  we have
%\[
%C'_{g^+_{\ld\mu}}
%= \sum\limits_{w\in \D_{\ld\mu}} d^{(\ld,\mu)}_{w,g} T_{W_\ld w W_\mu},
%\qquad \text{for some } d^{(\ld,\mu)}_{w,g} \in \mA.
%\]
%Then we have
%\[
%\{A\} = v^{\ell(\wo^\mu)} \sum\limits_{w\in \D_{\ld\mu}} d^{(\ld,\mu)}_{w,g}e_{\kappa(\ld, w, \mu)},
%\]
%as we readily verifies that both sides satisfy the same condition \eqref{canonicalbasis}.

Recall from \cite{FLLLW} that there is a bar involution on $\Sjjg$ and a (geometric) partial ordering
$\leq_{\textup{geo}}$ on $\Xi_{n,d}$. It was shown {\em loc. cit.} that there exists
a canonical basis of $\Sjjg$, $\{\{A\}^{\textup{geo}} \ |\ A\in \Xi_{n,d}\}$  (note
$\{A\}^{\textup{geo}}$ was denoted as $\{A\}_d$ in \cite[\S4.2]{FLLLW}).
By definition, $\{A\}^{\textup{geo}}$ is bar invariant, and
$\{A\}^{\textup{geo}}  \in [A]^{\textup{geo}} + \sum_{B <_{\textup{geo}} A} v^{-1}\ZZ[v^{-1}] \, [B]^{\textup{geo}}$.
Recall the algebra isomorphism $\psi: \Sjjg \rightarrow \Sjj$ from Proposition~ \ref{Schur-iso}.

\begin{prop}  \label{prop:psiCB}
  The isomorphism $\psi: \Sjjg \rightarrow \Sjj$ commutes with the bar maps and preserves the canonical bases; that is,
  $\psi$ sends $\{A\}^{\textup{geo}}$ to $\{A\}$ for $A\in \Xi_{n,d}$.
\end{prop}

\proof
The bar involution on $\Sjj$ is simply an algebraic reformulation of the bar involution on $\Sjjg$ in \cite{FLLLW}
(which goes back to \cite{BLM90}; see \cite{Du92}).
Hence the bar operations on $\Sjj$ and $\Sjjg$ are compatible, i.e., $\bar{\ } \circ \psi = \psi \circ \bar{\ }$.

We note here a crucial fact that $A\leq_{\textup{geo}} B$ implies $A\leq_{\textup{alg}} B$
(cf. \cite[\S4.2]{FLLLW}, where $\leq_{\textup{alg}}$ is denoted by $\leq$;
the counterparts in finite type A and in the affine type A can be found in \cite{BLM90, Lu99}).
Hence $\psi (\{A\}^{\textup{geo}})$ is bar invariant and
$\psi (\{A\}^{\textup{geo}}) \in [A]+ \sum_{B <_{\textup{alg}} A} v^{-1}\ZZ[v^{-1}] \,  [B]$.
Hence $\psi (\{A\}^{\textup{geo}}) = \{A\}$ by the characterization of $\{A\}$.
\endproof

%%%%%%%%%%%
%-----------------------------------------------------------------------------------------------------------
\section{A leading term}
  \label{sec:leading}

A pair $(B,A)$ of matrices in $\Xi_n\times \Xi_n$ is called \textit{admissible} if the following conditions (a)--(b) hold:
\enu
\item[(a)]  $B^\offd = \sum\limits_{i=1}^n b_{i,i+1} \Ett^{i,i+1}$ (see \eqref{def:offd} for notation $B^\offd$);
\item[(b)] $A^\offd = \sum\limits_{j=1}^x \sum\limits_{i = 1}^n a_{i,i+j} \Ett^{i,i+j}$
for some $x\in \NN$,
where $a_{i,i+x} \geq b_{i,i+1}$ for all $i$.  (If $a_{i,i+x} \neq 0$ for some $i$, we say that \emph{$A$ is of depth $x$}.)
\endenu

Assume that $A,B \in \Xi_{n,d}$ with $B$ tridiagonal and $\roC(A) = \coC(B)$.
We produce a matrix
\eq \label{eq:MBA}
M(B,A) \in \Xi_{n,d}
\endeq
as follows.
For each row $i$, find the unique $j$ such that $b_{i,i+1} \in \big{(} \sum_{y>j} a_{iy} .. \sum_{y\geq j} a_{iy} \big{]}$,
and denote
\eq  \label{eq:TBA}
\Tp =\Tp_{B,A}  = \sum_{i=1}^{n} \Big( (b_{i,i+1} - \sum_{y>j} a_{iy}) \Ett^{ij} + \sum_{y>j} a_{iy} \Ett^{iy} \Big) \in \Tt_n.
\endeq
Set
$M(B,A) = A^{(\Tp)};\text{ see  } \eqref{def:A^TS}.$
\endAlg
%-----------------------------------------------------------------------------------------------------------
\lem\label{lem:ht}
Assume that $A,B \in \Xi_{n,d}$ with $B$ tridiagonal and $\roC(A) = \coC(B)$.
The highest term (with respect to $\leq_\alg$) on the RHS of \eqref{eq:mult3} exists
and is equal to $[M(B,A)]$ (with coefficient being some nonzero scalar in $\mA$).
\endlem
%-----------------------------------------------------------------------------------------------------------
\proof
Note that
\eq\label{eq:sigA}
\sigma_{ij}(A^{(E^{xy})})
=\bc{
\sigma_{ij}(A)+1&\tif j< i=x-1, j\leq y,
\\
\sigma_{ij}(A)-1&\tif j>i=x, j\geq y,
\\
\sigma_{ij}(A)&\otw.
}
\endeq
It follows that
$
A^{(E^{ij})} <_\alg A^{(E^{i,j+1})}
$
for all
$i,j \in\ZZ$.
Therefore, for any $T\in\Tt_{B,A}, S\in \Gamma_T$ we have $A^{(T-S)} \leq_\alg A^{(T)} \leq_\alg A^{(\Tp)}= M(B,A)$.
\endproof
%-----------------------------------------------------------------------------------------------------------
The corollary below is a direct consequence of \eqref{eq:sigA}.
%-----------------------------------------------------------------------------------------------------------
\begin{cor}\label{cor:lower}
Let $A', A, B \in \Xi_{n,d}$, with $B$ tridiagonal.
If $A' <_\alg A$, then $M(B,A') <_\alg M(B,A)$.
\end{cor}

Recall $\LR{A;S;T}$ from \eqref{def:[AST]}.
\lem\label{lem:M}
Let $A,B \in \Xi_{n}$.
If $(B,A)$ is admissible, then $\LR{A;0;\Tp_{B,A} } =1$.
%\[
%e_B e_A = v^{\ell(A,T_+)} e_M + \textup{lower terms}.
%\]
\endlem
%%-----------------------------------------------------------------------------------------------------------
\proof
Write $\Tp =\Tp_{B,A}$, cf. \eqref{eq:TBA}.
As $(B,A)$ is admissible,  we have  $\Tp =  \sum\limits_{i=1}^{n} b_{i,i+1} \Ett^{i,i+k}$.
Moreover, we have
\[
\LR{A;0;\Tp}
=
\dfrac{1}{[\Tp]^!}
\dfrac{[A + \^{\Tp}^+_\tt - \Tp_\tt]^!_\fc}
{[A-\Tp_\tt]^!_\fc}
=
\frac{1}{\prod_{i=1}^n \lr{m_i}}
\cdot
\prod_{i=1}^n \lr{m_i}
 = 1.
\]
The lemma is proved.
\endproof

%-----------------------------------------------------------------------------------------------------------
\lem \label{lem:B'}
Let $A,B \in \Xi_{n,d}$ with $B$ tridiagonal and $\roC(A) = \coC(B)$.
If $B' <_\alg B$, then $B'$ is also tridiagonal.
Moreover, we have %$(B')^\offd < B^\offd$, and
$M(B', A)  <_\alg M(B,A).$
\endlem

\proof
Since $ B' \leq_\alg B$, we have
\[
\sig_{i,i+2}(B') \leq \sig_{i,i+2}(B) = 0
\quad
\tfor
\quad
i = 1, \ldots, n.
\]
Therefore $\sig_{i,i+2}(B') = 0$ for all $i$ and hence $B'$ is tridiagonal.
Also, we have
\[
\sig_{i,i+1}(B') = b'_{i,i+1} \leq \sig_{i,i+1}(B) = b_{i,i+1}
\quad
\tfor
\quad
i = 1, \ldots, n.
\]
It follows from this and the definition \eqref{eq:MBA} that $M(B', A) <_\alg M(B,A).$
\endproof

%-----------------------------------------------------------------------------------------------------------
By Lemma \ref{lem:M}, we can rewrite \eqref{eq:mult3} as
\eq\label{eq:BAM}
[B]   [A]
=
v^{\beta(A,0,\Tp)}[M(B,A)]+ \textup{lower terms}.
\endeq
One can show that $\beta(A,0,\Tp) = 0$ by a direct but lengthy computation. Here we present a
formal argument of this fact via the bar involution.

\lem \label{lem:ss}
If $(B,A)$ is admissible (and we retain the notation in Conditions (a)--(b) in \S\ref{sec:leading}), then $\beta(A,0,\Tp) = 0$. Hence, we have
\[
[B] [A] = [M(B,A)] + \textup{ lower terms}.
\]
Moreover, we have $M(B,A)=A-\sum_{i=1}^nb_{i,i+1}(E_\theta^{i,i+x}-E_\theta^{i-1,i+x})$.
\endlem
%-----------------------------------------------------------------------------------------------------------
\proof
Write $M = M(B,A)$, cf. \eqref{eq:MBA}. By taking bar on both sides of \eqref{eq:BAM}, we obtain
\eq \label{BA1}
\={[B]} \, \={[A]} = v^{-\beta(A,0,\Tp)}\={[M]}
+ \textup{lower terms}.
\endeq
By Proposition ~\ref{prop:Abar}, we can write
\eq  \label{BBAA}
\={[B]} =
[B] + \sum_{B' <_\alg B} \gamma_{B,B'} [B'],
\qquad
\={[A]}
=
[A] + \sum_{A' <_\alg A} \gamma_{A,A'} [A'].
\endeq
For any $B' <_\alg B$, by Lemma \ref{lem:B'} and Corollary \ref{cor:lower} we have
\[
M(B',A') <_\alg M(B',A) <_\alg M.
\]
Therefore, by \eqref{BBAA} and \eqref{eq:BAM} we have
\eq \label{BA2}
\={[B]} \, \={[A]}
=
[B]\,[A] + \textup{lower terms}
=
v^{\beta(A,0,\Tp)}[M]
+ \textup{lower terms}.
\endeq
By comparing the leading coefficients in \eqref{BA1} and \eqref{BA2}, we obtain   $\beta(A,0,\Tp) =0$.

The formula for $M(B,A)$ follows by definition of \eqref{eq:MBA}.
\endproof

%=========================================================
\section{A semi-monomial basis}

Below we provide an algorithm that constructs a monomial basis in a diagonal-by-diagonal manner involving only admissible pairs.
A similar algorithm for monomial basis in affine type A was given in \cite{LL17}.
%-----------------------------------------------------------------------------------------------------------
\Alg\label{alg:mono}
For each $A=(a_{ij}) \in \Xi_{n,d}$,
we construct tridiagonal matrices $A^{(1)}, \ldots, A^{(x)} \in \Xi_{n,d}$ as follows:
\enu
\item Initialization: set $t=0$, and set $C^{(0)} = A$.
\item If $C^{(t)}\in \Xi_{n,d}$ is a tridiagonal matrix, then set $x=t+1$,
$
A^{(x)} = C^{(t)}, %(C^{(t)})^\offd + \textup{a diagonal determined by } \eqref{eq:diag},
$
and the algorithm terminates.

Otherwise, for $C^{(t)} =(c^{(t)}_{ij}) \in \Xi_{n,d}$,
set $k= \max \big\{ |i-j| ~|~ c^{(t)}_{ij} \neq 0 \big \} \geq 2$,
%Set $k \in \NN$ to be the smallest positive integer such that $c^{(t)}_{ij} = 0$ if $|i-j|>k$,
and
\[
\Tp^{(t)}= \sum_{i=1}^n c^{(t)}_{i,i+k} E^{i,i+k}.
\]
%to be the matrix consisting of the outermost nonzero diagonals of $C^{(t)}$.
\item
Define matrices
\[
\bA{
&\ds A^{(t+1)} = \sum_{i=1}^n c^{(t)}_{i,i+k} \Ett^{i,i+1} + \textup{a diagonal determined by } \eqref{eq:diag},
\\
&\ds C^{(t+1)}= C^{(t)} - \Tp^{(t)}_\tt + {\V{\Tp^{(t)}} }_\tt,
}
\]
where $\V{X}$ is the matrix obtained from $X$ by shifting every entry down by one row; also see \eqref{Ttt} for notation.
%(Note that $\max \big\{ |i-j| ~|~ c^{(t+1)}_{ij} \neq 0 \big\} \leq k-1$.)

\item Increase $t$ by one and then go to Step (2).
\endenu
Here the diagonal entries of $A^{(t)}$ are uniquely determined by
\eq\label{eq:diag}
\roA(A^{(1)}) = \roA(A),
\quad
\roA(A^{(t+1)}) = \coA(A^{(t)}),
\quad
1\leq t\leq x-1.
\endeq
\endAlg
%-----------------------------------------------------------------------------------------------------------
\thm\label{thm:min gen}
{\quad}
\enu
\item
For each $A \in \Xi_{n,d}$, the tridiagonal matrices $A^{(1)}, \ldots, A^{(x)} \in \Xi_{n,d}$
in Algorithm~\ref{alg:mono} satisfy that
\eq \label{eq:semiMB}
m'_A :=
[ A^{(1)}] [ A^{(2)}]\cdots [ A^{(x)}]
= [A] +\textup{lower terms}.
\endeq

\item
The set $\{m'_A ~|~ A \in \Xi_{n,d}\}$
forms an $\mA$-basis of $\Sjj$ (called a semi-monomial basis).
\item
The algebra $\Sjj$ is generated by the $[A]$, for tridiagonal matrices $A \in \Xi_{n,d}$.
\endenu
\endthm
%-----------------------------------------------------------------------------------------------------------
\proof
Algorithm \ref{alg:mono} guarantees that each pair $(A^{(j)}, C^{(j)})$
is admissible, and moreover, $M(A^{(j)}, C^{(j)}) = C^{(j-1)}$ for  $2\le j \le x$.
By construction we have $A^{(x)} = C^{(x-1)}$.
 Hence by Corollary \ref{cor:lower} and by Lemma \ref{lem:ss}, we have
\begin{align*}
[ A^{(1)}] [ A^{(2)}] \cdots \bp{ [ A^{(x-1)}] [ A^{(x)}]}
&=
[ A^{(1)}] [ A^{(2)}]\cdots
\bp{[ A^{(x-1)}] [ C^{(x-1)}]}
\\
&=
[ A^{(1)}] [ A^{(2)}] \cdots [ A^{(x-2)}]
\bp{ [ C^{(x-2)}] + \textup{lower terms}
}
\\
&= [A] + \textup{lower terms}.
\end{align*}
This proves (1).
Part (2) follows from (1) as the transition matrix between the standard basis
$\{[A]~|~ A \in \Xi_{n,d}\}$ and $\{m_A' ~|~ A \in \Xi_{n,d}\}$ is uni-lower-triangular.
Part~(3) follows from (2).
The theorem is proved.
\endproof

In general, the element $m'_A$ is not bar-invariant since for an arbitrary tridiagonal matrix $B$, $[B]$ is not necessarily bar-invariant.

%===========
\section{ A monomial basis for $\Sjj$ }
Recall from \eqref{eq:MBA} that $M(B,A)$ is defined for admissible pairs. We now extend the definition to arbitrary pairs in $\Xi_n \times \Xi_n$ using the semi-monomial bases.
Let $B^{(1)}, \ldots, B^{(x)} \in \Xi_{n}$ be the tridiagonal matrices in
Algorithm \ref{alg:mono} and Theorem~\ref{thm:min gen} (now associated to $B$)
satisfying
\eq\label{eq:ht2Bi}
[ B^{(1)}] [ B^{(2)}] \cdots [ B^{(x)}] = [B] +\textup{lower terms}.
\endeq
We define matrices
\eq  \label{MBA:general}
M^{(x+1)} = A, \quad
%$M^{(x)} = M(B^{(x)}, A)$,
M^{(i)} = M(B^{(i)}, M^{(i+1)}) \; \text{ for } 1\leq i \leq x,
\endeq
and then set
\eq \label{MBAM1}
M(B,A) = M^{(1)}.
\endeq
We have the following generalization of Corollary~ \ref{cor:lower} and Lemma~\ref{lem:B'}.

\lem\label{lem:lower2}
Let $A', A, B, B' \in \Xi_{n,d}$.
\enua
\item
If $A' <_\alg A$, then $M(B,A') <_\alg M(B,A)$.
\item
If $B' <_\alg B$, then $M(B',A) <_\alg M(B,A)$.
\endenua
\endlem

\proof
Part (a) follows from applying Corollary~ \ref{cor:lower} repeatedly.

(b) Assume now $B' <_\alg B$.
We keep the notations in \eqref{eq:ht2Bi}, \eqref{MBA:general} and \eqref{MBAM1}.
Again, Algorithm \ref{alg:mono} produces tridiagonal matrices ${B'}^{(j)}$ for $1\le j \le y$ satisfying
\eq
[ {B'}^{(1)}]
[ {B'}^{(2)}]
\cdots
[ {B'}^{(y)}]
= [B'] +\textup{lower terms}.
\endeq
It follows from the construction of semi-monomial basis that $y \leq x$,
and we define matrices ${B'}^{(i)}$ for $y< i\leq x$ to be the diagonal matrix with $\coC({B'}^{(i)}) = \roC(A)$.
Therefore, we have ${B'}^{(i)} \leq_\alg B^{(i)}$ for all $i$, and the strict inequality holds for some $i_0 \ge 1$.
Following \eqref{MBA:general}, we define
\[
{M'}^{(x+1)} = A, \quad
{M'}^{(i)} = M({B'}^{(i)}, {M'}^{(i+1)}) \; \text{ for } 1\leq i \leq x.
\]
It follows by definition \eqref{MBAM1} that
$M(B',A) = {M'}^{(1)}.$
We claim that
\eq \label{Mi}
{M'}^{(i)} \leq_\alg M^{(i)}, \forall 1\le i \le x+1;
\text{ moreover the inequality is strict for  }1\le i\le i_0.
\endeq
We prove \eqref{Mi} by a downward induction on $i$.
The base case of the induction is trivial as ${M'}^{(x+1)} =A\leq_\alg A =M^{(x+1)}$.
Assume ${M'}^{(i+1)} \leq_\alg M^{(i+1)}$ for some $i \ge 1.$
Then by using first Lemma~\ref{lem:B'} and then Corollary \ref{cor:lower} we have
\eq \label{Mi:ind}
{M'}^{(i)} = M({B'}^{(i)}, {M'}^{(i+1)})
\leq_\alg  M({B'}^{(i)}, M^{(i+1)})
\leq_\alg  M(B^{(i)}, M^{(i+1)})
= M^{(i)}.
\endeq
This proves the first statement in \eqref{Mi}. Now for $i= i_0$, the first inequality in \eqref{Mi:ind}
must be strict by Lemma~\ref{lem:B'}, and for $i<i_0$, the second inequality in \eqref{Mi:ind}
must be strict by Corollary \ref{cor:lower}. The proof of \eqref{Mi} is completed.

The statement \eqref{Mi} for $i=1$ gives us $M(B',A) <_\alg M(B,A)$, whence (b).
\endproof

%%%%%%%%%%%%%%%%%%%%%%%%%%%
\iffalse
%%%%%%%%%%%%%%%%%%%%%%%%%%%
\prop\label{prop:ht2}
Let $A,B \in \Xi_{n,d}$ be such that $\roC(A) = \coC(B)$.
Then the highest term (with respect to $\leq_\alg$) in the standard basis expansion of
the product $[B] [A]$ exists and is equal to $[M(B,A)]$.
\endprop
%-----------------------------------------------------------------------------------------------------------
\proof
We prove by induction on $B$ with respect to $\leq_\alg$.
When $B$ is diagonal the statement is trivial.
Now for a fixed $B$, assume the statement holds for any $B' <_\alg B$.

Retaining the notation $B^{(i)}$ as in \eqref{eq:ht2Bi}, we have
\eq  \label{BAB'}
[B] [A] \in [B^{(1)}]  \ldots  [B^{(x)}]  [A] + \sum_{B' <_\alg B} \mathcal{A} \, [B'][A].
\endeq
By Corollary \ref{cor:lower} and Lemma \ref{lem:ht}, for all $i$ we have
\[
[B^{(i)}]([M^{(i+1)}_B] + \textup{lower terms}) = c_i  [M^{(i)}_B] +  \textup{lower terms},
\quad \text{ for } c_i \in \mA \backslash \{0\}.
\]
Hence
\eq \label{B1Bx}
[B^{(1)}]  \ldots  [B^{(x)}]  [A] =c_1\cdots c_x [M(B,A)] +  \textup{lower terms}.
\endeq
By inductive assumption, we have $[B'][A] =c' [ M(B',A)] +\textup{lower terms}$, for some $c'\in \mA$.
The proposition follows by this and \eqref{BAB'}-\eqref{B1Bx}.
\endproof
%%%%%%%%%%%%%%%%%%%%%%%%%%%
\fi
%%%%%%%%%%%%%%%%%%%%%%%%%%%

Modifying \eqref{eq:semiMB}, for $A\in \Xi_n$, we define
\eq  \label{eq:MB}
m_A = \{A^{(1)}\}  \{A^{(2)}\} \cdots  \{A^{(x)}\}.
\endeq
%-----------------------------------------------------------------------------------------------------------
\begin{thm}\label{thm:mono}
The element $m_A$, for $A \in \Xi_{n,d} $,   is bar-invariant, and moreover, $m_A = [A] +$ lower terms.
Hence, the set $\{m_A~|~A \in \Xi_{n,d} \}$ forms an $\mA$-basis of $\Sjj$.
\end{thm}

\proof
By definition $m_A$ is bar invariant.
It follows by Corollary~\ref{cor:lower} and  Lemma~\ref{lem:B'} that
$m_A = [A] +$ lower terms. Hence, the set $\{m_A~|~A \in \Xi_{n,d} \}$ forms an $\mA$-basis of $\Sjj$ since
$\{[A]~|~A \in \Xi_{n,d} \}$ is a basis.
\endproof

\rmk
Traditionally, the monomial basis was introduced as an intermediate step toward construction of canonical basis.
We have reversed the order of introducing these two bases for $\Sjj$.
Algorithm~\ref{alg:mono}, which plays a fundamental role in constructing the monomial basis, will be adapted
to construct the monomial and canonical bases for the stabilization algebra arising from the family (or its variants)
of Schur algebras $\Sjj$ as $d$ varies
in the next chapters.
\endrmk

\newpage
\part{Affine quantum symmetric pairs}
  \label{part2}
%%\include{FLpart2.5}

%=========================================================
\chapter{Stabilization algebra  $\dKjj$ arising from affine Schur algebras}
  \label{sec:coideal}

In this chapter,
we shall establish a stabilization property for the family of affine Schur algebras $\Sjj$ as $d$ varies,
which leads to a stabilization algebra $\dKjj$.
The multiplication formula for $\Sjj$ with tridiagonal generators in the previous chapter plays a fundamental role.
The bar-involution on $\dKjj$ follows from a new formula of the image of the bar-involution on the tridiagonal generators.
We construct a monomial basis and a stably canonical basis for $\dKjj$.
The algebra $\dKjj$ here is obtained from a different stabilization procedure than the stabilization algebra $\dot{\bK}^{\fc, \textup{geo}}_{n}$ in \cite{FLLLW}.
We show that these two algebras are indeed isomorphic, and the canonical bases are preserved under such an isomorphism.

%=========================================================
\section{A BLM-type stabilization}
   \label{sec:stab}
%=========================================================

Let
\begin{align}
 \label{eq:Xitn}
\begin{split}
\Xit_n = \Big\{A=(a_{ij}) \in \text{Mat}_{\ZZ\times\ZZ}(\ZZ)~\big |~
& a_{-i,-j} =a_{ij}=a_{i+n, j+n}, \forall i, j\in \ZZ;
\\
a_{k\ell}\ge 0, \, \forall k\neq \ell \in \ZZ; \;  &
a_{00}, a_{r+1,r+1} \text{ are odd}
 \Big \}.
 \end{split}
\end{align}
Extending the partial ordering $\leq_\alg$  for $\Xi_{n}$,
we define a partial ordering $\leq_\alg$ on $\Xit_n$ using the same recipe
\eqref{eq:order}.

For each $A \in \Xit_n$ and $p \in\NN$, let $\p{A} = A+pI$ where $I$ is the identity matrix.
Then $\p{A} \in \Xi_n$ for even $p \gg 0$.
Let $v', v''$ be two indeterminates (independent of $v$),
and $\mathcal{R}_1$ be the subring of $\QQ(v)[v']$ generated by
\eq
\prod_{i=1}^t \dfrac{v^{2(a+i)} v'^{2} -1}{v^{2i} -1},
\quad
\prod_{i=1}^t \dfrac{v^{4(a+i)} v'^{2} -1}{v^{2i} -1},
\textup{ and }
v^a,
\quad
\text{ for } a\in\ZZ,
t\in \ZZ_{>0}.
\endeq
Let $\mathcal{R}_2$ be the subring of $\QQ(v)[v', v'^{-1}]$ generated by
\eq
\ba{{lllll}
\ds\prod_{i=1}^t \dfrac{v^{2(a+i)} v'^{2} -1}{v^{2i} -1},
&
\ds\prod_{i=1}^t \dfrac{v^{4(a+i)} v'^{2} -1}{v^{2i} -1},
\\
\ds\prod_{i=1}^t \dfrac{v^{-2(a+i)} v'^{-2} -1}{v^{-2i} -1},
&
\ds\prod_{i=1}^t \dfrac{v^{-4(a+i)} v'^{-2} -1}{v^{-2i} -1},
\textup{ and }
v^a,
&
\text{ for } a\in\ZZ,
t\in \ZZ_{>0}.
}
\endeq
Let $\mathcal{R}_3 =\mathcal{R}_2[v'', v''^{-1}]$ be a subring of $\QQ(v)[v',v'^{-1},v'',v''^{-1}]$.

%-----------------------------------------------------------------------------------------------------------
\prop\label{prop:stab1}
Let $A_1, \ldots, A_f \in \Xit_n$ be such that $\coC(A_i) = \roC(A_{i+1})$ for all $i$.
Then there exists matrices $Z_1, \ldots, Z_m \in \Xit_n$ and $G_i(v,v') \in \mathcal{R}_2$
such that for  even integer $p\gg 0 ,$
\eq\label{eq:stab1}
[\p{A_1}] [\p{A_2}]  \cdots  [\p{A_f}] = \sum_{i=1}^m G_i(v, v^{-p}) [\p{Z_i}].
\endeq
\endprop
%-----------------------------------------------------------------------------------------------------------
\proof
The proof follows the basic idea as for \cite[4.2]{BLM90}, except one step which is
more complicated in our setting (because of the complexity of the multiplication formula in Theorem~\ref{thm:multformula2});
this is also responsible for the subtle difference of presence of $\mathcal{R}_2$ instead of $\mathcal{R}_1$ in the proposition.
We will present this step in more detail.

Similar to \cite{BLM90}, for $A \in \Xit_n$, we introduce
\eq \label{eq:PsiA}
\Psi(A) = \sum_{(i,j) \in I^+} { |i-j| +1 \choose 2} a_{ij} \in \NN.
\endeq
Noting $\Psi(A)=\sum_{(i,j) \in I^+}\sigma_{ij}(A)$ (see \eqref{def:sig} for notation $\sigma_{ij}(A)$), we have
%$\Psi(A) < \Psi(B)$ if $A <_\alg B$, for $A, B \in \Xit_n$.
\[
\Psi(A) < \Psi(B) \text{ if } A <_\alg B, \text{ for } A, B \in \Xit_n.
\]

Assume for now we have proved the identity \eqref{eq:stab1} when $f=2$ and $A_1$ is tridiagonal.
By iteration inducting on $f$, \eqref{eq:stab1} holds for general $f$ when $A_i$ for all $1\le i \le f$ are tridiagonal.
 Exactly as in proof of \cite[4.2]{BLM90}, using this special case we just proved together with Theorem \ref{thm:min gen}
 we can prove \eqref{eq:stab1} when $f=2$ (and for general $A_1$) by induction on $\Psi(A_1)$.
 Then by iteration inducting on $f$, the identity \eqref{eq:stab1}  in full generality follows.

It remains to verify the identity \eqref{eq:stab1} when $f=2$ and $B =A_1$ is tridiagonal. Set $A= A_2$.
This is the step more complicated than \cite[4.2]{BLM90} which we alluded to at the beginning.
For even integer $p \gg 0$ such that all entries in $A_i$ are in $\NN$, we can apply Theorem~\ref{thm:multformula2} and obtain
\eq  \label{eq:pBA}
[\p{B}]   [\p{A}] = \sum_{\substack{T \in \Tt_{B, _pA} \\ S \in \Gamma_T}}
v^{\beta(\p{A},S,T)}
(v^2-1)^{n(S)}
\={\LR{\p{A};S;T}} \,
 [\p{A}^{(T-S)}].
\endeq
Recall $\beta(A,S,T)$ from \eqref{eq:beta(A,S,T)}, and hence
\[
\beta(\p{A},S,T) = d'(\p{B}) + d'(\p{A}) - d'(\p{A^{(T_\tt - S_\tt)}}) + 2\ell(w_{\p{A},T}) + \gamma(\p{A},S,T) -2n(S) -2 h(S,T).
\]
We compute the difference $\beta(\p{A},S,T) - \beta(A,S,T)$ term by term as follows:
\eq\label{eq:stab1eq1}
\bA{
d'(\p{B})- d'(B)
&= -p \sum\limits_{i=0}^{2r+1} b_{i,i+1} ,
\\
d'(\p{A})- d'(A)
&=\frac{p}{2}
\sum\limits_{i=0}^{2r+1}
	\Bp{
	\sum\limits_{x<i, y<i} a_{xy}
	-
	\sum\limits_{y>i} a_{iy}
	}
\\
&+\frac{p}{2}
\sum\limits_{(i,j) \in I^+}
	\Bp{
	\sum\limits_{x=1}^r \Bp{ \sum\limits_{i<x<j} + \sum\limits_{i>x>j} } a_{ij}
	-
	\sum\limits_{i\neq j} a_{ij}
	},
\\
 \ell(w_{\p{A},T}) - \ell(w_{A,T})
&= p
 \sum\limits_{i=0}^{2r+1}\sum\limits_{j>i}
t_{ij},
\\
\gamma(\p{A},S,T) - \gamma(A,S,T)
&= -p \Bp{
\sum\limits_{k \in \{0,r+1\}}(S + \^{T} - \^{S})_{\tt,kk}
-2 \sum\limits_{(i,j) \in \Ia} (S + \^{T} - \^{S})_{\tt,ij}
}.
}
\endeq
Combining these gives us
\[
\beta(\p{A},S,T) = \beta(A,S,T) + p G_1 (B,A,S,T),
\]
where $G_1(B,A,S,T)$ depends only on the entries of $B,A,S,T$ (and is independent of $p$).

On the other hand, set
\eq\label{eq:stab1eq2}
\bA{
a^{(1)}_{ij} &= p\delta_{ij}+(A-T_\tt)_{ij} + s_{-i,-j} + (\^{T-S})_{ij} + (\^{T-S})_{-i,-j},
\\
a^{(2)}_{ij} &= p\delta_{ij}+(A-T_\tt)_{ij} + (\^{T-S})_{ij} + (\^{T-S})_{-i,-j},
\\
a^{(3)}_{ij} &= p\delta_{ij}+(A-T_\tt)_{ij} + (\^{T-S})_{-i,-j},
\\
a^{(4)}_{ij} &= p\delta_{ij}+(A-T_\tt)_{ij},
\\
a^{(5)}_{kk} &= p+ a_{kk} - 2t_{kk} - 1 + 2s_{kk}\in 2\ZZ,
\\
a^{(6)}_{kk} &= p+ a_{kk} - 2t_{kk} - 1  \in 2\ZZ
.
}
\endeq
We have
\begin{equation*}
\bA{
\LR{\p{A};S;T}
&=
\prod_{(i,j) \in \Ia} \Big{(}
\prod_{l=1}^{s_{ij}} \frac{[a^{(1)}_{ij}+l]}{[l]}
\prod_{l=1}^{s_{-i,-j}} \frac{[a^{(2)}_{ij}+l]}{[l]}
\prod_{l=1}^{(\^{T-S})_{ij}} \frac{[a^{(3)}_{ij}+l]}{[l]}
\prod_{l=1}^{(\^{T-S})_{-i,-j}} \frac{[a^{(4)}_{ij}+l]}{[l]}
\Big{)}
\\
&\cdot \prod_{k \in \{0, r+1\}}
\left(
\prod_{l=1}^{(\^{T-S})_{kk}} \frac{[a^{(5)}_{kk}+2l]}{[l]}
\prod_{l=1}^{s_{kk}} \frac{[a^{(6)}_{kk}+2l]}{[l]}
\right)
\cdot  \LR{S}.
}
\end{equation*}
Hence  $v^{\beta(\p{A},S,T)} \={\LR{\p{A};S;T}}$ is of the form $G(v,v^{-p})$ for some $ G(v,v') \in \mathcal{R}_2$.

This finishes the proof of the identity \eqref{eq:stab1} when $f=2$ and $B =A_1$ is tridiagonal.
The proposition is proved.
\endproof

%\rmk
%The finite type A analog of $\beta(A,S,T)$ is referred as $\beta(\mathbf{t}) = \beta(\mathbf{t})(A)$ in \cite[3.4]{BLM90},
%which satisfies that $\beta(\mathbf{t})(A) = \beta(\mathbf{t})(\p{A})$ for all $p$.
%This leads to a stronger statement \cite[Proposition 4.2]{BLM90} that the structure constants lie in a smaller ring
%(referred as $\mathcal{R}_2$), which is not essential for the stabilization procedure.
%The analogs of $\beta$ are referred as $\beta(t)$ for finite type B/C in \cite{BKLW14}, and as $P_{T,A}$ for affine type A
%in \cite{DF14}. These functions are also unchanged when $A$ is replaced by $\p{A}$.
%\endrmk

%--------------------------------------
\section{Stabilization of bar involutions}

\prop\label{prop:stab2}
For any $A \in \Xit_n$, there exist matrices $T_1, \ldots, T_s\in \Xit_n$ and $H_i(v,v',v'') \in \mathcal{R}_3$  such that,
for  even integer $p\gg 0,$
\eq\label{eq:stab2}
\={[\p{A}]} = \sum_{i=1}^s H_i(v, v^{-p},v^{p^2}) [\p{T_i}].
\endeq
\endprop
%-----------------------------------------------------------------------------------------------------------
\proof
We follow the strategy of the proof of \cite[4.3]{BLM90} closely.
We will be sketchy on the almost identical steps, except one extra step which we will present the detail below.
We can assume without loss of generality that $A \in \Xi_n$, by replacing $A$ by ${}_{p_0}A$ if necessary.

We prove the identity
\eqref{eq:stab2} by induction on $\Psi(A)$. If $\Psi(A)=0$, then $A$ is diagonal, and $\={[\p{A}]} =  [\p{A}]$.
Assume $\Psi(A)\ge 1$.
%We proceed with an induction as in Proposition \ref{prop:stab1}, with initial case $A$ being tridiagonal.
Combining Theorem \ref{thm:min gen} and Proposition \ref{prop:stab1}, for $p \gg 0$ we have
\eq\label{eq:stab2a}
[\p{A}] = [\p{A}_1]   \cdots   [\p{A}_f] + \sum_{k} G_k(v, v^{-p},v^{p^2}) [\p{Z}_k],
\endeq
where $A_j$'s are all tridiagonal, $G_k \in \mathcal{R}_3$ and $\Psi(Z_k) < \Psi(A)$.
By inductive hypothesis, there are $T''_{i,k}\in \Xit_n$ and $H''_{i,k} \in \mathcal{R}_3$
such that $\={[\p{Z}_k]} = \sum_i H''_{i,k}(v, v^{-p},v^{p^2})[\p{T''_{i,k}}]$.

Assume for now the identity \eqref{eq:stab2} holds for all tridiagonal $A$.
Hence, there are $T'_{i,j}\in \Xit_n$ and $H'_{i,j} \in \mathcal{R}_3$ such that
$\={[\p{A}_j]} = \sum_i H'_{i,j}(v, v^{-p},v^{p^2}) [\p{T'_{i,j}}]$.
Applying the bar map to both sides of \eqref{eq:stab2a}, we have
\begin{align*}
\={[\p{A}]}
&=
\Big(\sum_i H'_{i,1}(v, v^{-p},v^{p^2}) [\p{T'_{i,1}}] \Big)   \cdots
\Big (\sum_i H'_{i,f}(v, v^{-p},v^{p^2}) [\p{T'_{i,f}}] \Big)
\\
&\qquad \qquad + \sum_{i,k} \overline{G}_k(v, v^{-p},v^{p^2}) H''_{i,k}(v, v^{-p},v^{p^2}) [\p{T''_{i,k}}].
\end{align*}
Now the identity \eqref{eq:stab2} in full generality follows from this by Proposition~ \ref{prop:stab1}.

It remains to prove the identity \eqref{eq:stab2} when $A$ is tridiagonal, which is the extra step we alluded to above.
(In the setting of \cite[4.3]{BLM90}, this step is trivial where $A$ is bidiagonal, and $\={[\p{A}]} =  [\p{A}]$.)

For $\ld \in \Ld_{r,d}$ we define $\p{\ld} \in \Ld_{r,d+p(r+1)}$ by
\[
\p{\ld} = \Bp{\ld_0 + \frac{p}{2}, \ld_1+p, \ldots, \ld_r+p, \ld_{r+1} + \frac{p}{2} }.
\]
We fix a tridiagonal matrix $A = \kappa(\ld,g,\mu) \in \Xi_n$ for some $\ld,\mu \in \Ld_{r,d}$, $g\in \D_{\ld\mu} \subseteq W(d)$.
(Here we have denoted the affine Weyl group $W$ as $W(d)$ to indicate its dependence on $d$; we shall see
$\p{g} \in W(d+p(r+1))$ below.)

Recall from \eqref{g1i}  $g = \prod_{i=0}^r g^{(i)}$, where
$g^{(i)} \in \textup{Stab}(R^{\delta(A)}_{3i+1} \cup R^{\delta(A)}_{3i+2})$ has a reduced expression
\[
g^{(i)}
=
(s_{m_i+\beta_i}\cdots s_{m_i+2}s_{m_i+1})
(s_{m_i+\beta_i+1}\cdots s_{m_i+2})
\cdots
(s_{m_i+\beta_i+\alpha_i-1}\cdots s_{m_i+\alpha_i}),
\]
where $\af_i \ge 0, \beta_i \ge 0$ are such that
$R^{\delta(A)}_{3i-2} = (m_i .. m_i+\af_i]$ and $R^{\delta(A)}_{3i-1} = (m_i+\af_i .. m_i+\af_i+\beta_i]$.
Hence $\p{A} = \kappa(\p{\ld}, \p{g}, \p{\mu})$ where $\p{g} \in W(d+p(r+1))$ is uniquely determined by letting,
for $1 \leq x \leq d+p(r+1)$,
\[
\p{g}(x) =
\bc{
 x + a_{i-1,i} &\tif x\in R^{\delta(\p{A})}_{3i+1},
\\
 x - a_{i,i-1} &\tif x \in R^{\delta(\p{A})}_{3i+2},
\\
x &\otw.
}
\]
Equivalently, by setting $p_i = p(\frac{1}{2}+i)$, we have $\p{g} = \prod_{i=0}^r \p{g}^{(i)}$ where
\[
\p{g}^{(i)}
=
(s_{p_i + m_i+\beta_i}\cdots s_{p_i +m_i+2}s_{p_i +m_i+1})
\cdots
(s_{p_i +m_i+\beta_i+\alpha_i-1}\cdots s_{p_i +m_i+\alpha_i}).
\]
Now we define $\p{x} \in W(d + p(r+1))$ for any  $x \in W(d)$ with $x \leq g$ as follows.
Let $x = \prod_{i=0}^r x^{(i)}$, and $x^{(i)}= s_{j_{i,1}} \ldots s_{j_{i,k_i}}$ be a reduced expression, which is a subexpression of $g^{(i)}$. Define
\[
\p{x} = \prod_{i=0}^r \p{x}^{(i)}
\quad
\textup{where}
\quad
\p{x}^{(i)} = s_{p_i +j_{i,1}} \ldots s_{p_i +j_{i,k_i}}.
\]
In other words, $\p{x}$ is obtained from $x$ by enlarging the domain
(i.e., $R^{\delta(A)}_{3i}$ is replaced by $R^{\delta(\p{A})}_{3i}$) on which it acts as an identity map.
The non-trivial action of $\p{x}$ is the same as the non-trivial action of $x$ (up to a shift by $p_i$).
%
%We begin with some technical lemmas that are used in the proof.
%-----------------------------------------------------------------------------------------------------------
%\lem\label{lem:p<}
%Let $A = \kappa(\ld,g,\mu)$ be tridiagonal.  Then
We observe that
\eq \label{D=D}
\{ w \in \D_{_p\ld, _p\mu} ~|~ w \leq \p{g}\} = \{\p{x} \in \D_{\ld\mu} ~|~ x \leq g\}.
\endeq
%\endlem
%-----------------------------------------------------------------------------------------------------------
%\proof
Indeed, for $w \leq \p{g}$, it follows from construction that $x$ is a subexpression of $\p{g}$ and hence there is a unique $x \leq g$ such that $\p{x}=w$ obtained by subtracting $p_i$ from the indices of simple reflections.
It follows by Lemma \ref{lem:Dld} that $x \in \D_{\ld\mu}$.
The other inclusion in \eqref{D=D} is similar.
%\endproof

%\lem\label{lem:pKL}
%Let $A = \kappa(\ld,g,\mu)$ be tridiagonal.
As the non-trivial portions of $(_p{x},_p{g})$ are those of $(x,g)$,
the corresponding Kazhdan-Lusztig polynomials coincide, i.e.,
\eq  \label{P=P}
P_{x,g} =P_{_p{x},_p{g}}.
\endeq
%\endlem
%-----------------------------------------------------------------------------------------------------------
%\proof
%Both of them coincide with the Kazhdan-Lusztig polynomial $P_{x,g}$ for the subgroup of $W$ generated by
%$s_j$ with $j \in \bigcup_{i=0}^{r} (m_i .. m_i+\af_i+\beta_i)$.
%
%{\color{red}
%Alternatively, one can argue that they are constructed from the identical $R$-polynomials,
%which arise from the same recurrence relations.
%}
%\endproof
%-----------------------------------------------------------------------------------------------------------

Denote by $\p{g}^+_{\ld\mu}$ the longest element in $(W_{_p\ld}) \p{g} (W_{_p\mu})$.
By Lemma \ref{lem:Cur}(c), we have
\eq\label{eq:pbar1}
T_{(W_{_p\ld})  \p{g} ( W_{_p\mu})}
= v^{\ell(\p{g}^+_{\ld\mu})} C'_{_p{g}^+_{\ld\mu}} + \sum_{\substack{w\in \D_{_p\ld,_p\mu}\\ w < _pg}} c^{(_p\ld,_p\mu)}_{w,_pg} C'_{w^+_{_p\ld,_p\mu}}.
\endeq
By \eqref{D=D}, Equation~ \eqref{eq:pbar1} can be written as
\eq\label{eq:pbar2}
T_{(W_{_p\ld})\p{g}( W_{_p\mu})}
= v^{\ell(\p{g}^+_{\ld\mu})} C'_{_p{g}^+_{\ld\mu}} + \sum_{\substack{x\in \D_{\ld\mu}\\ x < g}} c^{(_p\ld,_p\mu)}_{_px,_pg} C'_{_px^+_{\ld\mu}}.
\endeq
In particular,
$
T_{(W_{_p\mu}) \id  (W_{_p\mu})}=x_{_p \mu} = v^{\ell(\wo^{_p\mu})} C'_{\wo^{_p\mu}},
$
where
\begin{align}  \label{eq:pbar3}
\ell(\wo^{_p\mu})
&= \ell(\wo^{\mu}) + \Bp{\mu_0+\frac{p}{2} }^2 - \mu_0^2 + \sum_{i=1}^r\Bp{ {\mu_i+p \choose 2} -{\mu_i \choose 2}}+ \Bp{\mu_{r+1}+\frac{p}{2}}^2 - \mu_{r+1}^2
\notag
\\
&=\ell(\wo^{\mu}) + p\mu_0 + \frac{p^2}{4} +\sum_{i=1}^r p\mu_i+ r{p \choose 2}+ p\mu_{r+1} + \frac{p^2}{4}
\\
&= \ell(\wo^{\mu}) %+ \frac{p}{2}\Bp{ 2d - r + p(r+1) }.
+p(\textstyle d-\frac{r}{2}) + p^2(\frac{r+1}{2}).
\notag
\end{align}
By Lemma \ref{lem:Cur} again, for any $x \in \D_{\ld\mu}$ such that $x \leq g$ with $A_x = \kappa(\ld,x,\mu)$, we have
\begin{align}  \label{eq:pbar4}
&\ell(\p{x}_{\ld\mu}^+) - \ell(x_{\ld\mu}^+)
\notag
\\
&= \Bp{\ell(\wo^{_p\ld}) - \ell(\wo^\ld)} + \ell(\p{x}) - \ell(g) - \Bp{\ell(\wo^{\delta(\p{A_x})}) - \ell(\wo^{\delta(A)})} + \Bp{\ell(\wo^{_p\mu}) - \ell(\wo^\mu)}
\notag  \\
&= %\frac{p}{2}\Bp{ 2(d + |A_x^\offd|) - r + p(r+1)}
p(d + |A_x^\offd| - \textstyle\frac{r}{2}) + p^2(\frac{r+1}{2}).
\end{align}
Here $|A_x^\offd|$ is the sum of off-diagonal entries of $A_x$ over $\Ia$.
Hence $v^{\ell(\p{x}^+_{\ld\mu})}$ is a specialization at $(v, v^{-p},v^{p^2})$ of some element in $\mathcal{R}_3$,
for $x \leq g, x\in \D_{\ld\mu}$.
By \cite[Theorem ~1.10, Corollary ~1.12]{Cur85},  the matrix
$(c^{(\p{\ld},\p{\mu})}_{{}_p{x},{}_p{g}})_{x,g \in \D_{\ld\mu}}$ is the inverse of the matrix
$(v^{-\ell(\p{g}_{\ld\mu}^+)}P_{_p{x}^{+}_{\ld\mu},_p{g}^{+}_{\ld\mu}})_{x,g}$. Hence
it follows by \eqref{P=P} and \eqref{eq:pbar4} that $c_{_px,_pg}^{(_p\ld,_p\mu)}$
is a specialization at $(v, v^{-p},v^{p^2})$ of some element in $\mathcal{R}_3$.

Therefore we deduce from \eqref{eq:pbar2} that
\eq
\Bp{\={[\p{A}]} - [\p{A}]} (C'_{\wo^{_p\mu}})
=
\sum_{\substack{x\in \D_{\ld\mu} \\ x<g}}
\Bp{ v^{\ell(\p{g^+_{\ld\mu}})} \={c}_{_px,_pg}^{(_p\ld,_p\mu)}
 - v^{-\ell(\p{g^+_{\ld\mu}})} c_{_px,_pg}^{(_p\ld,_p\mu)} } C'_{\p{x}^+_{\ld\mu}},
\endeq
where the coefficients on the RHS are specializations at $(v, v^{-p},v^{p^2})$ of some elements in $\mathcal R_3$.
By an induction on the Bruhat order,
there exist $H_i \in \mathcal{R}_3$ and $w^{(i)} \in \D_{\ld\mu}$ satisfying
\eq
\={[\p{A}]} = [\p{A}] + \sum_i H_i(v,v^{-p},v^{p^2}) [\p{T}_i],
\endeq
where $T_i = \kappa(\ld, w^{(i)}, \mu)$.
%For arbitrary tridiagonal $A \in \Xit_n$,
%let $q \in\ZZ_{\geq 0}$ be the smallest integer such that $\leftidx{_q}{A} \in \Xi_n$,
%and we write $A' = \leftidx{_q}{A} = \kappa(\ld,g,\mu)$. We obtain, similarly,
%\eq\label{eq:pAlower}
%\={[\p{A'}]} \in [\p{A'}] + \sum_i H_i(v,v^{-p},v^{p^2}) [\p{T}_i],
%\endeq
%with $T_i = \kappa(\ld,x,\mu)-qI$ for some $x\in\D_{\ld\mu}$ such that $x< g$.
This finishes the proof of the identity \eqref{eq:stab2} when $A$ is tridiagonal.

The proposition is proved.
\endproof

Let $\dKjj$ be the free $\mathbb Q(v)$-module with basis given by the symbols $[A]$ for $A\in \Xit_n$ (which will be
called a standard basis of $\dKjj$).
By Propositions~ \ref{prop:stab1}--\ref{prop:stab2} and applying a specialization at $v'=1$, we have the following corollary.
\begin{cor}
There is a unique associative $\mathbb Q(v)$-algebra structure on $\dKjj$ with multiplication given by
\[
[A_1] \cdot [A_2] \cdot \ldots \cdot [A_f]
= \bc{
\sum_{i=1}^m G_i(v, 1) [Z_i]
&\tif \coC(A_i) = \roC(A_{i+1})\textup{ for all }i,
\\
0
&\textup{otherwise}.
}
\]
Moreover, the map $\bar{\ }: \dKjj \rw \dKjj$ given by
$
\={v^k [A]} = v^{-k} \sum_{i=1}^s H_i(v, 1, 1) [T_i]
$
is a $\QQ$-linear involution.
\end{cor}

%=====================
\section{Multiplication formula for $\dKjj$}

The following multiplication formula in $\dKjj$ follows directly from
Theorem~\ref{thm:multformula2} by the stabilization construction (see Proposition~\ref{prop:stab1}).
\begin{thm}
 \label{thm:multK}
Let $A, B \in \Xit_{n}$ with $B$  tridiagonal and $\roC(A)=\coC(B)$.
Then we have
\eq\label{eq:multdotK}
[B]\, [A] = \sum_{\substack{T \in \widetilde{\Tt}_{B,A}\\ S \in \Gamma_T}} v^{\beta(A,S,T)} (v^2-1)^{n(S)} \={\LR{A;S;T}}~ [A^{(T - S)}],
\endeq
where
\eq
\widetilde{\Tt}_{B, A} = \big\{T\in\Tt_n ~|t_{ij}+t_{-i,-j}\leq a_{ij} \mbox{ unless $i=j$}, \;  \roA(T)_i = b_{i-1,i} \textup{ for all } i \big\}.
\endeq
\end{thm}

%=====================
\section{Monomial and stably canonical bases for $\dKjj$}

Recall an admissible pair of matrices  $(A, B) \in \Xi_{n}\times \Xi_{n}$ is defined by Conditions (a)--(b) in \S\ref{sec:leading}.
We extend this to a definition of admissible pair $(A, B) \in \Xit_{n}\times \Xit_{n}$ by imposing the same conditions.
%if the following conditions (a)--(b) hold:
%\enu
%\item[(a)] $B^\offd = \sum\limits_{i=1}^n b_{i,i+1} \Ett^{i,i+1}$ (see \eqref{def:offd} for notation $B^\offd$);
%\item[(b)]  $A^\offd = \sum\limits_{j=1}^x \sum\limits_{i = 1}^n a_{i,i+j} \Ett^{i,i+j}$
%for some $x\in \NN$,
%where $a_{i,i+x} \geq b_{i,i+1}$ for all $i$.  (In this case, we say that \emph{$A$ is of depth $x$}).
%\endenu

The following lemma is similar to Lemma \ref{lem:ss} and can be proved by the same arguments.

\begin{lemma}  \label{lem:ssK}
If $(B,A)$ is admissible and $A$ is of depth $x$ (see (a)--(b) in \S\ref{sec:leading}),
then
\[
[B] [A] = [M(B,A)] + \textup{ lower terms},
\]
where $M(B,A)=A-\sum_{i=1}^nb_{i,i+1}(E_\theta^{i,i+x}-E_\theta^{i-1,i+x})$.
\end{lemma}

The following is an analogue of Theorem~\ref{thm:min gen}.
\begin{prop}  \label{prop:mAK'}
For any $A\in\Xit_{n}$ of depth $x$, there exist tridiagonal matrices $A^{(1)}, \ldots, A^{(x)}$ in $\Xit_{n}$
satisfying $\roC(A^{(1)})=\roC(A)$, $\coC(A^{(x)})=\coC(A)$, $\coC(A^{(i)})=\roC(A^{(i+1)})$ for $1\leq i\leq x-1$ and
$A^{(i)}- (\sum_{j=1}^n   \sum_{m\le j-x+i}a_{m,j+1})E_\theta^{j,j+1}$ is diagonal for all $1\leq i\leq x$ such that
\eq \label{eq:mAK}
m'_A:= [A^{(1)}] [A^{(2)}] \cdots  [A^{(x)}]=[A]+\mbox{lower terms} \; \in \dKjj.
\endeq
\end{prop}

\proof
Lemma~\ref{lem:ssK} leads to an algorithm (almost) identical to Algorithm~\ref{alg:mono},
which produces the tridiagonal matrices $A^{(i)}$ (whose diagonal entries might be negative) as needed.
The rest of the proof is the same as for Theorem~\ref{thm:min gen}. % (where $A_i$ is denoted by $A^{(i)}$).
\endproof

Hence $\{m'_A~|~A\in\Xit_{n}\}$ forms a basis for $\dKjj$ (called a \emph{semi-monomial basis}).

Let $B \in \Xit_n$ be tridiagonal. The analogue of Lemma \ref{lem:B'} holds here, and so we have
\[
\overline{ [B] } \in [B] + \sum_{B'  \text{ tridiagonal}, B' <_\alg B} \mA [B'].
\]
If $B$ is diagonal, set $\{B\} =[B]$.
Arguing inductively on the partial order $\le_\alg$ and using a standard argument (cf. \cite[24.2.1]{Lu93})
there exists a unique element $\{B\} \in \dKjj$ such that
\[
\overline{ \{B\} } =\{B\}, \qquad
\{B\} \in [B] + \sum_{B' \text{ tridiagonal}, B' <_\alg B} v^{-1}\ZZ[v^{-1}] [B'].
\]
Modifying \eqref{eq:mAK}, for $A\in \Xit_n$, we now define
\eq \label{eq:mAK2}
m_A =  \{A^{(1)}\}  \{A^{(2)}\} \cdots  \{A^{(x)}\}.
\endeq

%\lem  \label{lem:pAbar}
%For any $A \in \Xit_{n}$, we have
%$\={[A]} = [A] + \textup{lower terms} \in \dKjj.$
%\endlem
%
%\proof
%The assertion for $A$ tridiagonal follows from \eqref{eq:pAlower}.
%
%For arbitrary $A$, by Proposition~\ref{prop:mAK'}
%there are tridiagonal matrices $A^{(1)}, \ldots, A^{(x)} \in \Xit_{n}$ such that
%\eq
%[A] = [ A^{(1)}] [ A^{(2)}] \cdots [ A^{(x)}] +\textup{lower terms} \in \dKjj.
%\endeq
%Apply the bar map to the above equation. The lemma follows by induction on $A$ (with respect to $\leq_\alg$)
%and an analogue of Lemma \ref{lem:lower2} on $\dKjj$.
%\endproof

\begin{thm}  \label{thm:dKjj CB}
{\quad}
\enu
\item[(a)]
%The element $m_A$, for $A \in \Xit_{n} $,   is bar-invariant, and moreover, $m_A = [A] +$ lower terms.
We have $m_A \in [A] + \sum_{B\in \Xit_n, B <_\alg A} \mA [B]$, for any $A \in \Xit_{n} $.
Hence the set $\{m_A~|~A \in \Xit_{n} \}$ forms an $\mA$-basis of $\dKjj$ (called a monomial basis of $\dKjj$).

\item[(b)]
There exists a unique basis $\Bjj =\{ \{A\}~|~A \in \Xit_{n} \}$ of $\dKjj$
such that  $\overline{ \{A\} } =\{A\}$ and
$\{A\} \in [A] + \sum_{B\in \Xit_n, B <_\alg A} v^{-1} \ZZ[v^{-1}]\, [B]$ ($\Bjj$ is called  stably canonical basis of $\dKjj$).
\endenu
\end{thm}

\proof
With the help of Proposition~\ref{prop:mAK'},
the assertion for monomial basis is proved in the same way as for Theorem~\ref{thm:mono}, and hence skipped.

It follows by (a) that $\={[A]} = [A] + \textup{lower terms} \in \dKjj.$
The canonical basis follows by this and a standard argument (cf. \cite[24.2.1]{Lu93}).
\endproof

%=====================
\section{Isomorphism $\dot{\bK}^{\fc, \textup{geo}}_{n} \cong \dKjj$}

Recall from ~\cite[Chapter 9.4]{FLLLW}, there is an associative $\mathbb Q(v)$-algebra $\dot{\bK}^{\fc, \textup{geo}}_{n}$,
with a standard basis given by $\{ [A]^{\geo} ~|~A\in \Xit_n \}$. and a bar map defined by a stabilization procedure.
(Note $\dot{\bK}^{\fc, \textup{geo}}_{n}$ and $[A]^{\geo}$ was denoted by $\dKjj$ and $[A]$ therein.)

\begin{prop}\label{Knc}
  There is an algebra isomorphism $\dot{\bK}^{\fc, \textup{geo}}_{n} \stackrel{\simeq}{\longrightarrow}  \dKjj$,
  $[A]^{\geo} \mapsto [A]$. The isomorphism commutes with the bar maps, and it preserves the canonical bases,
  i.e., $\{A\}^\geo \mapsto \{A\}$ for all $A$.
\end{prop}

\begin{proof}
Recall from Proposition~\ref{Schur-iso} the algebra isomorphisms $\psi: \Sjjg \stackrel{\simeq}{\rightarrow} \Sjj$, for all $d$.
As the algebra structures
$\dot{\bK}^{\fc, \textup{geo}}_{n}$ and $\dKjj$ arise from the same stabilization procedure from the family of algebras
$\{\Sjjg \}_{d\ge 0}$ and $\{\Sjj \}_{d\ge 0}$, respectively
(by comparing Proposition~ \ref{prop:stab1} with ~\cite[Proposition 9.2.5]{FLLLW} and noting
 $\mathcal{R}_1 \subset \mathcal{R}_2$), we obtain the algebra isomorphism
$\dot{\bK}^{\fc, \textup{geo}}_{n} \simeq  \dKjj$.

By Proposition~\ref{prop:psiCB}, $\psi: \Sjjg \rw \Sjj$ commutes with the bar maps.
As the bar maps on $\dot{\bK}^{\fc, \textup{geo}}_{n}$ and $\dKjj$ are defined by the same stabilization procedure
 from the bar maps on the family of algebras $\{\Sjjg \}_{d\ge 0}$ and $\{\Sjj \}_{d\ge 0}$, respectively
 (by comparing Proposition~ \ref{prop:stab2} with ~\cite[Proposition 9.2.7]{FLLLW} and noting
 $\mathcal{R}_2 \subset \mathcal{R}_3$), the compatibility of the 2 bar maps under the isomorphism follows.
 The last claim is proved by the same argument as for Schur algebras in Proposition~\ref{prop:psiCB}.
\end{proof}

\rmk  \label{K=Sjj}
The algebra isomorphism $\dot{\bK}^{\fc, \textup{geo}}_{n} \cong  \dKjj$ in Proposition~\ref{Knc}
allows us to transport further results in \cite[\S9.7]{FLLLW} over here, and so we do not give new proofs here.
Set ${}_{\QQ} \Sjj =\QQ(v)\otimes_{\ZZ[v,v^{-1}]} \Sjj$.
In particular, the map
\begin{align*}
\Psi_d: \dKjj \longrightarrow {}_{\QQ}\Sjj,
\qquad
[A] \mapsto
\begin{cases}
[A], & \text{ for }A \in \Xi_{n,d},
\\
0, & \text{ for }A \in \Xit_n \big \backslash \Xi_{n,d}   %\text{ otherwise}
\end{cases}
\end{align*}
is a homomorphism which preserves the canonical bases;
that is, $\Psi_d$ sends $\{A\} \mapsto \{A\}$ for $A \in \Xi_{n,d}$ and $\{A\} \mapsto 0$ otherwise.
\endrmk

%=========================================================
\chapter{The  quantum symmetric pair $(\Kn, \Kjj)$ }
  \label{sec:coidealK}

In this chapter, we improve the results in \cite{FLLLW} on the idempotented form of quantum symmetric pairs
to genuine quantum symmetric pairs.
We construct an algebra $\Kjj$ as a subalgebra of a completion of the algebra $\dKjj$, after reviewing
a similar type A construction. We study a comultiplication on $\Kjj$,
show that $\Kjj$ is a coideal subalgebra of $\Kn$ (a stabilization algebra of affine type A),
and that $(\Kn, \Kjj)$ forms a quantum symmetric pair.

%=====================
\section{The  algebra $\Kn$ of type A}
In this section we review briefly the affine type A construction (which goes back in finite type A to \cite{BLM90}).

Recall that
\eq
\ZZ_n = \big \{ \ld = (\ld_i)_{i \in \ZZ} ~|~ \ld_i \in \ZZ, \ld_i = \ld_{i+n}, \forall i\in \ZZ \big \}.
\endeq
Let $\hKn$ be the $\QQ(v)$-vector space of all (possibly infinite) linear combinations
$\sum_{A\in\~{\Tt}_n} \xi_A \ ^\fa[A]$, for $\xi_A \in \QQ(v)$ and ${}^\fa[A] \in \dKn$,
such that   the sets $\{A\in \~{\Tt}_{n} ~|~ \xi_A \neq 0, \roA(A) = \ld\}$ and
$\{A\in \~{\Tt}_{n} ~|~ \xi_A \neq 0, \coA(A) = \ld\}$ are finite, for any $\ld \in \ZZ_n$.

The following multiplication defines an algebra structure for $\hKn$:
\[
\Bp{\sum_A \xi_A \ ^\fa[A]}  \Bp{\sum_B \eta_B \ ^\fa[B]}
= \sum_{A,B} \xi_A \eta_B (\ ^\fa[A] \, \ ^\fa[B]).
\]

Let
\[
\Tt^0_n =\{ A=(a_{ij}) \in \Tt_n~|~a_{ii}=0, \forall i \in \ZZ\}.
\]
For $\af \in \ZZ_n$, we set
\[
D_{\af} = \textup{diag}(\af) \in \Tt_n.
\]
For each $\bj \in \ZZ_n$ and $A\in \Tt^0_n$, we set
\eq
\ ^\fa A(\bj) = \sum_{\af \in \ZZ_n} v^{\bj\cdot \af} \ ^\fa[A +D_\af] \in \hKn,
\endeq
where $\bj\cdot\af = \sum_{i=1}^n j_i \af_i$.
For $t\in \ZZ$ we define
\eq
\vec{t} = (\ldots,t,t,t,\ldots) \in \ZZ^\ZZ.
\endeq
In particular, we have
\eq
\ ^\fa A(\vec{0}) = \sum_{\af \in \ZZ_n} \ ^\fa[A +D_\af],
\quad
\ ^\fa 0(\bj) = \sum_{\af \in \ZZ_n} v^{\bj\cdot \af} \ ^\fa[D_\af].
\endeq
Let $\Kn$ be the $\QQ(v)$-subspace of $\hKn$ spanned by
$
\{{}^\fa A(\bj) ~|~ A\in \Tt^0_n, \bj \in \ZZ_n\}.
$
%A matrix $A\in \Tt^0_n$ is {\em bidiagonal} if $a_{ij}=0$ unless $i-j =0,1$ (or $i-j =0, -1$).
Let
\begin{align*}
\Tt^{1}_n &=\{ A=(a_{ij}) \in \Tt^0_n~|~a_{ij}=0 \text{ unless } j-i=1, \forall i,j \in \ZZ\},
\\
%\qquad
\Tt^{-1}_n &=\{ A=(a_{ij}) \in \Tt^0_n~|~a_{ij}=0 \text{ unless } j-i=-1, \forall i,j \in \ZZ\},
\quad \Tt^{\pm 1} = \Tt^{1}_n \cup \Tt^{-1}_n.
\end{align*}
The proposition below is an affine type A analogue of a result in \cite{BLM90}.
%-----------------------------------------------------------------------------------------------------------
\begin{prop}  [{\cite[Theorem~1.1]{DF15}}]
\label{Prop:DF}
The space $\Kn$ is a subalgebra of $\hKn$ generated by
${}^\fa A(\vec{0}), \ ^\fa 0(\bj)$,  for $A\in \Tt_n^{\pm 1}$ and $\bj \in \ZZ_n.$
Moreover, the algebra $\Kn$ is isomorphic to the quantum affine $\gl_n$.
\end{prop}

%%%%%%%%%%%
\section{The algebra $\Kjj$}  \label{sec:Kjj}
   \label{sec:Kjj-def}

Now we are back to the construction of an algebra $\Kjj$ out of $\dKjj$ so that $\dKjj$ is a modified (or idempotented) version of $\Kjj$.
Recall that
\eq\label{eq:Znc}
\ZZ_n^\fc = \{ \ld = (\ld_i)_i \in \ZZ_n ~|~ \ld_i = \ld_{n+i} = \ld_{-i}; \ld_0,\ld_{r+1} \in 2\ZZ+1\}.
\endeq
Let $\hKcn$ be the $\QQ(v)$-vector space of all (possibly infinite)  linear combinations
$\sum_{A\in\~{\Xi}_n} \xi_A [A]$ for
$\xi_A \in \QQ(v), \;  [A] \in \dKcn, $
such that, for any $\ld \in \ZZ^\fc_n$, the sets
\begin{equation}\label{fincondition}
\mbox{$\big\{A\in \~{\Xi}_{n} ~|~ \xi_A \neq 0, \roA(A) = \ld \big\}$ and
$\big\{A\in \~{\Xi}_{n} ~|~ \xi_A \neq 0, \coA(A) = \ld \big\}$ are finite.}
\end{equation}

We have an algebra structure on  $\hKcn$ which extends the one on $\Kjj$:
\[
\Bp{\sum_A \xi_A [A]}\cdot \Bp{\sum_B \eta_B [B]}
= \sum_{A,B} \xi_A \eta_B [A] [B].
\]
Let $\Xi^0_n$ be the subset of $\Tt^0_n$ consisting of centro-symmetric matrices in $\Tt^0_n$.
For any $\nu \in \ZZ_n$, we define $\nu_\tt \in \ZZ^\fc_n$ by
\eq
(\nu_\tt)_i = \nu_i + \nu_{-i} + \sum_{k\in\ZZ} \delta_{i,k(r+1)}.
\endeq
For each $\bj \in \ZZ_{r+2}, A\in \Xi_n^0$, we set
\[
A(\bj) = \sum_{\af \in \ZZ^\fc_n} v^{\bj\bullet \af} [A +D_\af] \in \hKcn,
\]
where
\[
\bj\bullet \af = j_0\frac{\af_0-1}{2} + \sum_{i=1}^{r} j_i \af_i + j_{r+1}\frac{\af_{r+1}-1}{2} \in \ZZ.
\]
In particular, we have
\eq
A(\vec{0}) = \sum_{\af \in \ZZ^\fc_n} [A +D_\af],
\quad
0(\bj) = \sum_{\af \in \ZZ^\fc_n} v^{\bj\cdot \af} [D_\af].
\endeq
Let $\Kcn$ be the $\QQ(v)$-subspace of $\hKcn$ spanned by
\[
\mathfrak{B} = \{A(\bj) ~|~ A\in \Xi^0_n, \bj \in \ZZ_{r+2}\}.
\]

\prop\label{prop:Kjjgen}
The subspace $\Kcn$ is a subalgebra of $\hKcn$ generated by
$ A(\vec{0}), 0(\bj)$, where $A\in \Xiz$ is tridiagonal and $\bj \in \ZZ_{r+2}$.
\endprop

\proof
Denote by ${\bf K}_n'$ the subalgebra of $\hKcn$ generated by the elements specified in the proposition.
We shall show that the space $\Kcn$ is closed under the left multiplication by the elements in the proposition.
We give a detailed proof that $B(\vec{0}) \cdot A(\bj) \in \Kjj$, for tridiagonal $B$ and arbitrary $A$ in $\Xi_n^0$.
(We shall skip a similar proof that  %$A(\bj)  B(\vec{0})$,
$0(\bj_1) A(\bj)$, % $A(\bj) 0(\bj_1)\in\Kjj$
for any $A\in \Xi^0_n$ and $\bj_1,\bj \in \ZZ_{r+2}$.) It follows that ${\bf K}_n'\subseteq \Kcn$ by noting $0(\bj) \in \hKcn$.

We have
\eq
\bA{
B(\vec{0}) \cdot A(\bj)
&= \sum_{\gamma \in \ZZ_n^\fc} v^{\bj\bullet\gamma} [B+D_\beta]  [A+D_\gamma],
}
\endeq
for a unique $\beta$ satisfying $\coC(B+D_\beta) = \roC(A+D_\gamma)$.
Let $m_{B,A}^C \in \mA$ be the structure constants in $\dKjj$ such that
$
[B]\cdot[A] = \sum_{C} m_{B,A}^C [C] \in \dKjj.$
By the multiplication formula in \eqref{eq:mult3}, we have
\[
B(\vec{0}) \cdot A(\bj)
=
\sum_{\gamma \in \ZZ_n^\fc}  v^{\bj\bullet\gamma}
\sum_{T\in \Tt_{B+D_\beta, A+D_\gamma}}
\sum_{S\in\Gamma_T}
m_{B+D_\beta, A+D_\gamma}^{A^{(T-S)}+D_\gamma} [A^{(T-S)}+D_\gamma].
\]
Note that $ \Tt_{B+D_\beta, A+D_\gamma} =  \Tt_{B, A+D_\gamma}$ by construction. For any $T \in \Tt_n$ such that $T_\tt \leq_e A$, we set $\tau = \tau(B,A,T) \in \ZZ_n$ be such that
\[
\tau_i = b_{i-1,i} - \roA(T)_i.
\]
Therefore, $T+D_\tau \in \Tt_{B, A+D_\gamma}$ for all $\gamma \in \ZZ_n^\fc$ such that $\tau_\tt \leq \gamma$, and hence
\[
B(\vec{0}) \cdot A(\bj)
=
\sum_{\substack{T\in\Tt_n, T_\tt \leq_e A\\ \roA(T) \leq (b_{i-1,i})_i}}
\sum_{S\in\Gamma_{T+D_\tau}}
\sum_{\gamma \in \ZZ_n^\fc}  v^{\bj\bullet\gamma}
m_{B+D_\beta, A+D_\gamma}^{A^{(T+D_\tau-S)}+D_\gamma} [A^{(T+D_\tau-S)}+D_\gamma].
\]
We claim that
\eq\label{eq:52claim}
m_{B+D_\beta, A+D_\gamma}^{A^{(T+D_\tau-S)}+D_\gamma}
= \sum_i v^{\bk^{(i)} \bullet\gamma} m^{(i)}_{T,S}
\endeq
for some
$
\bk^{(i)}\in\ZZ_{r+2},
m^{(i)}_{T,S} \in \QQ(v).
$
As a consequence, we obtain
\[
\bA{
B(\vec{0}) \cdot A(\bj)
&=
\sum_{T,S,i}
m^{(i)}_{T,S}
\sum_{\gamma \in \ZZ_n^\fc}  v^{(\bj+\bk^{(i)})\bullet\gamma}
[A^{(T+D_\tau-S)}+D_\gamma]
\\
&=
\sum_{T,S,i}
m^{(i)}_{T,S}
A^{(T+D_\tau,S)}(\bj+\bk^{(i)}) \in \Kjj.
}
\]

Let us prove the claim \eqref{eq:52claim}.
Note that if $[A+D_\gamma] \in \Sjj$, the structure constants are given by
\[
m_{B+D_\beta, A+D_\gamma}^{A^{(T+D_\tau-S)}+D_\gamma}
= (v^2-1)^{n(S)}
v^{\beta(A+D_\gamma, S, T+D_\tau)+\gamma(A+D_\gamma, S, T+D_\tau)}
\LR{A+D_\gamma; S; T+D_\tau}.
\]
A detailed calculation shows that the $v$-exponent $\beta(A+D_\gamma, S, T+D_\tau)+\gamma(A+D_\gamma, S, T+D_\tau)$
 is a polynomial in variables $\gamma_{k}$  $(0\leq k \leq r+1)$ with total degree one.
Also, by \eqref{eq:[AST]}
we have
\[ %\label{eq:[ADSTD]}
\bA{
\LR{A+D_\gamma; S; T+D_\tau}
&=
M\cdot\prod_{1 \leq k \leq r} \prod_{i=1}^{(S + \^{T-S})_{\tt,kk}} (v^{2(\gamma_k-\tau_k-\tau_{-k}+i)} -1)
\\
&\qquad\qquad \cdot
\prod_{k \in\{0, r+1\}} \prod_{i=1}^{(S + \^{T-S})_{kk}} (v^{2(\gamma_k-2\tau_k-1+2i)} -1),
}
\]
for some $M\in\QQ(v)$ that is independent of $\gamma_k$'s.
Thus, \eqref{eq:52claim} follows on the Schur algebra level, and hence it follows on the stabilization algebra level.

It remains to show that $A(\bj)$, for arbitrary $A\in \Xi^0_n, \bj \in \ZZ_{r+2}$,
can be generated by these (expected) generators in the proposition.
This follows from Theorem \ref{thm:min gen} and an induction on $\Psi(A)$ \eqref{eq:PsiA}.
This then implies that ${\bf K}_n'\supseteq \Kcn$. The proposition is proved.
\endproof

\section{The algebra $\Kjj$ as a subquotient}
   \label{sec:Kjj-quot}

In this section and the subsequent \S\ref{sec:coKjj}, we shall use \cite[Section~9]{FLLLW} substantially.
In this preliminary section, we identify the algebra $\Kjj$ as a subquotient of a higher rank algebra generated by Chevalley generators.

Let
\[
\breve n = n+2,
\]
and
\eq
\XitL = \big\{A \in \Xit_{\breve n} ~|~\roA(A)_1 =\coA(A)_1 =0 \big\}.
\endeq
Recall from \cite[\S 9.8]{FLLLW} that there is a subalgebra $\dKcL$ of $\dKcb$:
\eq\label{eq:dKcL}
\dKcL = \Span_{\QQ(v)} \big\{  [A] \in \dKcb ~|~ A \in \XitL \big\}.
\endeq
The algebra $\dKcL$ contains an ideal
\[
\cI = \Span_{\QQ(v)}\{[A] \in \dKcL ~|~ a_{11} < 0\}.
\]
For $A \in \Xit_n$, let $\ddot A$ be the matrix in $\XitL$ obtained from $A$ by inserting
rows/columns of zeros between the 0th and $\pm1$st rows/columns (mod $n$).
The assignment $\Xit_n \rw \XitL$,
$A \mapsto \ddot A$, induces  an isomorphism
\eq
\trho: \dKjj \longrightarrow \dKcL/\cI,
\quad
[A] \mapsto [\ddot A] + \cI.
\endeq

Let $\hKcL$ be the subalgebra of $\hKjj$ consisting of (possibly infinite) formal sums of the form
$
\sum_{A \in \XitL} \xi_A [A],
\text{ for } \xi_A \in \QQ(v).
$
In particular, $\hKcL$ contains the
{\em restricted sums}
\eq
M[\bj]
= \sum_{\af \in \ZZ^{\fc r}_{\breve n}} v^{\~{\bj}\cdot \af}[M + D_\af],
%= \sum_{\bz\in\NN^{r+2}} v^{\bj\cdot\bz} [M + \ddot{D_\af}]
\quad \text{ for } M \in  \Xit_{\breve n, 1, 0},
\endeq
where we denote, for $\bj \in \ZZ_{r+2}$,
\[
\~{\bj} = (j_0, 0, j_1, j_2, \ldots, j_{r+1}) \in \ZZ_{r+3},
\qquad
\mbb Z_{\breve n}^{\fc r} = \{ \alpha =(\alpha_i) \in \mbb Z_{\breve n}^{\fc}\ | \ \alpha_1=0\}.
\]
(We note that all infinite sums in this section are restricted.)
The isomorphism $\trho$ induces an isomorphism
\eq
\label{hrho}
\hrho: \hKjj \longrightarrow \hKcL \big/ \widehat \cI,
\quad
[A] \mapsto [\ddot A] + \widehat \cI
\endeq
where $\widehat \cI$ is the completion of $\cI$ in $\hKjj$.
In particular, $\hrho$ sends $A(\bj)$ to $\ddot{A}[\bj]+\widehat \cI$, for $A \in \Xiz$ and $\bj \in \ZZ_{r+2}$.

Denote
\[
\tilde \Xi_{\breve n}' = \{A = (a_{ij}) \in \tilde \Xi_{\breve n}\ | \ a_{ii}=0, \ {\rm if}\ i \neq \pm 1 \mod \breve n\}.
\]
By abuse of notation, we denote by $\mbf Y_{\breve n}^{\fc}$
the subspace of $\widehat{\mbf K}_{\breve n}^{\fc}$ spanned by all restricted sums $A[\mbf j]$, for $A \in \tilde \Xi_{\breve n}' $
and $\mbf j\in \ZZ_{r+2}$.

\begin{lem}
  The subspace $\mbf Y_{\breve n}^{\fc}$ is a subalgebra of $\widehat{\mbf K}_{\breve n}^{\fc}$ generated by
 $A[\vec{0}]$ and $0[\mbf j]$, for all  tridiagonal matrices $A$ in $\tilde{\Xi}_{\breve n}'$ and $\mbf j\in \ZZ_{r+2}$.
\end{lem}

\begin{proof}
 The proof is essentially the same as that for Proposition~ \ref{prop:Kjjgen} and we shall be brief.
  For the reader's convenience, we show that  $\mbf Y_{\breve n}^{\fc}$ is closed
  under multiplication on the left by $S_{\alpha}[\vec{0}]$ for some tridiagonal matrices $S_{\alpha}$.
  By definition, we have
  \begin{equation*}
    S_{\alpha}[\vec{0}] \cdot A[\mathbf j]
= \sum_{\beta \in \mathbb Z_{\breve n}^{\fc r}} [S_{\alpha} + D_{\beta'}] \cdot v^{\mathbf j \bullet \beta} [A + D_{\beta}],
  \end{equation*}
  where $\beta'$ is uniquely determined by the condition $\coA( S_{\alpha} + D_{\beta'}) = \roA( A + D_{\beta})$.
  Note that if $\coA(S_{\alpha})_1 \neq \roA(A)_1$ then $S_{\alpha}[\vec{0}] \cdot A[\mathbf j] =0$.
  Otherwise, by a similar argument as that for Proposition~ \ref{prop:Kjjgen}, the product
  $[S_{\alpha} + D_{\beta'}] \cdot v^{\mathbf j \bullet \beta} [A + D_{\beta}]$ is a linear combination of
  $[A^{(T-S)} + D_{\beta}]$ whose coefficient is a linear combination of $v^{\beta'' \cdot \mbf k}$ for some $\mbf k$
  and $\beta'' =(\beta_{i}'') \in \mbb Z_{\breve n}^{\fc r}$ such that
  $\beta_i''$ is the $(i,i)$th entry of $A^{(T-S)} + D_{\beta}$ if $ i \neq \pm 1 \mod \breve n$ and $0$ otherwise.
  By changing indices and taking sums over $\beta'' \in \mbb Z_{\breve n}^{\fc r}$,
  the product $S_{\alpha}[\vec{0}] \cdot A[\mathbf j]$ is a linear combination of $A'[\mbf k]$
  for some $A'\in \tilde \Xi_{\breve n}'$ and $\mbf k \in \mbb Z_{r+2}$.
\end{proof}

Let $\KcL $ be the subspace of $\mbf Y_{\breve n}^{\fc}$ spanned by $\ddot A[\mbf j]$ for all $A \in \Xiz, \bj \in \ZZ_{r+2}$.
It follows by \eqref{hrho} that
\eq  \label{eq:KcnL}
\Kcn \cong   \hrho(\Kcn) =  (\KcL + \widehat \cI) \big/\widehat \cI.
\endeq
%Let $\Kcap$ be the subalgebra of $\mbf Y_{\breve n}^{\fc}$ generated by the elements
%$E^{i,i+1}_{\tt, \breve n} [\vec{0}], 0[\bj]$,
%for $1 \leq i \leq \breve n$ and $\bj \in \ZZ_{r+2}.$
We now introduce a subalgebra of $\mbf Y_{\breve n}^{\fc}$
\eq  \label{Kcap}
\Kcap = \text{subalgebra generated by } E^{i,i+1}_{\tt, \breve n} [\vec{0}], 0[\bj],
\forall 1 \leq i \leq \breve n, \bj \in \ZZ_{r+2}.
\endeq
(Here the notation $E^{i,i+1}_{\tt, \breve n}$ is adapted from \eqref{Ekl}--\eqref{Ttt}
with additional subscript $\breve n$ to indicate it lies in $ \Xit_{\breve n}$. Similar self-explanatory notation will be used below.)

\begin{lemma}\label{lem:Kcapgen}
For any tridiagonal $A \in\Xiz$, the element  $\ddot{A}[\vec{0}]$ lies in $\Kcap+\widehat \cI$.
In particular, we have $\Kcn \cong
(\KcL + \widehat \cI) \big/\widehat \cI \subseteq (\Kcap + \widehat \cI) \big/\widehat \cI.$
\end{lemma}

\begin{proof}
The second part follows by definition and \eqref{eq:KcnL}.
So it remains to check the first statement in the lemma.

  For any $A  \in \tilde{\Xi}_n$ such that $A - \sum_{1\leq i \leq n}\beta_i E^{i, i+1}_{\theta, n}$ is diagonal, we set
  \begin{equation*}
    \ddot {\mbf f}_{A}[\vec{0}] = \beta_0E^{01}_{\theta, \breve n} [\vec{0}] \cdot \beta_{n-1}E^{n,n+1}_{\theta, \breve n} [\vec{0}] \cdot \beta_{n-1}E^{n+1,n+2}_{\theta, \breve n} [\vec{0}]
    \cdot (\beta_{n-2}E^{n-1,n}_{\theta, \breve n} [\vec{0}] \cdots \beta_{1}E^{23}_{\theta, \breve n} [\vec{0}]  \cdot  \beta_0E^{12}_{\theta, \breve n} [\vec{0}]).
  \end{equation*}
  By a similar argument as for \cite[Lemma~ 9.1.2]{FLLLW},  we have
  \begin{equation*}
    [\ddot A + D_{\alpha}] = \ddot{\mbf f}_A[\vec{0}] \cdot [D_{\coA(\ddot A + D_{\alpha})}] + {\rm lower\ terms},
 \quad \forall \alpha\in \mbb Z_{\breve n}^{\fc r} .
  \end{equation*}
  Therefore we have
  \begin{equation} \label{A0}
    \begin{split}
      \ddot A[\vec{0}] &= \sum_{\alpha \in \mbb Z_{\breve n}^{\fc r}} [A + D_{\alpha}]
      = \ddot {\mbf f}_A[\vec{0}]\sum_{\alpha \in \mbb Z_{\breve n}^{\fc r}} [D_{\coA(\ddot A + D_{\alpha})}] + {\rm lower\ terms} \\
      & = \ddot {\mbf f}_A[\vec{0}] + {\rm lower\ terms}.
    \end{split}
  \end{equation}
  We now consider ``lower terms" in \eqref{A0}.
  Since both $\ddot A[\vec{0}]$ and $\ddot{\mbf f}_A[\vec{0}]$ are in $\mbf Y^{\fc}_{\breve n}$,
  these ``lower terms" are also in $\mbf Y^{\fc}_{\breve n}$.
  Hence they are linear combinations of $B[\mbf j]$ for some $B <_\alg A$.
  Since $B[\mbf j] = v^{-\tilde{\mbf j} \bullet \coA(B)} B[\vec{0}] \cdot 0[\mbf j]$,
  by induction, it suffices to show any such $B$ appearing in ``lower terms" in \eqref{A0}
is equal to $\ddot C$ up to $\mathcal I$ for some $C \in \Xiz$.
It follows by \eqref{A0} that $B \in \tilde{\Xi}_{\breve n, 1,0}$.
  Therefore we have $B \in \mathcal I$ or $B = \ddot C$ for some $C \in \Xiz$.  The lemma is proved.
\end{proof}

\vspace{.2cm}
Now (in this paragraph) we repeat some of the above constructions in the affine type A setting quickly.
We define
\eq
\TtL = \{A \in \~{\Tt}_{\breve n} ~|~ \roA(A)_i = 0 = \coA(A)_i, i = \pm1\}.
\endeq
The algebra $\dot{\bf K}_{\breve n}$ contains a subalgebra
$\dot{\bf K}_{\breve n,1,0}= \Span_{\QQ(v)} \big\{  [A] \in \dot{\bf K}_{\breve n} ~|~ A \in \TtL \big\}.$
For $\bj \in \ZZ_{n}$, let $\~{\bj}\in \ZZ_{\breve n}$ be uniquely determined by
$
\~{\bj}_1 = 0 = \~{\bj}_{-1},
\~{\bj}_0 = \bj_0,
\~{\bj}_i = \bj_{i-1}
\,
(2\leq i \leq n-1).
$
For $M \in \~{\Tt}_{\breve n}$, we define the restricted sum
\[
M[\bj]^\fa
= \sum_{\substack{\af \in \ZZ_{\breve n}\\ M+D_\af\in \TtL }}
v^{\~{\bj}\cdot \af}[M + D_\af]
%= \sum_{\bz\in\NN^{r+2}} v^{\bj\cdot\bz} [M + \ddot{D_\af}]
\in \hKL.
\]
As counterparts of $\hKcL$ (respectively,  $\Kcap$ and $\KcL$), we have their type A counterparts $\hKL$ (respectively,  $\Kap$ and $\KL$).
More precisely, $\hKL$ is the subalgebra of $\hKb$ consisting of formal sums of the form
$\sum_{A \in \TtL} \xi_A [A]
\quad
(\xi_A \in \QQ(v)),$
$\Kap$ is the subalgebra of $\hKL$ generated by
$
E^{i,i+1}_{\breve n} [\vec{0}]^\fa, E^{i+1,i}_{\breve n} [\vec{0}]^\fa, 0[\bj]^\fa
\quad
(1\leq i \leq \breve n, \bj \in \ZZ_{n}),
$
and $\KL$ is the subalgebra of $\hKL$ consisting of $A[\mbf j]^{\fa}$ for all $A \in \tilde \Theta_{\breve n,1,0}$.
One can show that $\KL$ is generated by $E^{i,i+1}_{\breve n} [\vec{0}]^\fa, E^{i+1,i}_{\breve n} [\vec{0}]^\fa$ and $ 0[\bj]^\fa$ for
$2\leq i \leq \breve n-1, \bj \in \ZZ_{n}$.

%%%%%%%%%%%%%%%%%%
\section{Comultiplication on $\Kjj$}
 \label{sec:coKjj}

In this section we shall show
(in Theorem~\ref{thm:coK} and Theorem~\ref{thm:aqsp})  that $ \Kjj$ is a coideal subalgebra of $\Kn$ and $(\Kn,  \Kjj)$ forms a quantum symmetric pair.
The construction in \S\ref{sec:Kjj-quot} allows us to study the comultipication on $\Kjj$ via Chevalley generators.

Recall from \cite[\S 9.6]{FLLLW} a comultiplication
$\dot \Delta^{\fc}: \dKjj \longrightarrow \dKjj \otimes \dKn$.
In the same way there is a comultiplication of $\dot{\mathbf K}_{\breve n}^{\fc}$,
whose restriction to $\dKcL$ is denoted by
\eq
\dot \Delta^{\fc}_{\breve n}:
\dKcL \longrightarrow \dKcL \otimes \dKL.
\endeq
By \cite[Lemma 9.3.1, Proposition 9.3.4]{FLLLW}, the two comultiplications
 $\dot \Delta^{\fc}$ and $\dot \Delta_{\breve n}^{\fc}$ are compatible, i.e., the following diagram commutes:
\[
\xymatrix{
	\dKcL
	\ar[rr]^-{\dot \Delta^{\fc}_{\breve n}}
	\ar[d]_{q_1}
&
&
	\dKcL \otimes \dKL
	\ar[d]_{q_2}
\\
	\dKjj
	\ar[rr]^-{\dot \Delta^{\fc}}
&
&
	\dKjj \otimes \dKn,
}
\]
where both vertical maps are the canonical quotient maps.
Let $\mathcal J = \Ker(q_2)$.
By diagram chasing, we have $\dot{\Delta}^{\fc}_{\breve n}(\cI) \subseteq \mathcal J$.

By passing to completions, $\dot{\Delta}^{\fc}_{\breve n}$ and $\dot{\Delta}^{\fc}$ induce
the following comultiplications:
\eq  \label{Delta2}
\widehat{\Delta}_{\breve n}^{\fc}:
\hKcL
\rightarrow
\hKcL \otimes \hKL
\quad
\textup{and}
\quad
\widehat{\Delta}^{\fc}:
\widehat{\mathbf K}^{\fc}_{n}
\rightarrow
\widehat{\mathbf K}^{\fc}_{n} \otimes \widehat{\mathbf K}_{n}.
\endeq

\begin{thm}   \label{thm:coK}
The  restriction of $\widehat{\Delta}^{\fc}$ in \eqref{Delta2} to ${\Kjj}$ provides an algebra homomorphism
\[
\Delta^{\fc}=\widehat{\Delta}^{\fc} |_{\Kjj}: \Kjj \longrightarrow  \Kjj \otimes \Kn.
\]
\end{thm}

\begin{proof}
Denote by $\Delta_{\breve n}^{\fc}$ the restriction of $\widehat{\Delta}_{\breve n}^{\fc}$ to $\KcL$.
We first show $\Delta^{\fc}_{\breve n} (\Kcap) \subseteq   \Kcap \otimes \Kap$.
To this end, it suffices to check that $\Delta^{\fc}_{\breve n}(E^{i,i+1}_\tt[\vec{0}]) \in \Kcap \otimes \Kap$ for $1\leq i \leq \breve n$.

Fix $\mathbf a,\mathbf b \in \ZZ^\fc_n$, and let $\adKcb$ be the subspace of $\dKcb$ spanned by the standard basis elements $[A]$ such that $\roA(A) = \mathbf a$ and $\coA(A) = \mathbf{b}$.
Recall from~\cite[\S 9.6]{FLLLW} that
\[
\Delta^{\fc}_{\breve n} = (\Dcbab)_{\mathbf a', \mathbf b' \in \ZZ^\fc_n, \mathbf a'', \mathbf b'' \in \ZZ_n},
\]
where $
\Dcbab
:  \adKcb
\rightarrow
\adKcbp \otimes \adKbpp
$
is defined  for any $\mathbf a', \mathbf b' \in \ZZ^\fc_n, \mathbf a'', \mathbf b'' \in \ZZ_n$ such that $(\mathbf a', \mathbf a'') \models \mathbf a, (\mathbf b', \mathbf b'')\models \mathbf b$, or equivalently,  for
\[
\mathbf a = \mathbf a' + \mathbf a''_\tt,
\quad
\mathbf b = \mathbf b' + \mathbf b''_\tt.
\]
We have
\[
\Dcbab(E^{i,i+1}_\tt[\vec{0}])
= \sum_{\substack{\af \in \ZZ^\fc_{\breve n}\\ E^{i,i+1}_\tt+D_\af\in \XitL }}
\Dcbab([E^{i,i+1}_\tt + D_\af] ).
\]

We compute the contribution of $\Dcbab([E^{i,i+1}_\tt + D_\af] )$ to each $\adKcbp \otimes \adKbpp$. We have
\[
\bA{
\mathbf{a} &= \roA(E^{i,i+1}_\tt + D_\af) = \epsilon^i_\tt + \af = \mathbf{a}' +\mathbf{a}''_\tt,
\\
\mathbf{b} &= \coA(E^{i,i+1}_\tt + D_\af) = \epsilon^{i+1}_\tt + \af = \mathbf{b}' +\mathbf{b}''_\tt,}
\]
where
\[
\epsilon^i = \bc{
1&\tif i \equiv j \textup{ (mod } \breve n),
\\
0&\otw,
}\in\ZZ_{\breve n}
\]
 and $\epsilon^i_\tt = \epsilon^i + \epsilon^{-i} \in \ZZ^\fc_{\breve n}$.
In other words, for any quadruple $(\mathbf a', \mathbf b', \mathbf a'', \mathbf b'' ) \in \ZZ^\fc_n \times \in \ZZ^\fc_n \times \ZZ_n \times \ZZ_n$, $\Dcbab ([E^{i,i+1}_\tt + D_\af] )$ contributes to $\adKcbp \otimes \adKbpp$ if and only if $(\mathbf a', \mathbf b', \mathbf a'', \mathbf b'' )$  satisfies
\[
 \af
= \mathbf{a}' +\mathbf{a}''_\tt - \epsilon^i_\tt
=  \mathbf{b}' +\mathbf{b}''_\tt- \epsilon^{i+1}_\tt;
\]
In this case the contribution is computed explicitly by~\cite[Lemma 5.3.4]{FLLLW} as
\eq\label{eq:comultlKcnm}
\Dcbab ([E^{i,i+1}_\tt + D_\af] )
= \sum_{j=1}^3 g_j [B_j + D_{\beta^{(j)}}] \otimes \ ^\fa [C_j + D_{\gamma^{(j)}}],
\endeq
where $\beta^{(j)} \in \ZZ^\fc_{\breve n}, \gamma^{(j)} \in \ZZ_{\breve n}, B_j \in \Xizb, C_j \in \Tt^0_{\breve n}$ and $g_j \in \QQ(v) $ are give by
\[
\ba{{c|cc|cc|cccccccc}
j&B_j&C_j&\mathbf{a}'&\mathbf{b}'&\mathbf{a}''&\mathbf{b}''
\\
\hline
1& E^{i,i+1}_\tt&0
&\epsilon^i_\tt + \beta^{(1)}&\epsilon^{i+1}_\tt + \beta^{(1)}
&\gamma^{(1)}&\gamma^{(1)}
\\
2& 0 &E^{i,i+1}
&\beta^{(2)}&\beta^{(2)}
&\gamma^{(2)}+\epsilon^{i}&\gamma^{(2)}+\epsilon^{i+1}
\\
3& 0 & E^{-i,-i-1}
& \beta^{(3)}&\beta^{(3)}
&\gamma^{(3)}+\epsilon^{-i}&\gamma^{(3)}+\epsilon^{-i-1}
}
\]

\begin{align}
\label{eq:g1af}
g_1 &=
\bc{
%v^{\gamma^{(1)}_{\breve n-i} - \gamma^{(1)}_{\breve n-1-i}}=
v^{(\epsilon^{\breve n-i}-\epsilon^{\breve n-1-i})\cdot\gamma^{(1)}}
&\tif \af = \beta^{(1)} + \gamma^{(1)}_\tt,
\\
0
&\otw.
}
\\
  \label{eq:g2af}
g_2 &=
\bc{
%v^{-\delta_{i,0}+ \beta^{(2)}_{i}-\beta^{(2)}_{i+1}+\gamma^{(1)}_{\breve n-i} - \gamma^{(1)}_{\breve n-1-i}}=
v^{-\delta_{i,0}}
v^{(\epsilon^{i}-\epsilon^{i+1})\cdot\beta^{(2)}}
v^{(\epsilon^{\breve n-i}-\epsilon^{\breve n-1-i})\cdot\gamma^{(2)}}
&\tif \af = \beta^{(2)} + \gamma^{(2)}_\tt,
\\
0
&\otw.
}
\\
\label{eq:g3af}
g_3 &=
\bc{
1
&\tif \af = \beta^{(3)} + \gamma^{(3)}_\tt,
\\
0
&\otw.
}
\end{align}

Next we show that $\Delta^{\fc}_{\breve n} (E^{i,i+1}_{\theta}(\vec{0}))\in   \KcL \otimes \KL$
by assembling these together. Note that
\[
B[\bj] \otimes C[\bk]^\fa
=
\sum_{\substack{\beta \in \ZZ^\fc_{\breve n}\\ B+D_\beta\in \XitL }}
\sum_{\substack{\gamma \in \ZZ_{\breve n}\\ C+D_\gamma\in \~{\Tt}_{\breve n,1,0}}}
v^{\~{\bj}\bullet \beta}v^{\~{\bk}\cdot \gamma} [B + D_\beta] \otimes \ ^\fa [C + D_\gamma].
\]
Fix $B = B_j, C = C_j, \beta \in \ZZ^\fc_{\breve n}, \gamma \in \ZZ_{\breve n}$ such that $B+D_\beta\in \XitL$ and $C+D_\gamma\in \TtL$.
There is a unique $\af \in \ZZ^\fc_{\breve n}$ determined by either \eqref{eq:g1af}, \eqref{eq:g2af} or \eqref{eq:g3af} such that
$\Dcbab ([E^{i,i+1}_\tt + D_\af] )$ contributes to $B[\bj] \otimes C[\bk]^\fa$.
It can be verified that $E^{i,i+1}_\tt+D_\af\in \XitL$, and the contribution for such $\af$ is counted.
Therefore we have
\[
\bA{
\Delta^{\fc}_{\breve n}(E^{i,i+1}_\tt[\vec{0}] )
=&
E^{i,i+1}_\tt[\vec{0}] \otimes 0[\epsilon^{\breve n-i}-\epsilon^{\breve n-1-i}]^\fa
\\
&+0[\epsilon^{i}-\epsilon^{i+1}] \otimes E^{i,i+1}[\epsilon^{\breve n-i}-\epsilon^{\breve n-1-i}]^\fa
+0[\vec{0}] \otimes E^{-i,-i-1}[\vec{0}]^\fa.
}
\]
This completes the proof that $\Delta^{\fc}_{\breve n} (\Kcap) \subseteq   \Kcap \otimes \Kap$.

Therefore, for any $\ddot A[\vec{0}] \in \KcL$, by Lemma \ref{lem:Kcapgen} we have
$\Delta^{\fc}_{\breve n} (\ddot A[\vec{0}]) \in \Kcap \otimes \Kap + \mathcal J$.
We now consider the summand in $\Kcap \otimes \Kap$.
Since ${\roA}(\ddot A)_1 = {\coA}(\ddot A)_1=0$, by \cite[Proposition~ 9.3.4]{FLLLW},
the image of $\Delta^{\fc}_{\breve n, \mbf a', \mbf b', \mbf a'',\mbf b''} (\ddot A[\vec{0}])$
in $\Kcap \otimes \Kap$ is zero unless $a'_1=0, b'=1$ and $a''_1=0=b''_1$.
This implies that $\Delta^{\fc}_{\breve n} (\ddot A[\vec{0}]) \in \KcL \otimes \KL + \mathcal J$.
Finally, note that $\Delta^{\fc}_{\breve n}( \mathcal I) \subseteq \mathcal J$. The proposition follows.
\end{proof}

The coassociativity of $\dot \Delta^{\fc}$ descends to a similar one on $\Delta^{\fc}$ as formulated below.

\begin{cor}
The comultiplication $\Delta^{\fc}$ on $\K^{\fc}_n$ satisfies the coassoicativity property, i.e.,
$(1\otimes \Delta)  \Delta^{\fc} = (\Delta^{\fc} \otimes 1) \Delta^{\fc}$.
\end{cor}

\begin{prop}  \label{prop:embed}
There is a natural injective algebra homomorphism $\imath^{\fc}: \Kjj \longrightarrow \Kn$.
\end{prop}

\begin{proof}
The homomorphism $\imath^{\fc}: \Kjj \longrightarrow \Kn$ which we shall
construct should be regarded as a degenerate variant
of the comultiplication $\Delta^{\fc}: \Kjj \rightarrow \Kjj \otimes \Kn$, and the proof
here is similar to that for Theorem~ \ref{thm:coK}.
For the reader's convenience, we shall sketch the proof.

By setting $d'=0$ in ~\cite[Lemma~ 9.3.1]{FLLLW}, we have the following commutative diagram
\[
\xymatrix{
\Sjj \ar[rr]^{\iota_n} \ar[d]_{\rho_d} && \bS_{n,d} \ar[d]^{\rho_{d''}}\\
\bU^{\fc}_{\breve n,d} \ar[rr]^{\iota_{\breve n}} && \bU_{\breve n,d}
}
\]
By setting $d'=0$ in ~\cite[Proposition 9.3.4]{FLLLW}, we have a similar result for $\iota_n$ or
 $\iota_{\breve n}$ instead of $\Delta^{\fc}$.
 Moreover, we have the following commutative diagram
\[
\xymatrix{
\dKjj \ar[rr]^{\iota_n}  && \dKn \\
\dKcL \ar[rr]^{\iota_{\breve n}} \ar[u]^{q_1} && \dKL \ar[u]_{q_2}
}
\]
Recall that
\(
\cI = \Span_{\QQ(v)}\{[A] \in \dKcL ~|~ a_{11} < 0\}
\).
We also set \(
\mathcal L = \Span_{\QQ(v)}\{[A] \in \dKL ~|~ a_{11} < 0\}
\).
By diagram chasing, we have $\iota_{\breve n}(\cI) \subseteq \mathcal L$.

We now show that $\iota_{\breve n} ( \Kcap) \subseteq \Kap$.
By the definition of $\Kcap$ in \eqref{Kcap}, it suffices to show that
$\iota_{\breve n} ( E^{i,i+1}_{\tt, \breve n} [\vec{0}]) \in \Kap$.
We observe that for any $\lambda \in \mbb Z_{\breve n}^{\fc}$, $\iota_{\breve n}( D_{\lambda}) = D_{\lambda'}$ for some $\lambda' \in \mbb Z_{\breve n}$.
Hence, we have
\begin{equation*}
  \begin{split}
    \iota_{\breve n}( E^{i, i+1}_{\theta}[\vec{0}]) &= \sum_{\alpha \in \mbb Z_{\breve n}^{\fc}, E^{i, i+1}_{\theta}+ D_{\alpha} \in \widetilde {\Xi}_{\breve n,1,0}} \iota_{\breve n} ( [E^{i, i+1}_{\theta}+ D_{\alpha}])\\
    & = \sum_{\alpha \in \mbb Z_{\breve n}^{\fc}, E^{i, i+1}_{\theta}+ D_{\alpha} \in \widetilde {\Xi}_{\breve n,1,0}} \iota_{\breve n}(f_i D_{\alpha'})\\
    &= \sum_{\alpha \in \mbb Z_{\breve n}^{\fc}, E^{i, i+1}_{\theta}+ D_{\alpha} \in \widetilde {\Xi}_{\breve n,1,0}}
    e_{n-1-i}D_{\beta} + v^{\delta_{i0}}f_i k_{n-1-i} D_{\beta}\\
    &= E^{n-i, n-i-1}[\vec{0}]^{\fa} + v^{\delta_{i0}}E^{i,i+1}[\vec{0}]^{\fa} 0[\epsilon^{\breve n-i}- \epsilon^{\breve n-1-i}]^{\fa} \in \Kap,
  \end{split}
\end{equation*}
where $\alpha'\in \mbb Z_{\breve n}^{\fc}$ and $\beta\in \mbb Z_{\breve n}$ are determined by
$\coA( E^{i, i+1}_{\theta}+ D_{\alpha}) = \alpha'$ and $\iota_{\breve n}( D_{\alpha'}) = D_{\beta}$, respectively.

By tracking $\roA(A)_1$ and $\coA(A)_1$, we have $\iota_{\breve n}(\ddot A[\vec{0}]) \in \KL$, i.e. $\iota_{\breve n}( \KcL) \subseteq \KL$.
Therefore $\iota_{\breve n}(\KcL/\cI) \subseteq \KL/\mathcal L$, and this induces the desired
homomorphism $\imath^{\fc}: \Kjj \rightarrow \Kn$. The injectivity of $\imath^{\fc}$ follows from
the injectivity of $\iota_{\breve n}.$
The proposition is proved.
\end{proof}

Recall the notation of quantum symmetric pairs as defined in \cite{Le02} for finite type and in \cite{Ko14} for Kac-Moody type.
We rephrase Theorem~ \ref{thm:coK} and Proposition~\ref{prop:embed} as follows.
(We recall by Proposition~\ref{Prop:DF} that the algebra $\Kn$ is isomorphic to the quantum affine $\gl_n$.)
\begin{thm}
\label{thm:aqsp}
The pair $(\K_n, \K^{\fc}_n)$ forms a quantum symmetric pair of affine type.
\end{thm}

%=========================================================
\chapter{Stabilization algebras arising from other Schur algebras}
  \label{sec:coideal2}

In this chapter we formulate three more variants (denoted by $\jmath\imath, \imath\jmath, \imath\imath$) of
affine Schur algebras and their corresponding stabilization algebras.
We construct the standard, monomial, and stably canonical bases of these algebras.
We will present more details for the type $\imath\jmath$. We will merely be formulating the main statements for the remaining
types $\jmath\imath$ and $\imath\imath$.

Recall $n =2r+2$, and we now set
$$\nn =n-1=2r+1 \quad (r\geq 1).
$$

%=========================================================
\section{Affine Schur algebras of type $\imath\jmath$ }
  \label{sec:Sij}

Recall $\Xi_{n,d}$ from \eqref{eq:Xind}. Let
\eq \label{eq:Xij}
\Xij = \{ A \in \Xi_{n,d}~|~\roC(A)_0 = 0 = \roC(A)_0 \}.
\endeq
The additional condition for $A \in \Xi_{n,d}$ to be in $\Xij$  can be equivalently reformulated as
\[
\roA(A)_i
=\bc{ 1 &\tif i \equiv 0 \mod n,
\\
 0 &\otw.}
\]
Moreover, a general element $A\in \Xij$ \eqref{eq:Xij} is of the form below
(where $r\#\backslash c\#$ stands for row \# and column \# of the matrix):
\eq\label{eq:AinXij}
A=
\ba{{c||cc:c:ccc:c:ccc}
r\# \backslash c\#  &\cdots&{-1}&0&1&\cdots&{n-1}&n&{n+1}&\cdots
\\
\hline \hline
\vdots&\ddots&&\vdots&&&&\vdots
\\
{-1}&&a_{n-1,n-1}&0&a_{-1,1}&&*&0&*
\\
\hdashline
{0}&\cdots&0&1&0&\cdots&0&0&0&\cdots
\\
\hdashline
{1}&&a_{1,-1}&0&a_{11}&&*&0&*
\\
\vdots&&&\vdots&&\ddots&&\vdots
\\
{n-1}&&*&0&*&&a_{n-1,n-1}&0&a_{-1,1}
\\
\hdashline
{n}&\cdots&0&0&0&\cdots&0&1&0&\cdots
\\
\hdashline
{n+1}&&*&0&*&&a_{1,-1}&0&a_{11}
\\
\vdots&&&\vdots&&&&\vdots&&\ddots
\\
}
\endeq
%Here each entry $a_{ij}$ is encoded in the intersection of $r_i$ and $c_j$, and the dashed stripes indicate the dummy rows/columns.
Recall $\Ld=\Ld_{r,d}$ \eqref{def:Ld}.
Let
\begin{align*}
\Ldij &=\{ \ld = (\ld_i)_{i \in \ZZ} \in \Ld~|~  \ld_0 = 0 \},
\\
 \D_{\nn,d}^{\imath\jmath} &= \{(\ld, g, \mu) ~|~ \ld,\mu\in\Ldij, g\in\D_{\ld\mu} \}.
\end{align*}
The lemma below is the $\imath\jmath$-analog of Lemma~\ref{lem:kappa}, which follows by a similar argument.
%-----------------------------------------------------------------------------------------------------------
\begin{lem}\label{lem:kappai'}
%The restriction of $\kappa^{-1}$ on $\Xij$ is a bijection. In particular,
The map
$
\kappa^{\imath\jmath}:  \D_{\nn,d}^{\imath\jmath} \longrightarrow \Xij
$
sending $(\ld, g, \mu)$ to $(|R_i^\ld \cap g R_j^\mu|)$ is a bijection.
\end{lem}
%-----------------------------------------------------------------------------------------------------------
Now we define the {\em affine $q$-Schur algebra of type $\imath\jmath$} as
\eq
\Sij = \textup{End}_{\bH}
\big(
\mathop{\oplus}_{\ld\in \Ldij} x_\ld \bH \big).
\endeq
By definition the algebra $\Sji$ is naturally a subalgebra of $\Sjj$.
Moreover, both  $\{  e_A ~|~ A \in \Xij \}$ and $\{  [A] ~|~ A \in \Xij \}$ are bases of $\Sji$ as a free $\mA$-module.
%-----------------------------------------------------------------------------------------------------------
Note that although Algorithm~\ref{alg:mono} applies to arbitrary $A \in \Xij (\subset \Xi_{n,d})$,
the resulting matrices $A^{(i)}$ do not lie in $\Xij$ in general. In order to define a monomial basis for $\Sji$, we need a modified matrix interpretation by collapsing those fixed rows and columns indexed by $n\ZZ$ in \eqref{eq:Xij}.
Let
\[
\ZZij = \ZZ\backslash n\ZZ,
\]
and
\begin{align}
\label{eq:Xijp}
\begin{split}
\Xijp = \Big\{A=(a_{ij}) \in \text{Mat}_{\ZZij \times\ZZij}(\NN)~\big |~
& a_{-i,-j} =a_{ij}=a_{i+n, j+n}, \forall i, j\in \ZZij;
\\
&
a_{00}, a_{r+1,r+1} \text{ are odd};
\sum_{1 \leq i \leq n-1} \sum_{j\in\ZZij} a_{ij}= D-1
 \Big \}.
 \end{split}
\end{align}
That is, a general element $A\in \Xijp$ \eqref{eq:Xijp}  is of the form below:
\eq\label{eq:AinXijp}
A=
\ba{{c||cc|ccc|ccccc}
r\# \backslash c\#   &\cdots&{-1}&1&\cdots&{n-1}&{n+1}&\cdots
\\
\hline\hline
\vdots&\ddots&&&&&
\\
{-1}&&a_{11}&a_{1,-1}&&*&*
\\
\hline
{1}&&a_{1,-1}&a_{11}&&*&*
\\
\vdots&&&&\ddots&&
\\
{n-1}&&*&*&&a_{11}&a_{1,-1}
\\
\hline
{n+1}&&*&*&&a_{1,-1}&a_{11}
\\
\vdots&&&&&&&\ddots
\\
}
\endeq
Here the solid lines replace the dashed stripes in \eqref{eq:AinXij}.
Let $f^{\imath\jmath}_c: \Xij \rw \Xijp$ be the ``collapsing" map
by removing the columns/rows indexed by $n\ZZ$, and let
$f^{\imath\jmath}_e:  \Xijp \rw \Xij$ be the expanding map
by inserting the fixed columns/rows indexed by $n\ZZ$ as in \eqref{eq:AinXij}.
Clearly $f^{\imath\jmath}_c$ and $f^{\imath\jmath}_e$ are inverse maps to each other, and hence are bijections.

\exa
We shall follow the conventions of dashed/solids lines in \eqref{eq:AinXij} and \eqref{eq:AinXijp}.
%Note that due to the periodicity condition, the complete $r_i$/$c_j$ notation is not necessary once we indicate the dashed/solid lines and the period $n$.
Let $n=6$ (and hence $r=2$), and
\[
A
=\bM{[ccc:c:cccc]
3&&&&&&\\
*&4&&&&&\\
3&*&5&&&&\\
&2&*&0&1&&\\
\hdashline
&&0&1&0&\\
\hdashline
&&1&0&*&2&\\
&&&&5&*&3\\
&&&&&4&*\\
&&&&&&3\\
} \in \Xij,
\quad \text{centered at the $(0,0)$th entry}.
\]
That is, $A^\offd = \Ett^{1,-1} + 2\Ett^{12} + 3\Ett^{23} + 4\Ett^{32} + 5\Ett^{21}$ (cf. \eqref{def:offd} for notation).
Applying Algorithm~\ref{alg:mono}, we get
\[
[A^{(1)}] [A^{(2)}]
=
[A] + \textup{lower terms} \in \Sjj,
\]
where
\[
A^{(1)}
= \bM{[ccc:c:cccc]
&&&&&&\\
*&&&&&&\\
&*&&&&&\\
&&*&1&&&\\
\hdashline
&&&1&&\\
\hdashline
&&&1&*&&\\
&&&&&*&\\
&&&&&&*\\
&&&&&&\\
} \in \Xi_{n,d},
\quad
A^{(2)}
=
\bM{[ccc:c:cccc]
3&&&&&&\\
*&4&&&&&\\
3&*&5&&&&\\
&2&*&0&&&\\
\hdashline
&&1&1&1&\\
\hdashline
&&&0&*&2&\\
&&&&5&*&3\\
&&&&&4&*\\
&&&&&&3\\
}\in \Xi_{n,d}.
\]
Note that neither $A^{(1)}$ nor $A^{(2)}$ lies in $\Xij$.
On the other hand, we have
\[
f^{\imath\jmath}_c(A) =\bM{[ccc|cccc]
3&&&&&\\
*&4&&&&\\
3&*&5&&&\\
&2&*&1&&\\
\hline
&&1&*&2&\\
&&&5&*&3\\
&&&&4&*\\
&&&&&3\\
} \in\Xijp.
\]
That is,
 $f^{\imath\jmath}_c(A)^\offd = \Ett^{1,-1} + 2\Ett^{12} + 3\Ett^{23} + 4\Ett^{32} + 5\Ett^{21}$.
\endexa
\vspace{.2cm}

%================
\section{Monomial and canonical bases for $\Sij$}

While $A \in \Xij$ is not tridiagonal, $f^{\imath\jmath}_c(A) \in \Xijp$ is tridiagonal
in the following sense: $X =(x_{ij})\in\Xijp$ is called tridiagonal if
\eq\label{eq:tridij}
x_{ij} = 0
\quad
\tif
\quad
| \tau(i) - \tau(j) | > 1,
\endeq
where
\eq  \label{eq:tau}
\tau:\ZZij \longrightarrow \ZZ
\endeq
is the order-preserving bijection determined by setting $\tau(1)=1$ (and thus $\tau(-1)= 0$, and so on).
We claim that the set $\{[A] \in \Sij ~|~ A\in \Xij \text{ such that } f^{\imath\jmath}_c(A) \in \Xijp \textup{ is tridiagonal}\}$ is a generating set (even though in this case $A \in \Xij$ may not be tridiagonal).
%-----------------------------------------------------------------------------------------------------------
We shall now give an algorithm in the framework of $\Xij$ (compare Algorithm~\ref{alg:mono}).
%that generates $[A]$ for an arbitrary $A\in \Xij$ bypassing elements in $\Xijp$.
%-----------------------------------------------------------------------------------------------------------
\Alg\label{alg:mono'}
For any $A \in \Xij$,
we define matrices $\leftidx{^{\imath}}A^{(1)}, \ldots, \leftidx{^{\imath}}A^{(x)} \in \Xij$ as follows:
\enu
\item Initialization: set $t=0$, and set $C^{(0)} = f^{\imath\jmath}_c(A) \in \Xijp$.

\item If $C^{(t)} =(c_{ij}^{(t)})\in \Xijp$ is a tridiagonal matrix (cf. \eqref{eq:tridij}), then set $x=t+1$,
$A^{(x)} = C^{(t)},
%(C^{(t)})^\offd + \textup{a diagonal determined by } \eqref{eq:diag},
$
and the algorithm terminates.

Otherwise, set $k= \max \big\{ |\tau(i)-\tau(j)| ~\big |~ c^{(t)}_{ij} \neq 0 \big \} \geq 2$,
%Set $k \in \NN$ to be the smallest positive integer such that $c^{(t)}_{ij} = 0$ if $|\tau(i)-\tau(j)|>k$,
and
\[
\Tp^{(t)} =\sum_{i=1}^{n-1}
c^{(t)}_{i,\tau^{-1}(\tau(i)+k)} E^{i, \tau^{-1}(\tau(i)+k)}.
\]
%to be the matrix consisting of the outermost nonzero diagonals of $C^{(t)} \in \Xijp$.

\item
Define matrices
\[
\bA{
&\ds A^{(t+1)} = \sum_{i=1}^{n-1}
c^{(t)}_{i,\tau^{-1}(\tau(i)+k)} \Ett^{i,\tau^{-1}(\tau(i)+1)}
+ \textup{a diagonal determined by } \eqref{eq:diag},
\\
&\ds C^{(t+1)}= C^{(t)} - \Tp^{(t)}_\tt + {\V{\Tp^{(t)}} }_\tt,
}
\]
where $\V{X}$ is the matrix obtained from $X$ by shifting every entry down by one row; also see \eqref{Ttt} for notation.

\item Set $\leftidx{^{\imath}}A^{(t+1)} = f^{\imath\jmath}_e(A^{(t+1)})$.
Increase $t$ by one and then go to Step (2).
\endenu
\endAlg
\vspace{.2cm}

\thm\label{thm:min gen'}
For each $A \in \Xij$, the matrices $\leftidx{^{\imath}}A^{(j)} \in \Xij$,
for $j =1,\ldots, x$, in  Algorithm~\ref{alg:mono'} satisfy that
\eq \label{eq:semiMBij}
[ \leftidx{^{\imath}}A^{(1)}] [ \leftidx{^{\imath}}A^{(2)}]
 \cdots [ \leftidx{^{\imath}}A^{(x)}]
= [A] +\textup{lower terms} \in \Sij.
\endeq
\endthm

\proof
Applying Algorithm \ref{alg:mono} on $\leftidx{^{\imath}}A^{(t)}$, we obtain tridiagonal matrices
$D_1^{(t)}$ and $D_2^{(t)}$ in $\Xi_{n,d}$ satisfying that
\[
[D_1^{(t)}]   [D_2^{(t)}] =
[ \leftidx{^{\imath}}A^{(t)}] +\textup{lower terms} \in \Sjj,
\]
which implies from construction that
\[
[ \leftidx{^{\imath}}A^{(1)}]
[ \leftidx{^{\imath}}A^{(2)}]
\cdots [ \leftidx{^{\imath}}A^{(x)}]
= [A] +\textup{lower terms} \in \Sjj.
\]
It remains to show that each lower term $[D]$ in the above identity lies in $\Sij$.
By definition of multiplication on $\Sij$ we have $\roC(D) = \roC(A)$ and $\coC(D) = \coC(A)$.
In particular,  $D \in \Xij$ and hence $[D] \in \Sij$.
\endproof
%-----------------------------------------------------------------------------------------------------------
\exa
We illustrate how Algorithm \ref{alg:mono'} and proof of Theorem~\ref{thm:min gen'} work.
Let $n=4$ (and hence $r=1$), and
\[
A
=\bM{[cc:c:cccccc]
4&0&0&2&1&
\\
\hdashline
&0&1&0&
\\
\hdashline
1&2&0&0&4&
\\
&&&3&1&3&
\\
&&&&4&0&
\\
} \in \Xij.
\]
%That is, $A^\offd = 2\Ett^{1,-1} + \Ett^{1,-2}+ 4\Ett^{12}  + 3\Ett^{21}$, cf. \eqref{def:offd}.
We have
\[
C^{(0)} =\bM{[cc|cccc]
4&0&2&1&\\
\hline
1&2&0&4&\\
&&3&1&3\\
&&&4&0\\
} \in\Xijp.
\]
In this case, we have $k = 2$ and hence
\[
\Tp^{(0)} =\bM{[cc|cccc]
&0&&1&\\
\hline
1&&0&&\\
&&&0&\\
&&&&0\\
},
\quad
A^{(1)} =\bM{[c|cccc]
6&1&&\\
\hline
1&6&&\\
&&7&\\
&&&6\\
},
\quad
C^{(1)} =\bM{[c|cccc]
0&2&&\\
\hline
2&0&5&\\
&3&1&3\\
&&5&0\\
}\in \Xijp.
\]
Now $C^{(1)}$ is tridiagonal in $\Xijp$ and hence $A^{(2)} = C^{(1)}$, which yields that
\[
\leftidx{^{\imath}}A^{(1)}
=
\bM{[c:c:cccccc]
6&0&1&&
\\
\hdashline
0&1&0&
\\
\hdashline
1&0&6&&
\\
&&&7&&
\\
&&&&6&
\\
}
,
\quad
\leftidx{^{\imath}}A^{(2)}
=
\bM{[c:c:cccccc]
0&0&2&&
\\
\hdashline
0&1&0&
\\
\hdashline
2&0&0&5&
\\
&&3&1&3&
\\
&&&5&0&
\\
} \in \Xij.
\]
Therefore, we have
\[
D_1^{(1)}
=
\bM{[c:c:cccccc]
6&1&&
\\
\hdashline
&1&&
\\
\hdashline
&1&6&
\\
&&&7
\\
},
D_2^{(1)}
=
\bM{[c:c:cccccc]
0&0&&
\\
\hdashline
1&1&1
\\
\hdashline
&&6&
\\
&&&7
\\
},
D_1^{(2)}
=
\bM{[c:c:cccccc]
0&2&&
\\
\hdashline
0&1&0&
\\
\hdashline
&2&5&
\\
&&&7
\\
},
D_2^{(2)}
=
\bM{[c:c:cccccc]
0&0&&
\\
\hdashline
2&1&2
\\
\hdashline
&0&0&5
\\
&&3&1
\\
}.
\]
\endexa
\vspace{.3cm}

Modifying \eqref{eq:semiMBij}, for $A\in \Xij$, we define
\eq  \label{eq:MBij}
m_A = \{\leftidx{^{\imath}}A^{(1)}\}\,  \{\leftidx{^{\imath}}A^{(2)}\} \cdots  \{\leftidx{^{\imath}}A^{(x)}\}.
\endeq

Recall from Theorem~\ref{thm:CB-Sjj} the canonical basis $\mathfrak B_{n,d}^{\fc}$ for $\Sjj$.
The following is the the $\ij$-counterpart of Theorems~\ref{thm:CB-Sjj} and \ref{thm:mono} (for $\Sjj$).

\begin{thm}\label{thm:monoij}
{\quad}
\enu
\item[(a)]
The set $\{m_A~|~A \in \Xij \}$ forms an $\mA$-basis of $\Sij$. Moreover, we have
$m_A \in [A] +\sum_{B\in \Xij, B<_\alg A} \mA [B]$.

\item[(b)]
We have a canonical basis $\mathfrak B_{n,d}^{\imath\jmath} :=\{ \{A\}~|~A \in \Xij \}$ of $\Sij$ such that $\overline{\{A\} } =\{A\}$ and
$\{A\} \in [A] + \sum_{B\in \Xij, B<_\alg A} v \ZZ[v] [B]$. Moreover, we have
$\mathfrak B_{n,d}^{\imath\jmath} = \mathfrak B_{n,d}^{\fc}\cap \Sij$.
\endenu
\end{thm}

\proof
Part (a) is the counterpart of Theorem~\ref{thm:mono} (for $\Sjj$), and can be proved by the same argument, now with help from
Theorem~\ref{thm:min gen'}.
The first half of Part~(b) on the canonical basis follows by (a) and a standard argument (cf. \cite[24.2.1]{Lu93}).
The second half of (b) follows from the uniqueness characterization of the canonical basis $\mathfrak B_{n,d}^{\imath\jmath}$.
\endproof

%=========================================================
\section{Stabilization algebra of type $\imath\jmath$}
 \label{sec:Kij}

We define two subsets of $\Xit_n$  \eqref{eq:Xitn} as follows:
\eq
\Xit_n^< = \{A =(a_{ij}) \in \Xit_n ~|~ a_{00} < 0\},
\quad
\Xit_n^> = \{A  =(a_{ij})  \in \Xit_n ~|~ a_{00} > 0\}.
\endeq
For any matrix $A \in \Xit_n$ and $p\in\ZZ$, we define
\eq  \label{eq:ppA}
\pp{A} = A + p(I - E^{00}).
\endeq
%-----------------------------------------------------------------------------------------------------------
\lem \label{lem:stab-ij}
For $A_1, A_2,\ldots, A_f \in \Xit_n^>$, there exists $\mathcal Z_i \in \Xit_n^>$ and
$\mathcal G_i(v,v') \in \QQ(v) [v', v'{}^{-1}]$ ($i = 1, \ldots, m$ for some $m$) such that for all even integers
$p\gg 0$, we have an identity in $\Sjj$ of the form:
\[
[\pp{A_1}] \cdot [\pp{A_2}] \cdot \ldots \cdot  [\pp{A_f}] = \sum\limits_{i=1}^m \mathcal G_i(v, v^{-p}) [\pp{\mathcal Z_i}].
\]
\endlem

\proof
The proof is similar to the proof of Proposition~\ref{prop:stab1} where $\p{A} = A+pI$ is used instead of
$\pp{A}$ \eqref{eq:ppA}. A detailed calculation shows that replacing $p$ by $\breve{p}$ in \eqref{eq:stab1eq1}
and \eqref{eq:stab1eq2} gives similar formulas, which proves the lemma for the base case
($f=2, A=A_2, B=A_1$ is tridiagonal). The lemma then follows by induction.
\endproof
%-----------------------------------------------------------------------------------------------------------
Consequently, the vector space $\dKjjp$ over $\QQ(v)$ spanned by the symbols
$[A]$, for $A \in \Xit_n^>$, is a stabilization algebra whose multiplicative structure is given by
(with $f=2$; associativity follows from $f=3$):
\eq
[A_1] \cdot [A_2] \cdot \ldots \cdot [A_f]
= \bc{
\sum\limits_{i=1}^m \mathcal G_i(v, 1) [\mathcal Z_i]
&\tif \coC(A_i) = \roC(A_{i+1}) \; \forall i,
\\
0
&\textup{otherwise}.
}
\endeq
By arguments entirely analogous to those for Proposition~\ref{prop:stab2} and
Theorem~\ref{thm:dKjj CB},
$\dKjjp$ admits a (stabilizing) bar involution,
$\dKjjp$ admits  a monomial basis $\{m_A~|~ A \in \Xit_n^>\},$
and a  stably canonical basis $\dot{\mathfrak B}^{\fc,>}$.

\Def
Let $\dKij$ be the $\QQ(v)$-submodule of $\dKjjp$ generated by $\{ [A]~|~A \in \Xitij\}$, where
\eq  \label{dKij2}
\ba{{lll}
\Xitij &= \{A = (a_{ij}) \in \Xit_n ~|~ a_{0i} = a_{i0} = \delta_{0i} \}
\\
&= \{A \in \Xit_n^> ~|~ \coC(A)_0 = \roC(A)_0 = 0\}.
}
\endeq
\endDef
The goal of this section is to realize $\dKij$ as a subquotient of $\dKjj$ with compatible bases
by following \cite[Appendix~A]{BKLW}
(where an algebra $\dot\bU^\imath$ is realized as a subquotient of
an algebra $\dot\bU^\jmath$ with compatible  stably canonical bases).

It follows by the second characterization for $\Xitij$ in \eqref{dKij2} that $\dKij$ is a subalgebra of $\dKjjp$.
Since  the bar-involution on $\dKjjp$ restricts to an involution on $\dKij$, we reach the following conclusion.

\begin{lemma}  \label{lem:Kp}
The set $\dKij \cap \dot{\mathfrak B}^{\fc,>}$ forms a stably canonical basis of $\dKij$.
\end{lemma}

The submodule of $\dKjj$ spanned by $[A]$ for $A \in \Xitij$ is not a subalgebra.
This is why we need a somewhat different stabilization above to construct the stably canonical basis for $\dKij$.
We shall see below the stabilization above is related to the stabilization used earlier.

%Now we realize $\dKij$ as a subquotient of $\dKjj$.
Define $\Jql$ to be the $\QQ(v)$-submodule of $\dKjj$ spanned by $[A]$ for all $A \in \Xit_n^<$.
%-----------------------------------------------------------------------------------------------------------
\lem\label{lem:2ideal}
The submodule $\Jql$ is a two-sided ideal of $\dKjj$.
\endlem

\proof
Since $\Jql$  is clearly  invariant under the anti-involution $[A] \mapsto v^{- d_A + d_{{}^tA}} [{}^t A]$ for $\dKjj$,
the claim that $\Jql$ is left ideal of $\dKjj$ is equivalent to that $\mathcal{J}$ is a right ideal of $\dKjj$.
We shall show that $\Jql$ is left ideal of $\dKjj$. To that end,
it suffices to show that $[B]  [A] \in \Jql$ for arbitrary $A \in \Xit_n^<$ and tridiagonal $B\in \Xit_n$.

By the multiplication formula, the matrices corresponding to the terms showing up in $[B] \cdot [A]$ must be of the form
\[
A^{(T-S)} = A - (T-S)_\tt  + (\^{T-S})_\tt,
\quad
T \in \widetilde{\Theta}_{B, A}, 
\quad S \in \Gamma_T.
\]
Suppose that the $(0,0)$-entry $a_{00} - 2(t_{00}- s_{00}) + 2(\^{T-S})_{00}$ is positive. Note that we have
\[
\bA{
\LR{A;S;T} &=
\prod\limits_{(i,j) \in \Ia} \lrb{(A-T_\tt)_{ij} + s_{ij} + s_{-i,-j} + (\^{T-S})_{ij} + (\^{T-S})_{-i,-j}}{(A-T_\tt)_{ij} ; s_{ij} ; s_{-i,-j} ; (\^{T-S})_{ij} ; (\^{T-S})_{-i,-j}}
\\
&\cdot \prod\limits_{k \in \{0, r+1\}}
\left(
\frac{
\prod\limits_{i=1}^{s_{kk}+(\^{T-S})_{kk}}
[a_{kk}-2t_{kk} -1 + 2i]
}{[s_{kk}]^! [(\^{T-S})_{kk}]^!}
\right) \cdot  \LR{S}.
}
\]
Hence $\LR{A;S;T} = 0$ and  the term $[A^{(T-S)}]$ makes no contribution to $[B] [A]$. Therefore, we have $[B]  [A] \in \Jql$.
\endproof

Recall from \eqref{eq:mAK} a semi-monomial basis element
$m'_A= [A^{(1)}] [A^{(2)}] \cdots  [A^{(x)}]$,
and from \eqref{eq:mAK2} a monomial basis element
$m_A = \{A^{(1)}\}  \{A^{(2)}\} \cdots  \{A^{(x)}\}.$

\lem
We have $m'_A \in \Jql$, for $A \in \Xit_n^<$.
%In other words, $T_{A,B} = 0$ unless $B \in \Xit_n^<$.
\endlem
%-----------------------------------------------------------------------------------------------------------
\proof
From the construction of $A^{(i)}$ in $m'_A=[A^{(1)}] [A^{(2)}] \cdots  [A^{(x)}]$, the matrix $A^{(x)}$
has the same diagonal entries as $A$ and hence $[A^{(x)}] \in \Jql$. Thus it follows from Lemma \ref{lem:2ideal}
that $m'_A \in \Jql$.
\endproof

Recall $\dKjj$ admits a stably canonical basis of $\Bjj$ from Theorem~\ref{thm:dKjj CB}.
\begin{cor}\label{cor:monoJ}
The ideal $\Jql$ admits a monomial basis $\{m_A ~|~ A \in \Xit_n^<\}$, and
 a stably canonical basis $\Bjj \cap \Jql = \{\{A\} ~|~ A \in \Xit_n^<\}$.
\end{cor}

\proof
Recall the new basis $\f_A$ from ~\cite[Section 9.4]{FLLLW}.
By arguing in a way identical to the proof of ~\cite[Lemma A.6]{BKLW}, we obtain
that the $\f_A$ lies in $\bJ$ for all $A \in \Xit_n^<$ and hence forms a basis for $\bJ$.
Note that in the proof the role of Lemma A.5 in {\it loc. cit.} is played by Lemma \ref{lem:2ideal}.
By ~\cite[Proposition 9.4.4]{FLLLW}, we see that $\f_A$ is bar-invariant, and so is $\bJ$.
As a consequence, the set of $\{A\} $ for all $A\in \Xit_n^<$ is a basis of $\bJ$,
by an argument  similar to ~\cite[Proposition A.7]{BKLW}.
%Now by comparing $\f_A$ with $m_A$ for any $A \in \Xit_n^<$, one of the factors
%in $m_A= \{A^{(1)}\}  \{A^{(2)}\} \cdots  \{A^{(x)}\}$  must have a negative entry at $(0, 0)$.
So Lemma ~\ref{lem:2ideal} implies that the monomial
$m_A= \{A^{(1)}\}  \{A^{(2)}\} \cdots  \{A^{(x)}\}$ is in $\bJ$, for any $A \in \Xit_n^<$. The corollary follows.
\endproof

\prop  \label{prop:JK}
The following statements hold.
\enua
\item
The quotient algebra $\dKjj/\Jql$ admits a monomial basis $\{m_A + \Jql ~|~ A\in \Xit_n^>\}$.
\item
The quotient algebra $\dKjj/\Jql$ admits a stably canonical basis $\{\{A\} + \Jql ~|~ A \in \Xit_n^>\}$.
\item
The map $\sharp: \dKjj/\Jql \rw \dKjjp$ sending $[A] + \Jql \mapsto [A]$ is an isomorphism of $\QQ(v)$-algebras,
which matches the corresponding monomial bases and stably canonical bases.
\endenua
\endprop
%-----------------------------------------------------------------------------------------------------------
\proof
Parts~(a) and (b) follow directly from Corollary \ref{cor:monoJ}.
For (c), since the map $\sharp$ is clearly a linear isomorphism,
we still need to show that $\sharp$ is an algebra homomorphism. We write
\[
([B] + \Jql)\cdot([A] + \Jql) = \sum_C f_{B,A}^C ([C] + \Jql) \in \dKjj/\Jql,
\quad
[B]\cdot[A] = \sum_C g_{B,A}^C [C] \in \dKjjp
\]
for some structure constants $f_{B,A}^C$ and $g_{B,A}^C$ with $A,B,C \in \Xit_n^>$.
A detailed calculation shows that $f_{B,A}^C =g_{B,A}^C$ when $B$ is tridiagonal,
and hence $\sharp$ is an algebra homomorphism. Since $\sharp$ matches the tridiagonal generators,
it also matches the (semi-)monomial bases, and $\sharp$ commutes with the bar involutions.
Finally, $\sharp$ matches the stably canonical bases as the partial orders on the two algebras are compatible.
\endproof

We summarize Lemma~\ref{lem:Kp} and Proposition~\ref{prop:JK} above as follows.

\thm
  \label{thm:subqij}
As a $\QQ(v)$-algebra, $\dKij$ is isomorphic to a subquotient of $\dKjj$, with compatible standard, monomial, and stably canonical bases.
\endthm

\rmk  \label{rem:QSPij}
As in \S\ref{sec:Kjj-def}, \S\ref{sec:Kjj-quot} and \S\ref{sec:coKjj}, we can construct an algebra $\Kij$
as a subalgebra of a completion of the algebra $\dKij$.
Similar to Theorem~\ref{thm:aqsp}, the pair $(\K_{\nn}, \Kij)$ forms an affine quantum symmetric pair
associated to the involution as depicted in Figure~\ref{figure:ij}.
We omit the details.
\endrmk

\begin{figure}[ht!]
\caption{Dynkin diagram of type $A^{(1)}_{2r}$ with involution of type $\imath\jmath$.}
   \label{figure:ij}
\[
\begin{tikzpicture}
\matrix [column sep={0.6cm}, row sep={0.5 cm,between origins}, nodes={draw = none,  inner sep = 3pt}]
{
	&\node(U1) [draw, circle, fill=white, scale=0.6, label = 1] {};
	&\node(U3) {$\cdots$};
	&\node(U4)[draw, circle, fill=white, scale=0.6, label =$r-1$] {};
	&\node(U5)[draw, circle, fill=white, scale=0.6, label =$r$] {};
\\
	\node(L)[draw, circle, fill=white, scale=0.6, label =0] {};
	&&&&&
\\
	&\node(L1) [draw, circle, fill=white, scale=0.6, label =below:$2r$] {};
	&\node(L3) {$\cdots$};
	&\node(L4)[draw, circle, fill=white, scale=0.6, label =below:$r+2$] {};
	&\node(L5)[draw, circle, fill=white, scale=0.6, label =below:$r+1$] {};
\\
};
\begin{scope}
\draw (L) -- node  {} (U1);
\draw (U1) -- node  {} (U3);
\draw (U3) -- node  {} (U4);
\draw (U4) -- node  {} (U5);
\draw (U5) -- node  {} (L5);
\draw (L) -- node  {} (L1);
\draw (L1) -- node  {} (L3);
\draw (L3) -- node  {} (L4);
\draw (L4) -- node  {} (L5);
\draw (L) edge [color = blue, loop left, looseness=40, <->, shorten >=4pt, shorten <=4pt] node {} (L);
\draw (L1) edge [color = blue,<->, bend right, shorten >=4pt, shorten <=4pt] node  {} (U1);
\draw (L4) edge [color = blue,<->, bend left, shorten >=4pt, shorten <=4pt] node  {} (U4);
\draw (L5) edge [color = blue,<->, bend left, shorten >=4pt, shorten <=4pt] node  {} (U5);
\end{scope}
\end{tikzpicture}
\]
\end{figure}

By definition and Remark~\ref{K=Sjj}, $\Psi_d : \Kjj \rightarrow {}_{\QQ}\Sjj$ factors through $\Jql$ and hence we obtain a
surjective homomorphism $\tilde{\Psi}_d  : \Kjj/\Jql \rightarrow {}_{\QQ}\Sjj$.
We shall regard $\Kij \subset \dKjjp$ as a subalgebra of $\Kjj/\Jql$ via the identification
$\sharp: \dKjj/\Jql \cong \dKjjp$ in Proposition~\ref{prop:JK}.

Define ${}_{\QQ} \Sij =\QQ(v)\otimes_{\ZZ[v,v^{-1}]} \Sij$.
By restricting $\tilde{\Psi}_d$ to $\Kij$, we obtain a surjective homomorphism
$\Psi^{\imath\jmath}_d  : \Kij  \rightarrow {}_{\QQ} \Sij$.
Putting all constructions together we have obtained the following commutative diagram
\begin{align}
\label{CD:ij}
\begin{CD}
\Kij @>>> \Kjj/\Jql \\
@V\Psi_d^{\imath\jmath} VV @VV \tilde{\Psi}_d  V \\
{}_{\QQ} \Sij @>>> {}_{\QQ} \Sjj
\end{CD}
\end{align}

\prop  \label{K=Sij}
\enu
\item[(a)]
The homomorphism $\Psi^{\imath\jmath}_d  : \Kij  \rightarrow {}_{\QQ} \Sij$ sends
$[A] \mapsto [A]$ for $A \in \Xij$ and  $[A] \mapsto 0$ otherwise.

\item[(b)]
The homomorphism $\Psi^{\imath\jmath}_d  : \Kij  \rightarrow {}_{\QQ} \Sij$ preserves the (stably) canonical bases,
that is, it sends
$\{A\} \mapsto \{A\}$ for $A \in \Xij$ and $\{A\} \mapsto 0$ otherwise.
\endenu
\endprop

\proof
By definition, the 2 horizontal maps and $\tilde{\Psi}_d$ preserve the standard bases,
and so by diagram chasing of \eqref{CD:ij} we conclude that
$\Psi^{\imath\jmath}_d  : \Kij  \rightarrow {}_{\QQ} \Sij$ preserves the standard basis, whence (a).

By Lemma~\ref{lem:Kp} and Proposition~\ref{prop:JK}, the top horizontal map preserves
the stably canonical bases. The bottom horizontal map preserves
the stably canonical bases by construction.
The map $\tilde{\Psi}_d$ preserves the (stably) canonical bases by Remark~\ref{K=Sjj}.
Hence Part (b) follow by the same diagram chasing of  \eqref{CD:ij}.
\endproof
%=========================================================
\section{Stabilization algebra of type $\jmath\imath$ }

Recall $\Xi_{n,d}$ from \eqref{eq:Xind} and $\Ld=\Ld_{r,d}$ from \eqref{def:Ld}. Let
\begin{align*}
 %\label{eq:Xji}
\Xji &= \{ A \in \Xi_{n,d}~|~\roC(A)_{r+1} = 0 = \roC(A)_{r+1} \},
\\
\Ldji &=\{ \ld = (\ld_i)_{i \in \ZZ} \in \Ld~|~  \ld_{r+1} = 0 \},
\\
 \D_{\nn,d}^{\jmath\imath} &= \{(\ld, g, \mu) ~|~ \ld,\mu\in\Ldji, g\in\D_{\ld\mu} \}.
\end{align*}

The following lemma is obtained by restriction of the bijection $\kappa$ in Lemma~\ref{lem:kappa} to
$ \D_{\nn,d}^{\jmath\imath}$.
\begin{lem}  \label{lem:kappaji}
The map
$
\kappa^{\jmath\imath}:  \D_{\nn,d}^{\jmath\imath}  \longrightarrow \Xji
$
sending $(\ld, g, \mu)$ to $(|R_i^\ld \cap g R_j^\mu|)$ is a bijection.
\end{lem}
%-----------------------------------------------------------------------------------------------------------
Now we denote the {\em affine $q$-Schur algebra of type $\jmath\imath$} by
\eq
\Sji = \textup{End}_{\bH}
\big(
\mathop{\oplus}_{\ld\in \Ldij} x_\ld \bH \big).
\endeq
All results for $\Sij$, $\dKij$ and $\Kij$ in \S\ref{sec:Sij}--\ref{sec:Kij} admit counterparts in the current setting.
We shall formulate the main statements but skip the details to avoid much repetition.
The proposition below is an analogue of Theorem~\ref{thm:monoij}.
\begin{prop}  \label{prop:MBCBji}
The algebra $\Sji$ is naturally a subalgebra of $\Sjj$, admitting compatible standard, monomial and canonical bases.
\end{prop}

By repeating the constructions in \S\ref{sec:Sij}--\ref{sec:Kij} of the algebras of type $\imath\jmath$,
we can obtain an associative algebra $\dKji$ generated by $\{[A] | A\in \Xitji \}$, where
\eq
\ba{{lll}
\Xitji &= \{A = (a_{ij}) \in \Xit_n ~|~ a_{r+1,i} = a_{i,r+1} = \delta_{r+1,i} \}
\\
&= \{A \in \Xit_n^> ~|~ \coC(A)_{r+1} = \roC(A)_{r+1} = 0\}.
}
\endeq
Similarly, we construct a subalgebra $\Kji$ in a completion of $\dKji$.

%-----------------------------------------------------------------------------------------------------------
\begin{thm} \label{thm:main-ji}
The following statements hold for $\dKji$.
\enu
\item[(a)] The algebra $\dKji$ admits  a monomial basis and a  stably canonical basis.
\item[(b)] The algebra $\dKji$ is a subquotient of $\dKjj$ with compatible standard, monomial, and stably canonical bases.
\item[(c)] The pair $(\K_{n-1}, \Kji)$ forms an affine quantum symmetric pair associated to the involution as depicted in Figure~\ref{figure:ji}.
\endenu
\end{thm}

\begin{figure}[ht!]
\caption{Dynkin diagram of type $A^{(1)}_{2r}$ with involution of type $\jmath\imath$.}
   \label{figure:ji}
\[
\begin{tikzpicture}
\matrix [column sep={0.6cm}, row sep={0.5 cm,between origins}, nodes={draw = none,  inner sep = 3pt}]
{
	\node(U1) [draw, circle, fill=white, scale=0.6, label = 0] {};
	&\node(U2)[draw, circle, fill=white, scale=0.6, label =1] {};
	&\node(U3) {$\cdots$};
	&\node(U5)[draw, circle, fill=white, scale=0.6, label =$r-1$] {};
\\
	&&&&
	\node(R)[draw, circle, fill=white, scale=0.6, label =$r$] {};
\\
	\node(L1) [draw, circle, fill=white, scale=0.6, label =below:$2r$] {};
	&\node(L2)[draw, circle, fill=white, scale=0.6, label =below:$2r-1$] {};
	&\node(L3) {$\cdots$};
	&\node(L5)[draw, circle, fill=white, scale=0.6, label =below:$r+1$] {};
\\
};
\begin{scope}
\draw (U1) -- node  {} (U2);
\draw (U2) -- node  {} (U3);
\draw (U3) -- node  {} (U5);
\draw (U5) -- node  {} (R);
\draw (U1) -- node  {} (L1);
\draw (L1) -- node  {} (L2);
\draw (L2) -- node  {} (L3);
\draw (L3) -- node  {} (L5);
\draw (L5) -- node  {} (R);
\draw (R) edge [color = blue,loop right, looseness=40, <->, shorten >=4pt, shorten <=4pt] node {} (R);
\draw (L1) edge [color = blue,<->, bend right, shorten >=4pt, shorten <=4pt] node  {} (U1);
\draw (L2) edge [color = blue,<->, bend right, shorten >=4pt, shorten <=4pt] node  {} (U2);
\draw (L5) edge [color = blue,<->, bend left, shorten >=4pt, shorten <=4pt] node  {} (U5);
\end{scope}
\end{tikzpicture}
\]
\end{figure}

\begin{rem}
  By Theorem \ref{thm:multformula} and a similar argument as that for ~\cite[Theorem~ 4.8]{FL16},
  one obtain an algebra isomorphism $\Sij \simeq \Sji$.
  Applying the stabilization process, one further establishes an algebra isomorphism $\dKij \cong \dKji$.
\end{rem}

%=========================================================
\section{Stabilization algebra of type $\imath\imath$ }

In the following we deal with the variant of affine Schur algebra of type $\imath\imath$
corresponding to the involution as depicted Figure~\ref{figure:ii}. We set
\[
\eta =n-2 =\nn -1.
\]

\begin{figure}[ht!]
\caption{Dynkin diagram of type $A^{(1)}_{2r-1}$ with involution of type $\imath\imath$.}
   \label{figure:ii}
\[
\begin{tikzpicture}
\matrix [column sep={0.6cm}, row sep={0.5 cm,between origins}, nodes={draw = none,  inner sep = 3pt}]
{
	&\node(U1) [draw, circle, fill=white, scale=0.6, label = 1] {};
	&\node(U2) {$\cdots$};
	&\node(U3)[draw, circle, fill=white, scale=0.6, label =$r-1$] {};
\\
	\node(L)[draw, circle, fill=white, scale=0.6, label =0] {};
	&&&&
	\node(R)[draw, circle, fill=white, scale=0.6, label =$r$] {};
\\
	&\node(L1) [draw, circle, fill=white, scale=0.6, label =below:$2r-1$] {};
	&\node(L2) {$\cdots$};
	&\node(L3)[draw, circle, fill=white, scale=0.6, label =below:$r+1$] {};
\\
};
\begin{scope}
\draw (L) -- node  {} (U1);
\draw (U1) -- node  {} (U2);
\draw (U2) -- node  {} (U3);
\draw (U3) -- node  {} (R);
\draw (L) -- node  {} (L1);
\draw (L1) -- node  {} (L2);
\draw (L2) -- node  {} (L3);
\draw (L3) -- node  {} (R);
\draw (L) edge [color = blue, loop left, looseness=40, <->, shorten >=4pt, shorten <=4pt] node {} (L);
\draw (R) edge [color = blue,loop right, looseness=40, <->, shorten >=4pt, shorten <=4pt] node {} (R);
\draw (L1) edge [color = blue,<->, bend right, shorten >=4pt, shorten <=4pt] node  {} (U1);
\draw (L3) edge [color = blue,<->, bend left, shorten >=4pt, shorten <=4pt] node  {} (U3);
\end{scope}
\end{tikzpicture}
\]
\end{figure}

Let
\[
\Xii = \Xij \cap \Xji, \qquad
 \Ldii = \Ldji \cap \Ldij, \qquad
 \D_{\eta,d}^{\imath\imath} =  \D_{\nn,d}^{\imath\jmath} \cap  \D_{\nn,d}^{\jmath\imath}.
 \]
 The following lemma is obtained by restriction of the bijection $\kappa$ in Lemma~\ref{lem:kappa}.

\begin{lem} \label{lem:kappai'}
The map
$
\kappa^{\imath\imath}: \D_{\eta,d}^{\imath\imath} \longrightarrow \Xii
$
sending $(\ld, g, \mu)$ to $(|R_i^\ld \cap g R_j^\mu|)$ is a bijection.
\end{lem}
%-----------------------------------------------------------------------------------------------------------
Now we define the {\em affine $q$-Schur algebra of type $\imath\imath$} by
\eq
\Sii = \textup{End}_{\bH}
\big(
\mathop{\oplus}_{\ld\in \Ldii} x_\ld \bH \big).
\endeq
The algebra
$\Sii$ is naturally a subalgebra of $\Sij, \Sji$ and $\Sjj$, admitting compatible standard, monomial and canonical bases
(similar to Proposition~\ref{prop:MBCBji} for $\Sji$).

By a similar process, we construct an associative algebra $\dKii$ generated by $\{[A]| A\in \Xitii\}$, where
\eq
\Xitii = \Xitij \cap  \Xitji .
\endeq
We collect the main results for $\dKii$ which are similar to $\dKij$ and $\dKji$ earlier
(see Theorem~\ref{thm:main-ji}) in the following theorem.
The proofs are very similar to the previous cases, and hence skiped.

\begin{thm}
  \label{thm:subqii}
The following statements hold for $\dKii$.
\enu
\item[(a)] The algebra $\dKii$ admits  a monomial basis and a  stably canonical basis.
\item[(b)] The algebra $\dKii$ is a subquotient of $\dKji$ (and of $\dKij$, respectively), with compatible standard, monomial, and stably canonical bases.
\item[(c)] The pair $(\K_{\eta}, \Kii)$ forms an affine quantum symmetric pair associated to the involution as depicted in Figure~\ref{figure:ii}.
\endenu
\end{thm}
\vspace{.3cm}

The interrelation among Schur algebras as well as stabilization algebras of the four different types can be summarized below.
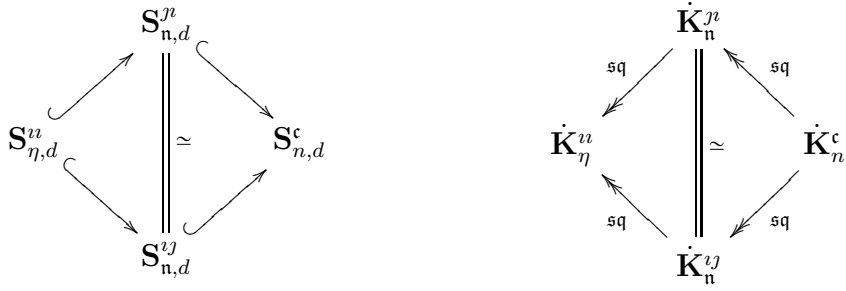
\begin{equation*}
\xymatrix{
&\Sji  \ar@{^{(}->}[dr]
\ar@{=}[dd]^{\simeq}
&&\\
\Sii  \ar@{^{(}->}[ur]  \ar@{^{(}->}[dr]  && \Sjj   \\
& \Sij \ar@{^{(}->}[ur] &&
}
\qquad\qquad
\xymatrix{
&\dKji  \ar@{->>}[dl]_{\mathfrak{sq}}  \ar@{=}[dd]^{\simeq}
&&\\
\dKii  && \dKjj \ar@{->>}[ul]_{\mathfrak{sq}}  \ar@{->>}[dl]^{\mathfrak{sq}} \\
& \dKij \ar@{->>}[ul]^{\mathfrak{sq}} &&
}
\end{equation*}
On the Schur algebra level, we have a commuting diagram for inclusions of Schur algebras.
On the stabilization algebra level, we have the following commutative diagram of subquotients,
where the notation $\mbf K_1 \overset{\mathfrak{sq}}{\twoheadrightarrow} \mbf K_2$ stands for
the statement that $\mbf K_2$ is a subquotient of $\mbf K_1$.
All the subquotients between various pairs of algebras preserve the  stably canonical bases.

%%%%===============================================================================================================
%%%%%%%%%%%%%%%%%%%
%%%%%%%%%%%%%%%%%%%
\appendix
\chapter[Length formulas in symmetrized forms]{Length formulas in symmetrized forms \for{toc}{by Zhaobing Fan, Chun-Ju Lai, Yiqiang Li and Li Luo}}
\label{chap:lengthformula}

\begin{center}
Z. Fan, C. Lai, Y. Li and L. Luo
\end{center}

%The purpose of this appendix is to give a geometric explanations of length formulas
%for double coset representatives in Weyl groups of finite and affine classical type.

In this appendix, we  provide a proof of the claim in Section~\ref{sjj} that
the symmetrized length formula (\ref{def:dA}) is the dimension of a generalized Schubert varieties/cells.
Thanks to~\cite{FLLLW}, we only need to match this formula with (4.1.1) in {\it loc. cit.}; see Proposition
~\ref{eq:l^Csym}.
We then further deduce length formulas in a similar fashion for all classical types.
As an application, we also formulate a symmetrized length formula for finite and affine classical Weyl groups,
which is more compact and easier to compute.
It is our hope that these analyses will facilitate a future study of $q$-Schur algebras of affine type $\tB$ and $\tD$.

%The generalized Schubert cells and varieties have played a fundamental role of the geometric construction  of quantum Schur algebras for finite and affine classical type.
%Their dimension formula were used as a twist in the definition of standard basis.
%We establish a new dimension formula as a symmetrized sum that recovers the dimension formulas in the literature.
%As an application, we obtain new length formulas of Weyl groups of finite and affine type in a symmetrizing fashion.

%----------------------------------------------------------------------------
\section{Dimension of generalized Schubert varieties}
%----------------------------------------------------------------------------

%----------------------------------------------------------------------------
\subsection{General setting}
%----------------------------------------------------------------------------
The Beilinson-Lusztig-MacPherson (BLM) stabilization procedure \cite{BLM90} is a geometric construction of the (modified) quantum groups of finite type $\tA$ together with their canonical bases.
The BLM-type stabilization procedures beyond type $\tA$ have important application in the theory of quantum symmetric pairs, and are developed rapidly (cf. \cite{Lu99,DF14} for affine type $\tA$, \cite{BKLW} for type $\tB/\tC$, \cite{FL15} for type $\tD$, \cite{FLLLW} for affine type $\tC$).

An initial step in the stabilization procedure is to construct the standard basis.
A geometric approach relies heavily on the generalized Schubert varieties; while an algebraic approach is also possible.
We start with an imprecise type-free setup for generalized Schubert varieties, and we leave the details in subsequent sections.

Fix non-negative integers $n,d \in \NN$ and an algebraically closed field, $k$, with positive characteristic.
Let $G$ be a classical group over $k$,
and $\cX_{n,d}$ be the set of $n$-step partial flags of rank $d$ over $k$ admitting $G$-action.
%The group $G$ acts on $\cX_{n,d}\times \cX_{n,d}$ diagonally.
Let $G\backslash (\cX_{n,d}\times \cX_{n,d})$ be the set of $G$-orbits on $\cX_{n,d}\times \cX_{n,d}$ via the diagonal action.
Denote the convolution algebra %(which we call the {\it quantum Schur algebra})
on the pairs of partial flags over the ring $\cA = \ZZ[v,v\inv]$ by
\eqs
\bS_{n,d}
= \cA_G(\cX_{n,d}\times \cX_{n,d})
= \{f : G\backslash (\cX_{n,d}\times \cX_{n,d}) \to \cA\}.
\endeqs
Except for types $\tD, \taffB, \taffD$, the $G$-orbits are in bijection with a matrix set $\cO_{n,d}$ which parametrizes any linear basis of $\bS_{n,d}$. For instance, the characteristic basis $\{e_A~|~A\in\cO_{n,d}\}$ defined by
\eqs
e_A(O_B) = \bc{1&\tif B=A;\\0 &\otw,}
\endeqs
where $O_B$ is the $G$-orbit corresponding to the matrix $B\in \cO_{n,d}$.
For $A\in \cO_{n,d}$ and fixed $L \in \cX_{n,d}$,
we define the {\it generalized Schubert cell} by
\eqs %\label{XAL}
X_A^L =\{ L' \in \cX_{n,d} | (L, L') \in O_A\}.
\endeqs

Fix a pair $(L, L')\in O_A$ and
let $P_{L}$ and $P_{L'}$ be the stabilizers of $L$ and $L'$, respectively, in $G$.
Then $X_A^L$ is isomorphic to the $P_L$-orbit of $L'$ in the partial flag variety $G/P_{L'}$ indexed by $A$.
If $P_L$ is a Borel subgroup, then the $P_L$-orbits in $G/P_{L'}$ are called Schubert cells in literature.
This is the reason why we call  $X^L_A$ a generalized Schubert cell.

The dimension function $\^\ell: \cO_{n,d} \to \NN$ is given by
\eqs
\^\ell(A) = \dim X_A^L.
\endeqs
We remark that $\^\ell(A)$ is independent of the choice of $L$, and it is also the dimension
of the closure of $X_A^L$, in an appropriate topology, which is called a generalized Schubert variety.

If $G$ is replaced by its loop group, then one can still define in a similar manner generalized Schubert cells with parabolics replaced
by parahorics. We shall not repeat this procedure again, but see Sections~\ref{Affine-A} and~\ref{Affine-C}.
From now on, we assume that $G$ and $\cX_{n,d}$ admit an $\mbb F_q$-structure.
Then their $\mbb F_q$-points can be described as follows.

\subsection{Type $\tA$}
%----------------------------------------------------------------------------
The results here are due to \cite{BLM90}.
Let $G = \GL(d,\FF_q)$ be the general linear group over the finite field $\FF_q$. We set
\eqs
\cX^\tA_{n,d} = \{(0=V_0 \subseteq \ldots \subseteq V_n = \FF_q^d)\}.
\endeqs
The set of $G$-orbits is well-known to be
\eqs \textstyle
\cO^\tA_{n,d} = \{ (a_{ij}) \in \tMat_{n\times n}(\NN) | \sum_{1\leq i,j\leq n} a_{ij} = d\}.
\endeqs
%----------------------------------------------------------------------------
\prop
The dimension function on $\cO^\tA_{n,d}$ is given by
\eqs
\^\ell^\tA(A) = \sum_{1\leq i,j \leq n}\sum_{\substack{x\leq i\\ y>j}} a_{ij} a_{xy}.
\endeqs
\endprop
%----------------------------------------------------------------------------
\proof
See \cite[2.3]{BLM90}, in which $\^\ell^\tA(A)$ is referred as $d(A) - r(A)$.
\endproof
%----------------------------------------------------------------------------
%----------------------------------------------------------------------------
\subsection{Type $\taffA$}
\label{Affine-A}
%----------------------------------------------------------------------------
The results here are essentially from \cite{Lu99} (see also \cite[\S 2]{FLLLW}).
Let $F = \FF_q((\ep))$ be the field of formal Laurent series,
and let $G = \GL(d, F)$ be the general linear group.
Let $V$ be an $F$-vector space of dimension $d$.
A {\it lattice} in $V$ is a  free $\FF_q[[\ep]]$-submodule of $V$ of rank $d$.
We set
\begin{align*}
\cX^{\taffA}_{n,d}
&= \{(L_i)_{i\in\ZZ} |L_i \textup{ are lattices in }V, L_{i-1} \subseteq L_i, L_i = \ep L_{i+n}\},
\\
I^{\taffA}
&= [1..n]\times\ZZ,
\\
\cO^{\taffA}_{n,d}
&= \{(a_{ij}) \in \tMat_{\ZZ\times \ZZ}(\NN) | \sum\limits_{(i,j) \in I^{\taffA}} a_{ij} = d, a_{ij} = a_{i+n,j+n}\}.
\end{align*}
%----------------------------------------------------------------------------
\prop
The dimension function on $\cO^\tA_{n,d}$ is given by
\eqs
\^\ell^{\taffA}(A) =  \sum_{(i,j) \in I^{\taffA}} \sum_{x\leq i, y>j} a_{ij} a_{xy}.
\endeqs
\endprop
%----------------------------------------------------------------------------
\proof
See \cite[Lemma~4.3]{Lu99}, in which $\^\ell^{\taffA}(A)$ is referred as $d'_A$.
\endproof
%----------------------------------------------------------------------------
%----------------------------------------------------------------------------
\subsection{Type $\tB$}
%----------------------------------------------------------------------------
Fix $n = 2r+1, D = 2d+1$,
and fix a non-degenerate symmetric bilinear form $Q:\FF_q^{D}\times\FF_q^{D}\to \FF_q$.

We follow \cite[\S2]{BKLW} with a shift of index from $[1..n]$ to $[-r..r]$.
Set
\begin{align*}
G &= O(D, \FF_q) = \{A \in \GL(D,\FF_q) | Q(Au,Av) = Q(u,v) \textup{ for }u,v \in \FF_q^{D}\},
\\
\cX^{\tB}_{n,d} &= \{(0=V_{-r} \subseteq \ldots \subseteq V_r = \FF_q^D)| V_i = V^\perp_{-i}\},
\\
I^{\tB} &= (\{0\} \times [1..r]) \sqcup ([1..r] \times [-r..r]).
\end{align*}
For $(i,j) \in I^{\tB}$, let
\eqs
a^\natural_{ij} = \bc{
\lfloor\frac{a_{ij}}{2}\rfloor &\tif (i,j) = (0,0);
\\
a_{ij} &\otw.
}
\endeqs
The $G$-orbits are
\eqs\textstyle
\cO^\tB_{n,d} = \left\{
(a_{ij}) \in \tMat_{[-r..r]\times [-r..r]}(\NN)
\middle|
\ba{{c}
a_{00}\textup{ is odd},
\sum_{(i,j) \in I^{\tB}} a^\natural_{ij} = d
\\
 a_{-i,-j} = a_{ij}
 }
 \right\}.
\endeqs
%----------------------------------------------------------------------------
\prop
The dimension function on $\cO^\tB_{n,d}$ is given by
\eqs
\^\ell^\tB(A) =  \frac{1}{2}\sum_{(i,j) \in I^{\tB}} \left(\sum_{\substack{x\leq i\\ y>j}}+\sum_{\substack{x\geq i\\ y<j}}\right) a^\natural_{ij} a_{xy}.
\endeqs
\endprop
%----------------------------------------------------------------------------
\proof
From \cite[(3.16)]{BKLW}, it is obtained that
\eqs
%\^\ell^\tB(A) =\sum_{\substack{i>  k, \;  j<l  \\ i+k<N+1}} a_{ij}a_{kl}
%+\sum_{\substack{i<n+1\;\mrm{\small{or}}\; j<N+1-l \\j< l}} a_{ij}a_{il}
%+\sum_{i\geq n+1> j} a_{ij}(a_{ij}-1)/2.
\^\ell^\tB(A) =\sum_{\substack{i>  k, \;  j<l  \\ k<-i}} a_{ij}a_{kl}
+\sum_{\substack{i<0\;\mrm{\small{or}}\; j<-l \\j< l}} a_{ij}a_{il}
+\sum_{i\geq 0 > j} a_{ij}(a_{ij}-1)/2.
\endeqs
It then follows by a direct calculation.
\endproof
%----------------------------------------------------------------------------
\exa
Let $r = 1, d = 2$, and $A \in \cO^\tB_{n,d}$ be the matrix corresponding to the Weyl group element
$(2,-2)$, i.e.,
\[
A =
\bM{
&&1
\\
&\gc3&\gc
\\
\gc1&\gc&\gc
}.
\]
Here the entries in $I^\tB$ are shaded. We then have
\[
\ba{{c|ccccc}
(i,j)&a_{ij}&a^\natural_{ij}&\{(x,y)|\substack{x\leq i\\ y>j} ~\tor \substack{x\geq i\\ y<j}\}&\textup{contribution to }2\^\ell(A)
\\
\hline
(0,0)&3&1&
\bM{~&&\gc1 \\ ~\gc&3&\gc \\ \gc1&&}
&2
\\
(1,-1)&1&1&
\bM{~&\gc&\gc1 \\ ~&\gc3&\gc \\ 1&\gc&\gc}
&4
}
\]
Hence, $\^\ell^\tB(A) = \frac{1}{2} (2 + 4) = 3 = \ell^\tB((2,-2))$ since $X_A^L$ is a genuine Schubert cell.
\endexa
%----------------------------------------------------------------------------
\subsection{Type $\tC$}
%----------------------------------------------------------------------------

Now we fix $n = 2r+1, D = 2d$,
and fix a non-degenerate skew-symmetric bilinear form $Q:\FF_q^{D}\times\FF_q^{D}\to \FF_q$. We follow \cite[\S6.2]{BKLW} with a shift of index from $[1..n]$ to $[-r..r]$.
Set
\begin{align*}
G &= \SP(D, \FF_q) = \{A \in \GL(D,\FF_q) | Q(Au,Av) = Q(u,v) \textup{ for }u,v \in \FF_q^{D}\},
\\
\cX^{\tC}_{n,d} &= \{(0=V_{-r} \subseteq \ldots \subseteq V_r = \FF_q^D)| V_i = V^\perp_{-i}\},
\\
\textstyle
\cO^\tC_{n,d} &= \left\{
(a_{ij}) \in \tMat_{[-r..r]\times [-r..r]}(\NN)
\middle|
\ba{{c}
a_{00}\textup{ is even},
\sum_{(i,j) \in I^{\tB}} a^\natural_{ij} = d
\\
 a_{-i,-j} = a_{ij}
 }
 \right\}.
\end{align*}
%----------------------------------------------------------------------------
\prop
The dimension function on $\cO^\tC_{n,d}$ is given by
\eqs
\^\ell^\tC(A) =  \^\ell^\tB(A+E_{00}),
\endeqs
where $E_{00} = (\delta_{i0}\delta_{j0})_{ij}$.
\endprop
%----------------------------------------------------------------------------
\proof
It follows from \cite[Proposition 6.7]{BKLW} that there is an algebra isomorphism $\bS^\tB_{n,d}\to \bS^\tC_{n,d}$ sending $e_A$ to $e_{A- E_{00}}$.
\endproof
%----------------------------------------------------------------------------
%----------------------------------------------------------------------------
\subsection{Type $\tD$}
%----------------------------------------------------------------------------
Fix $n=2r+1$. We use a variant of \cite[\S3]{FL15} with $D = 2d+1$ and a shift of index from $[1..n]$ to $[-r..r]$.

We set
\begin{align*}
G &= \SO(D, \FF_q) = \{A \in O(D,\FF_q) | \det A = 1\},
\\
\cX^{\tD}_{n,d} &= \{(0=V_{-r} \subseteq \ldots \subseteq V_r = \FF_q^D)| V_i = V^\perp_{-i}\}.
\end{align*}
%Following \cite{FL15}, for $A \in \cO^\tB_{n,d}$ we define
%\eq
%\textstyle
%\ro(A) = (\sum_{k=- r}^r a_{ik})_i,
%\quad
%\co(A) = (\sum_{k=- r}^r a_{kj})_j,
%\quad
%\ur(A) = \sum_{i<0, j > 0} a_{ij}.
%\endeq
The (unsigned) $G$-orbits are
\eqs
\label{def:OD}\textstyle
\cO^\tD_{n,d} = \left\{
(a_{ij}) \in \tMat_{[-r..r]\times [-r..r]}(\NN)
\middle|
\ba{{c}
a_{00}\textup{ is odd}, \sum_{(i,j) \in I^{\tB}} a^\natural_{ij} = d
\\
 a_{-i,-j} = a_{ij}
 }
 \right\}.
\endeqs
%\rmk (will be removed in the final version)
%The central entry is made odd so that it matches the length formula.
%{\red Shouldn't there be a condition \eqref{def:OD} that is analogous to $\ur(\sig) \in 2\ZZ $ for $\sig\in W^\tD_d$?}
%\endrmk
%----------------------------------------------------------------------------
\prop
The dimension function on $\cO^\tD_{n,d}$ is given by
\eqs
\^\ell^\tD(A) =  \frac{1}{2}
\Bp{
\sum_{(i,j) \in I^\tB}
\Bp{ 	\sum_{\substack{x \leq i\\ y >j}}
	+
	\sum_{\substack{x \geq i\\ y <j}}
	}
a^\natural_{ij} a_{xy}
-\sum_{(i,j)=(0,0)}\Bp{ 	\sum_{\substack{x \leq i\\ y >j}}
	+
	\sum_{\substack{x \geq i\\ y <j}}
	}
a_{xy}
}.
\endeqs
\endprop
%----------------------------------------------------------------------------
\proof
From \cite[Lemma~4.5.1]{FL15}, it is obtained that
\eq\label{eq:D1}
\^\ell^\tD(A) =
	\frac{1}{2} ( \sum_{-r\leq i,j \leq r}\sum_{x\leq i, y>j} a_{ij} a_{xy})
	- \sum_{k \geq 0 > l} a_{kl}.
\endeq
Due to the centrosymmetry condition $a_{ij} = a_{-i,-j}$, we have
\eq
\sum_{\substack{-r \leq i \leq -1 \\ -r \leq j \leq r}}
\sum_{\substack{x\leq i\\ y>j}} a_{ij} a_{xy}
=
\sum_{\substack{1 \leq -i \leq r \\ -r \leq -j \leq r}}
\sum_{\substack{-x \geq -i\\ -y <-j}} a_{-i,-j} a_{-x,-y}
=
\sum_{\substack{1 \leq i \leq r \\ -r \leq j \leq r}}
\sum_{\substack{x \geq i\\ y <j}} a_{ij} a_{xy},
\endeq
and hence
\eq\label{eq:D2}
\sum_{\substack{i \in [-r,-1] \cup [1,r]\\ -r \leq j \leq r}}
\sum_{x\leq i, y>j} a_{ij} a_{xy}
=
\sum_{\substack{(i,j) \in I^\tB\\1 \leq i \leq r}}
\Bp{ 	\sum_{\substack{x \leq i\\ y >j}}
	+
	\sum_{\substack{x \geq i\\ y <j}}
	}
a^\natural_{ij} a_{xy}.
\endeq
For $i=0$, we have $a_{00} = 2a^\natural_{00}$, $a^\natural_{0j} = a_{0j} = a_{0,-j}$ for $j> 0$, and thus
\eq\label{eq:D3}
\sum_{\substack{i =0\\ -r \leq j \leq r}}
\sum_{x\leq i, y>j} a_{ij} a_{xy}
=
\sum_{(0,j) \in I^\tB}
\Bp{ 	\sum_{\substack{x \leq i\\ y >j}}
	+
	\sum_{\substack{x \geq i\\ y <j}}
	}
a^\natural_{ij} a_{xy}.
\endeq
Finally,
\eq\label{eq:D4}
\sum_{k \geq 0 > l} a_{kl} =
\frac{1}{2}\sum_{(i,j)=(0,0)}\Bp{ 	\sum_{\substack{x \leq i\\ y >j}}
	+
	\sum_{\substack{x \geq i\\ y <j}}
	}
a_{xy}.
\endeq
The proposition follows from combining \eqref{eq:D1}, \eqref{eq:D2} -- \eqref{eq:D4}.
\endproof
%----------------------------------------------------------------------------
\exa
Let $r = d = 4$, and let $A \in \cO^\tD_{n,d}$ be the permutation matrix corresponding to the type $\tD$ simple reflection, i.e.,
\[
A = \bM{
1&&&&&&&&
\\
&1&&&&&&&
\\
&&&&&1&&&
\\
&&&&&&1&&
\\
 &&&&\gc 1&\gc&\gc&\gc&\gc
\\
\rowcolor{gray!40}&&1&&&&&&
\\
\rowcolor{gray!40}&&&1&&&&&
\\
\rowcolor{gray!40}&&&&&&&1&
\\
\rowcolor{gray!40}&&&&&&&&1
}
\]
Here the entries in $I^\tD$ are shaded. We obtain
\[
\ba{{c|ccccc}
(i,j)&(0,0)&(1,-2)&(2,-1)&(3,3)&(4,4)
\\
\hline
a_{ij}=a^\natural_{ij}&0&1&1&1&1
\\
\{(x,y)|\substack{x\leq i\\ y>j} ~\tor \substack{x\geq i\\ y<j}\}
&
\scalebox{0.4}{$\bM{
1&&&&&\gc&\gc&\gc&\gc
\\
&1&&&&\gc&\gc&\gc&\gc
\\
&&&&&\gc1&\gc&\gc&\gc
\\
&&&&&\gc&\gc1&\gc&\gc
\\
 \gc&\gc&\gc&\gc&1&\gc&\gc&\gc&\gc
\\
\gc&\gc&\gc1&\gc&&&&&
\\
\gc&\gc&\gc&\gc1&&&&&
\\
\gc&\gc&\gc&\gc&&&&1&
\\
\gc&\gc&\gc&\gc&&&&&1
}$}

&
\scalebox{0.4}{$\bM{
1&&&\gc&\gc&\gc&\gc&\gc&\gc
\\
&1&&\gc&\gc&\gc&\gc&\gc&\gc
\\
&&&\gc&\gc&\gc1&\gc&\gc&\gc
\\
&&&\gc&\gc&\gc&\gc1&\gc&\gc
\\
 &&&\gc&\gc1&\gc&\gc&\gc&\gc
\\
\gc&\gc&1&\gc&\gc&\gc&\gc&\gc&\gc
\\
\gc&\gc&&1&&&&&
\\
\gc&\gc&&&&&&1&
\\
\gc&\gc&&&&&&&1
}$}
&
\scalebox{0.4}{$\bM{
1&&&&\gc&\gc&\gc&\gc&\gc
\\
&1&&&\gc&\gc&\gc&\gc&\gc
\\
&&&&\gc&\gc1&\gc&\gc&\gc
\\
&&&&\gc&\gc&\gc1&\gc&\gc
\\
 &&&&\gc1&\gc&\gc&\gc&\gc
\\
&&1&&\gc&\gc&\gc&\gc&\gc
\\
\gc&\gc&\gc&1&\gc&\gc&\gc&\gc&\gc
\\
\gc&\gc&\gc&&&&&1&
\\
\gc&\gc&\gc&&&&&&1
}$}
&
\scalebox{0.4}{$\bM{
1&&&&&&&&\gc
\\
&1&&&&&&&\gc
\\
&&&&&1&&&\gc
\\
&&&&&&1&&\gc
\\
 &&&&1&&&&\gc
\\
&&1&&&&&&\gc
\\
&&&1&&&&&\gc
\\
\gc&\gc&\gc&\gc&\gc&\gc&\gc&1&\gc
\\
\gc&\gc&\gc&\gc&\gc&\gc&\gc&&1
}$}
&
\scalebox{0.4}{$\bM{
1&&&&&&&&
\\
&1&&&&&&&
\\
&&&&&1&&&
\\
&&&&&&1&&
\\
 &&&&1&&&&
\\
&&1&&&&&&
\\
&&&1&&&&&
\\
&&&&&&&1&
\\
\gc&\gc&\gc&\gc&\gc&\gc&\gc&\gc&1
}$}
}
\]
Hence, $\^\ell^\tD(A) = \frac{1}{2}(3+3-4) = 1$.
\endexa
\rmk
We could define
\eqs
a^\dagger_{ij} = \bc{a^\natural_{ij} -1 &\tif (i,j) =(0,0);\\ a_{ij} &\otw,}
\endeqs
and get an alternative formula as below:
\eqs
\^\ell^\tD(A) =  \frac{1}{2}
\sum_{(i,j) \in I^\tB}
\Bp{ 	\sum_{\substack{x \leq i\\ y >j}}
	+
	\sum_{\substack{x \geq i\\ y <j}}
	}
a^\dagger_{ij} a_{xy}
.
\endeqs
\endrmk
%----------------------------------------------------------------------------
\subsection{Type $\taffC$}
\label{Affine-C}
%----------------------------------------------------------------------------
The setup here are drawn from \cite[\S 3]{FLLLW}.
Fix $n =2r$.
Recall that $F = \FF_q((\ep))$ is the field of formal Laurent series.
Define matrices
\begin{align*} \label{J}
%\begin{split}
J =
\left(\begin{array}{ccccc}
0 & 0  & \cdots & 0 & 1\\
0 & 0  & \cdots & 1 & 0 \\
. & . & \cdots & . & . \\
1 & 0 & \cdots & 0 & 0
\end{array}
\right)_{d \times d},
\quad
M  = M_{2d} =
\left(
\begin{array}{cc}
0 & J\\
- J & 0
\end{array}
\right).
%\end{split}
\end{align*}
Let $V =F^{2d}$ be a symplectic vector space over $F$ with a symplectic form
$(,): V \times V \to F$ specified by $M$.
Let $^t A$ be the transpose of a matrix $A$.
Let $G$ be the symplectic group with coefficients in $F$, namely,
\eqs
G = \SP_F(2d) = \{ A \in \GL(2d,F) |  A = M \ \! ^t \!A^{-1} M^{-1} \}.
\endeqs
For any lattice $L$  of $V$, we set
\[
L^{\#}=\{ v\in V | (v, L)\subset \FF_q[[\ep]]\}.
\]
%Then the $\FF_q[[\ep]]$-module $L^{\#}$ is again a lattice of $V$ and $(L^{\#})^{\#}= L$.
%We shall use freely the following properties: for any two lattices $\mathcal L$ and $\mathcal M$
%\[
%(\mathcal L + \mathcal M)^{\#} = \mathcal L^{\#} \cap \mathcal M^{\#}, \qquad
%(\mathcal L \cap \mathcal M)^{\#} = \mathcal L^{\#} + \mathcal M^{\#}.
%\]
A lattice in $V$ is {\it symplectic} if both conditions below hold:
\enu
\item Either $L \subseteq L^{\#}$ or  $L \supseteq L^{\#}$;
\item Either $L$ or $L^{\#}$ is homothetic to a lattice $\Lambda$, i.e., $L$  or $L^{\#}$ is equal to $\ve^{a} \Lambda$ for some $a\in \mbb Z$, such that $\ep \Lambda  \subseteq \Lambda^{\#}  \subseteq  \Lambda$.
\endenu
We set
\begin{align*}
\cX^{\taffC}_{n,d}
&= \left\{(L_i)_{i\in\ZZ} \middle|\substack{L_i \textup{ are symplectic lattices in }V\\
 L_{i-1} \subseteq L_i, L_i = \ep L_{i+n}, L_i^\# = L_{-i-1}}\right\},
\\
I^{\taffC}
&= (\{0\} \times \NN) \sqcup ([1..r-1] \times\ZZ) \sqcup (\{r\}\times \ZZ_{\leq r}),
\\
a^\natural_{ij}
&= \bc{
\lfloor\frac{a_{ij}}{2}\rfloor  &\tif (i,j) \in \ZZ(r,r);
\\
a_{ij} &\otw,
}
\\
\textstyle
\cO^{\taffC}_{n,d}
&=
\left\{(a_{ij}) \in \tMat_{\ZZ\times \ZZ}(\NN)
\middle|
\ba{{c}
a_{00}, a_{rr} \textup{ are odd},
\sum_{(i,j) \in I^{\taffC}} a^\natural_{ij} = d
\\
a_{-i,-j} = a_{ij} = a_{i+n,j+n}
}
\right\}.
\end{align*}
We note that, under a similar embedding as \eqref{iota}, $a^\natural_{ij}$ is the same as $a_{ij}'$ in \eqref{def:a'}.
%----------------------------------------------------------------------------
\prop
The dimension function on $\cO^{\taffC}_{n,d}$ is given by
\eqs
\label{eq:l^Csym}
\^\ell^{\taffC}(A) =  \frac{1}{2}\sum_{(i,j) \in I^{\taffC}} \left(\sum_{\substack{x\leq i\\ y>j}}+\sum_{\substack{x\geq i\\ y<j}}\right) a^\natural_{ij} a_{xy}.
\endeqs
\endprop
%----------------------------------------------------------------------------
\proof
By rephrasing \cite[Lemma~4.2.1]{FLLLW}, we get
\eq\label{eq:l^C}
\^\ell^{\taffC}(A)
= \frac{1}{2} \Big (\sum_{\substack{i\geq k, j<l \\ 0 \leq i < n} } a_{ij} a_{kl}
-  \sum_{k\geq 0  >l } a_{kl}
- \sum_{k\geq  r >l } a_{kl}
\Big ).
\endeq
It suffices to prove that \eqref{eq:l^C} coincides with the symmetrized form \eqref{eq:l^Csym}.
For $i \neq 0, r$, we have
\eq\label{eq:l^C1}
\bA{
\sum_{\substack{i\geq k, j<l \\ i  \in [1..r-1] \cup [r+1 .. n-1]} } a_{ij} a_{kl}
&=
\sum_{1\leq i \leq r-1}
\Bp{
\sum_{\substack{x \leq i\\y >j}}
a_{ij} a_{xy}
+
\sum_{\substack{x \leq -i\\y >-j}}
a_{-i,-j} a_{xy}
}
\\
&=
\sum_{\substack{(i,j)\in I^{\taffC}\\ 1\leq i \leq r-1}}
\Bp{
\sum_{\substack{x \leq i\\y >j}}
+
\sum_{\substack{x \geq i\\y <j}}
}
a^\natural_{ij} a_{xy}
}
\endeq
For $i=0$, we have
\eq\label{eq:l^C2}
\bA{
\sum_{\substack{i\geq k, j<l \\ i = 0 } } a_{ij} a_{kl} - \sum_{k\geq 0 < l} a_{kl}
&=
\ds
\sum_{\substack{i=0\\ j\in \ZZ}}
\sum_{\substack{x \leq i\\y >j}} a_{0j} a_{xy}
-\sum_{\substack{i=0\\j=0}}\sum_{\substack{x \leq i\\y >j}} (1) a_{xy}
\\
&=
\ds
\sum_{(0,j) \in I^{\taffC}}
\sum_{\substack{x \leq 0\\y >j}} 2a^\natural_{0j} a_{xy}
=
\sum_{\substack{(i,j) \in I^{\taffC}\\ i = 0}}
\Bp{
\sum_{\substack{x \leq i\\y >j}}
+\sum_{\substack{x \geq i\\y <j}}
}
a^\natural_{ij} a_{xy}
}
\endeq
Similarly, we have
\eq\label{eq:l^C3}
\sum_{\substack{i\geq k, j<l \\ i = r } } a_{ij} a_{kl}
- \sum_{k\geq  r  >l } a_{kl}
=
\sum_{\substack{(i,j) \in I^{\taffC}\\ i = r}}
\Bp{
\sum_{\substack{x \leq i\\y >j}}
+\sum_{\substack{x \geq i\\y <j}}
}
a^\natural_{ij} a_{xy}
\endeq
Summing up \eqref{eq:l^C1} -- \eqref{eq:l^C3}, we are done.
\endproof
%----------------------------------------------------------------------------
%----------------------------------------------------------------------------

%----------------------------------------------------------------------------
\subsection{Summary}
%----------------------------------------------------------------------------
We summarize the dimension formulas in previous sections in the following.

\begin{thm}\label{thm:A1}
The dimension of generalized Schubert cell $X_A^L$ and its associated generalized Schubert variety, for $A \in \cO_{n,d}$, is given by
\begin{align*}
%\label{eq:lAA}
\^\ell^{\tA}(A) &=  \frac{1}{2}\sum_{(i,j) \in I^{\tA}} a^\natural_{ij} \sum_{x,y} a_{xy},
\\
\^\ell^{\tB}(A) &=  \frac{1}{2}\sum_{(i,j) \in I^{\tB}} a^\natural_{ij} \sum_{x,y} a_{xy},
\\
\^\ell^{\tD}(A) &=  \frac{1}{2}\sum_{(i,j) \in I^{\tB}} a^\dagger_{ij} \sum_{x,y} a_{xy},
\\
\^\ell^{\taffA}(A) &=  \frac{1}{2}\sum_{(i,j) \in I^{\taffA}} a^\natural_{ij} \sum_{x,y} a_{xy},
\\
%\label{eq:lCA}
\^\ell^{\taffC}(A) &=  \frac{1}{2}\sum_{(i,j) \in I^{\taffC}} a^\natural_{ij} \sum_{x,y} a_{xy},
%\\
%\^\ell^{\taffB}(A) &{\color{red}=  \frac{1}{2}\sum_{(i,j) \in I^{\taffC}} a^\dagger_{ij} \sum_{x,y} a_{xy}?}
%\\
%\^\ell^{\taffD}(A) &{\color{red}=  \frac{1}{2}\sum_{(i,j) \in I^{\taffC}} a^{\dagger\dagger}_{ij} \sum_{x,y} a_{xy}?}
\end{align*}
where $x,y$ are summed over the set $\{(x,y)|\substack{x\leq i\\ y>j} ~\tor \substack{x\geq i\\ y<j}\}$, and
 $a^\dagger_{ij}$ is equal to $a_{ij}$ except
\eqs
a^{\dagger}_{00}  = a^\natural_{00} -1.
\endeqs
\end{thm}
\rmk
It is known that a matrix $A\in \cO_{n,d}$ corresponds to a triplet $(\ld, g, \mu)$ where $\ld$ and $\mu$ label parabolic subgroups $W_\ld, W_\mu$ of the Weyl group $W$, and $g$ is the minimal length representative in the double coset $W_\ld g W_\mu$.
Hence, Theorem \ref{thm:A1} provides a geometric interpretation for the double coset representatives in the corresponding Weyl group.
\endrmk
%----------------------------------------------------------------------------
\section{Length formulas of Weyl groups}
%----------------------------------------------------------------------------
The Weyl groups are important examples of the Coxeter groups,
which admit length functions that count the total number of simple reflections in a reduced expression.
It is well-known that classical Weyl groups are identified with certain groups of (signed) permutations,
and hence the lengths are obtained by counting the inversions.
By using Theorem~\ref{thm:A1}, we are able to obtain a length formula
for finite and affine classical groups, which we believe much simpler than previous ones available in literature.

%----------------------------------------------------------------------------
\subsection{Length formulas}
%----------------------------------------------------------------------------
Realizations of  affine Weyl groups as infinite permutation groups were first mentioned by Lusztig \cite{Lu83} and further studied by B\'edard \cite{B86}, Shi \cite{Shi94} and Bj\"oner-Brenti \cite{BB96}.
A unified study of infinite permutation groups can be found in \cite{EE98},
in which the affine inversions are described.

We first recall some standard combinatorial statistics following \cite{BB05}.
For any integer interval $I \subseteq \ZZ$, denote by $\Perm I$ the set of permutations on $I$.
For $I \in \{ [1..d], [-d..d], \ZZ\}$ and $g \in \Perm I$, we define
\eq\label{def:neg}
\tneg(g) = {}^\sharp\{ i\in [1..d] | g(i) <0\},
\quad
\tnsp(g) = {}^\sharp\{(i,j) \in [1..d] | i+j<0\}.
\endeq
Denote the numbers of type $\tA/\tB$ inversions for $g\in \Perm I$ by
\eq\label{def:inv}
\tinv(g) = {}^\sharp\{(i,j)| i<j, g(i) > g(j)\},
\quad
\tinv_\tB(g) = \tinv(g) + \tneg(g) + \tnsp(g).
\endeq
We further define the sum of entries in the ``upper-right corner'' of the $(k,k)$th entry by
\eq
\ur_{k}(g) = {}^\sharp\{ i \in \ZZ_{<k}| g(i) > k\}.
\endeq
For finite and affine classical types, we identify the Weyl groups with (infinite) permutations as below:
%\begin{align}
\begin{equation}
\label{def:W}
\begin{split}
%\label{def:WA}
W^\tA_{d-1} &= \{ g \in \Perm[1..d]\},
\\
%\label{def:WB}
W^\tB_d &= \{ g \in \Perm[-d..d]| g(-i) = -g(i)\},
\\
%\label{def:WD}
W^\tD_d &= \{ g \in W^\tB_d| \ur_0(g) \textup{ is even} \},
\\
%\label{def:WaA}
W^{\taffA}_d &\textstyle= \{ g \in \Perm\ZZ| g(i+d) = g(i)+d, \sum_{i=1}^d g(i) ={d+1 \choose 2}\},
\\
%\label{def:WaC}
W^{\taffC}_d &= \{ g \in \Perm\ZZ| g(i+2d) = g(i)+2d, g(-i) = -g(i)\},
\\
%\label{def:WaB}
W^{\taffB}_d &= \{ g \in W^{\taffC}_d| \ur_d(g)  \textup{ is even} \},
\\
%\label{def:WaD}
W^{\taffD}_d &= \{ g \in W^{\taffC}_d| \ur_0(g),\ur_d(g) \textup{ are even}\}.
\end{split}
\end{equation}
%\end{align}

%----------------------------------------------------------------------------
\begin{prop}\label{prop:A2}
Let $\tinv, \tnsp, \tinv_\tB$ be functions defined as in \eqref{def:neg} -- \eqref{def:inv}.
Let $\ell^{\tX}$ be the length function on the Weyl group $W^\tX_d$ (see \eqref{def:W}) of type $\tX_d$. Then
\begin{align*}
\ell^\tA(g) &= \tinv(g),
\\
\ell^\tB(g) &= \tinv(g)%{}^\sharp\{(i,j)\in [d]^2 ~|~ \substack{i<j \\g(i) > g(j)}\}
			+{}^\sharp\{(i,j)\in [d]^2 ~|~ \substack{i\leq j \\g(-i) > g(j)}\},
\\
\ell^\tD(g)&= \tinv(g)%{}^\sharp\{(i,j)\in [d]^2 ~|~ \substack{i<j \\g(i) > g(j)}\}
			+{}^\sharp\{(i,j)\in [d]^2 ~|~ \substack{i< j \\g(-i) > g(j)}\},
\\
\ell^{\taffA}(g)&=\textstyle \sum\limits_{1\leq i <j \leq d} \left|\left\lfloor\frac{g(j) - g(i)}{d}\right\rfloor\right|,
\\
\ell^{\taffC}(g)&=\textstyle \tinv_\tB(g) + \sum\limits_{1\leq i \leq j \leq d} \left(\left\lfloor\frac{|g(j) - g(i)|}{d}\right\rfloor + \left\lfloor\frac{|g(j) + g(i)|}{d}\right\rfloor\right),
\\
\ell^{\taffB}(g)&=\textstyle  \tinv_\tB(g) + \sum\limits_{1\leq i < j \leq d} \left(\left\lfloor\frac{|g(j) - g(i)|}{d}\right\rfloor + \left\lfloor\frac{|g(j) + g(i)|}{d}\right\rfloor\right) + \sum\limits_{1\leq i \leq d} \left\lfloor\frac{|g(i)|}{d}\right\rfloor,
\\
\ell^{\taffD}(g)&=\textstyle \tinv(g) + \tnsp(g) + \sum\limits_{1\leq i < j \leq d} \left(\left\lfloor\frac{|g(j) - g(i)|}{d}\right\rfloor + \left\lfloor\frac{|g(j) + g(i)|}{d}\right\rfloor\right).
\end{align*}
\end{prop}
\proof
See \cite{EE98}.
\endproof
 %----------------------------------------------------------------------------
%Our main result (Theorem~\ref{thm:main}) is establishing the generalized length formulas for finite and affine classical type, which  equals to the dimension of generalized Schubert variety.
%As a corollary, we obtain an alternative formulation (Theorem~\ref{thm:main2}) of the length formulas above (Proposition~\ref{prop:l}).

%For any integer interval $X \subseteq \ZZ$, denote by the set of $X\times X$ permutation matrices by
%\eq
%\Pi_X := \{A =(a_{ij}) \in \tMat_{X\times X}([1]) |
%\textstyle
%\sum_{k} a_{ik} = 1 = \sum_{k} a_{kj} \textup{ for all }i,j \in X\}.
%\endeq
%For $A \in \Pi_X$, denote the sum of entries in the ``upper-right'' corner by
%\eq
%\ur(A) = \sum_{i<0<j} a_{ij}.
%\endeq

%\begin{align}
%W^\tA_{d-1} &= \Pi_{[d]},
%\\
%W^\tB_d &= \{ A \in \Pi_{[-d,d]}| a_{ij} = a_{-i,-j}\},
%\\
%W^\tD_d &= \{ g \in W^\tB_d| \ur(A) \in 2\NN \},
%\\
%W^{\taffA}_d &= \{ A \in \Pi_{\ZZ} | a_{ij} = a_{i+d,j+d}, \},
%\\
%W^{\taffC}_d &= \{ g \in \Perm[d+1]\},
%\\
%W^{\taffB}_d &= \{ g \in \Perm[d+1]\},
%\\
%W^{\taffD}_d &= \{ g \in \Perm[d+1]\},
%\end{align}
%----------------------------------------------------------------------------
\subsection{New length formulas}
%----------------------------------------------------------------------------
As an application of our dimension formulas (Theorem~\ref{thm:A1}), we achieve new length formulas for finite and affine Weyl groups in a symmetrized fashion.

We introduce the symmetrized inversion function  $\tinv_{I\times J}$ on integer intervals $I,J \subseteq \ZZ$ defined by
\eq\label{def:sinv}
\tinv_{I\times J}(g) = \textstyle\frac{1}{2}{}^\sharp
\left\{
	(i,j)\in I \times J
\middle|
	\substack{i<j \\ g(i) > g(j)} ~\tor \substack{i>j \\ g(i) < g(j)}
\right\}.
\endeq
\rmk
Note that $\ur_0(g) \geq \tneg(g)$, and the equality holds when $I = [-d..d]$.
Assuming $g(-i) = -g(i)$, we have
\eq
\ur_0(g) = \tinv_{\{0\}\times I}.
\endeq
Assuming further $g(i+2d) = g(i) + 2d$, we have $g(d+i) = -g(-d-i) = -g(d-i)$ and hence
\eq
\ur_d(g) = \tinv_{\{d\}\times I}.
\endeq
Therefore, the presentations of the Weyl groups of finite and affine classical types in \eqref{def:W} can be characterized alternatively using the symmetrized inversions.
\endrmk

For finite and affine classical types (except $\taffB$ and $\taffD$),
it is known that the $G$-orbits in the set of pairs of complete flags is in bijection with the set $\Sigma_d$ of permutation matrices in $\cO_{d,d}$.
When $X^L_A$ is a (genuine) Schubert variety, namely, $A$ is a permutation matrix, the dimension $\^{\ell}(A)$ coincides with the length $\ell(g)$, where $g$ is the corresponding permutation in the Weyl group.
%----------------------------------------------------------------------------
\begin{thm}\label{thm:A2}
Let $\tinv_{I\times J}$ be the symmetrized inversion function as in \eqref{def:sinv}.
The length functions admit the following symmetrized formulation:
\begin{align}
\label{eq:lA(g)}\ell^\tA(g) &= \tinv_{[1..d]\times[1..d]}(g),
\\
\ell^\tB(g) &= \tinv_{[1..d]\times[-d..d]}(g),
\\
\ell^\tD(g)&= \tinv_{[1..d]\times [-d..d]}(g) - \tinv_{\{0\}\times [-d..d]}(g),
\\
\ell^{\taffA}(g)&=\tinv_{[1..d]\times\ZZ}(g),
\\
\label{eq:laC(g)}\ell^{\taffC}(g)&= \tinv_{[1..d]\times\ZZ}(g),
\\
\ell^{\taffB}(g)&= \tinv_{[1..d]\times\ZZ}(g) - \tinv_{\{0\}\times \ZZ}(g),
\\
\ell^{\taffD}(g)&= \tinv_{[1..d]\times\ZZ}(g) - \tinv_{\{0,d\}\times \ZZ}(g).
\end{align}
\end{thm}
%----------------------------------------------------------------------------
\proof
We start with type $\tX \neq \taffB, \taffD$.
The dimension of a permutation matrix $A \in \Sigma_d$ is a symmetrized sum
\eqs
\^{\ell}^{\tX}(A) = \frac{1}{2}\sum_{(i,j) \in I^{\tX}} a^\bullet_{ij} \sum_{\substack{x\leq i\\ y>j} ~\tor \substack{x\geq i\\ y<j}} a_{xy}
\quad
(\bullet = \natural ~\tor \dagger),
\endeqs
Since $A$ is a permutation matrix, there is only one nonzero entry in each row. Hence,
\eq\label{eq:lX(A)}
\^{\ell}^{\tX}(A) = \frac{1}{2}\sum_{(i,j) \in I^{\tX}} a^\bullet_{ij} \sum_{\substack{x< i\\ y>j} ~\tor \substack{x> i\\ y<j}} a_{xy}.
\endeq
Let $\Sigma^\tX_d$ be the subset of permutation matrices in $\cO^\tX_{d,d}$. We have the  identification below:
\eqs
W^\tX_d \to \Sigma^\tX_d,
\quad
g \mapsto \sum_i E_{g(i),i}.
\endeqs
Under the identification, the right hand side of \eqref{eq:lX(A)} is exactly the symmetrized inversion in \eqref{eq:lA(g)} -- \eqref{eq:laC(g)}.

For types $\taffB$ and $\taffD$, the lengths formulas are closely related to the length formula for affine type C (cf. \cite[(8.65), (8.76)]{BB05}).
Hence, we have
\eqs
\ell^{\taffB} (g) = \ell^{\taffC} (g) - \#\{a\in\ZZ| a\leq d, g(a) \geq d+1\},
\endeqs
and
\eqs
\ell^{\taffD}(g) =
\ell^{\taffB} (g)- \#\{a\in\ZZ| a\leq 0, g(a) \geq 1\}.
\endeqs
The theorem then follows from a direct calculation.
\endproof

\clearpage
%\printindex{Index of Notation}
%\addcontentsline{toc}{chapter}{Index}

\end{document}